\documentclass[11pt,a4paper,reqno]{amsart} 
\usepackage[applemac]{inputenc}
\usepackage{amsmath}
\usepackage{amssymb}
\usepackage{euscript}
\usepackage{mathrsfs}
\usepackage{wasysym}
\usepackage[all]{xy}
\usepackage{color}
\usepackage{rotating}
\usepackage[colorlinks=true,linktocpage=true,pagebackref=false, citecolor=black,linkcolor=black]{hyperref}
\usepackage{tikz}
\usepackage{caption,subcaption}

\newtheorem{mthm}{Theorem}
\newtheorem{mcor}[mthm]{Corollary}
\newtheorem{mlem}[mthm]{Lemma}
\newtheorem{mlemdef}[mthm]{Lemma and Definition}
\newtheorem{thm}{Theorem}[section]
\newtheorem{lem}[thm]{Lemma}
\newtheorem{lemdef}[thm]{Lemma and Definition}

\newtheorem{cor}[thm]{Corollary}

\theoremstyle{definition}

\theoremstyle{remark}
\newtheorem{remark}[thm]{Remark}

\newtheorem*{mremark}{Remark}
\newtheorem{example}[thm]{Example}
\newtheorem*{mexample}{Example}
\newtheorem{examples}[thm]{Examples}

\renewcommand{\theequation}{\thesection.\arabic{equation}}
\numberwithin{equation}{section}

{\em}

{\em}

\newcounter{substep}
\def\thesubstep{\arabic{substep}}

\newenvironment{substeps}[1]{%
\refstepcounter{substep}\noindent{(\ref{#1}.\thesubstep)\ }\ }%
{\em}

\newcommand{\K}{{\mathbb K}} 
 \newcommand{\R}{{\mathbb R}}
 \newcommand{\C}{{\mathbb C}}

\newcommand{\psd}{{\mathcal P}} 
\newcommand{\an}{{\mathcal O}}

 \newcommand{\Cont}{{\mathcal C}}

\newcommand{\gtp}{{\mathfrak p}} \newcommand{\gtq}{{\mathfrak q}}
 
\newcommand{\gta}{{\mathfrak a}} 
\newcommand{\gtP}{{\mathfrak P}}

\newcommand{\Bb}{{\EuScript B}}

\newcommand{\Ss}{{\EuScript S}}
\newcommand{\pol}{{\EuScript K}}

\newcommand{\Vv}{{\EuScript V}}

\newcommand{\im}{\operatorname{im}}
\newcommand{\qf}{\operatorname{qf}}

\newcommand{\Int}{\operatorname{Int}}
\newcommand{\dist}{\operatorname{dist}}

\newcommand{\supp}{\operatorname{support}}

\newcommand{\Spec}{\operatorname{Spec}}

\newcommand{\Sper}{\Spec_r}
\newcommand{\tr}{\operatorname{tr}}

\newcommand{\id}{\operatorname{id}}

\newcommand{\zar}{\operatorname{zar}}
\newcommand{\gr}{\operatorname{graph}}
\newcommand{\cl}{\operatorname{Cl}}

\newcommand{\dgt}{\operatorname{d}}

\newcommand{\lc}{\operatorname{lc}}
\newcommand{\diam}{{\text{\tiny$\displaystyle\diamond$}}}

\newcommand{\x}{{\tt x}} \newcommand{\y}{{\tt y}} 
 \renewcommand{\t}{{\tt t}}

\def\sqbullet{\raise.2ex\hbox{\vrule width 3.5pt height 3.5pt}}

\newcommand{\veps}{\varepsilon}

\newcommand{\ol }{\overline}
\newcommand{\cc}[1]{[{#1}]}
\newcommand{\qq}[1]{\langle{#1}\rangle}

\begin{document}

\title[On the substitution theorem for rings of semialgebraic functions]{On the substitution theorem \\ for rings of semialgebraic functions}

\author{Jos\'e F. Fernando}
\address{Departamento de \'Algebra, Facultad de Ciencias Matem\'aticas, Universidad Complutense de Madrid, 28040 MADRID (SPAIN)}
\curraddr{}
\email{josefer@mat.ucm.es}
\thanks{Author supported by Spanish GR MTM2011-22435.}


\begin{abstract}
Let $R\subset F$ be an extension of real closed fields and ${\mathcal S}(M,R)$ the ring of (continuous) semialgebraic functions on a semialgebraic set $M\subset R^n$. We prove that every $R$-homomorphism $\varphi:{\mathcal S}(M,R)\to F$ is essentially the evaluation homomorphism at a certain point $p\in F^n$ \em adjacent \em to the extended semialgebraic set $M_F$. This type of result is commonly known in Real Algebra as Substitution Theorem. In case $M$ is locally closed, the results are neat while the non locally closed case requires a more subtle approach and some constructions (weak continuous extension theorem, \em appropriate immersion \em of semialgebraic sets) that have interest on their own. We afford the same problem for the ring of bounded (continuous) semialgebraic functions getting results of a different nature. 
\end{abstract}

\maketitle

\section*{Introduction and statements of the main results }\label{s1}

A basic relevant result in Commutative Algebra states that given a ring extension $A\subset B$, every $A$-homomorphism $\varphi:A[\x_1,\ldots,\x_n]\to B$ is completely determined by the images of the variables $\x_1,\ldots,\x_n$, that is, it is a point-evaluation homomorphism. This result is extended straightforwardly to finitely generated algebras over a ring $A$ but it also holds in many other situations. For instance, the following types of homomorphims are point-evaluation homomorphisms: (1) $R$-homomorphisms from a ring of Nash functions on a Nash submanifold of $R^n$ into a real closed extension $F$ of $R$ (see \cite{e}); (2) $\K$-analytic homomorphisms from the ring $\an(\K^n)$ of germs of analytic functions on $\K^n$ into the ring $\an(X)$ of germs of analytic functions on an analytic germ $X$ where $\K=\R$ or $\C$ (see \cite{z}); (3) $\R$-homomorphisms of the ring of smooth functions $\Cont^{\infty}(N,\R)$ on a differentiable manifold $N$ (or an open subset of a Banach space) into $\R$ (see \cite{bl,kms}); (4) $\R$-homomorphisms of the ring of $\Cont^k$ functions on an open subset $U$ of a Banach space into $\R$ under mild conditions (see \cite{gl,jaja}).

In Real Algebra the previous type of results is commonly known as `Substitution Theorem' referring to Efroymson's classical result \cite{e} for the ring of Nash functions on a Nash manifold cited above. In \cite[\S6.1]{fg} we prove that this result also holds for the ring of Nash functions on a semialgebraic set $M$, involving the pro-constructible set obtained by intersecting the Nash closures of $M$ in all open semialgebraic subsets of $R^n$ containing $M$; as one can expect, the result is completely satisfactory if and only if $M$ is a Nash set.

In what follows, let $R\subset F$ be an extension of real closed fields and $M\subset R^n$ a semialgebraic set, that is, a boolean combination of sets defined by polynomial equations and inequalities. A continuous function $f:M\to R$ is \emph{semialgebraic} if its graph is a semialgebraic subset of $R^{n+1}$. We denote the ring of semialgebraic functions on $M$ by ${\mathcal S}(M,R)$ and its subring consisting of those that are bounded by ${\mathcal S}^*(M,R)$. We use the notation ${\mathcal S}^{\diam}(M,R)$ when referring to both of them indistinctly.

In this work we afford the problem of understanding the $R$-homomorphisms $\varphi:{\mathcal S}(M,R)\to F$ in terms of evaluation homomorphisms at points ${\tt p}\in F^n$, which are either in the extension $M_F$ of $M$ to $F$ or `very close to this set' and `fill a big area' (see below the definition of semialgebraic depth of a point of $F^n$). We introduce and recall first several concepts to ease the exposition.

\subsection*{Real closed rings and real closure of a ring} It is well-known that the rings ${\mathcal S}^{\diam}(M,R)$ are particular cases of the so-called \em real closed rings \em introduced by Schwartz in the '80s of the last century, see \cite{s0}. The theory of real closed rings has been deeply developed until now in a fruitful attempt to establish new foundations for semi-algebraic geometry with relevant interconnections to model theory, see the results of Cherlin-Dickmann \cite{cd1,cd2}, Schwartz \cite{s0,s1,s2,s3}, Schwartz with Prestel, Madden and Tressl \cite{ps,sm,scht} and Tressl \cite{t0,t1,t2}. We refer the reader to \cite{s1} for a ring theoretic analysis of the concept of real closed ring. Moreover, this theory, which vastly generalizes the classical techniques concerning the semi-algebraic spaces of Delfs-Knebusch \cite{dk2}, provides a powerful machinery to approach problems concerning certain rings of real valued functions and contributes to achieve a better understanding of the algebraic properties of such rings and the topological properties of their spectra. We highlight some relevant families of real closed rings: (1) real closed fields; (2) rings of real-valued continuous functions on Tychonoff spaces; (3) rings of semi-algebraic functions on semi-algebraic subsets of $R^n$; and more generally (4) rings of definable continuous functions on definable sets in o-minimal expansions of fields.

Every commutative ring $A$ has a so called \em real closure \em ${\rm rcl}(A)$, which is unique up to a unique ring homomorphism over $A$. This means that ${\rm rcl}(A)$ is a real closed ring and there is a (not necessarily injective) ring homomorphism $\gamma:A\to{\rm rcl}(A)$ such that for every ring homomorphism $\psi:A\to B$ to some other real closed ring $B$ there exists a unique ring homomorphism $\ol{\psi}:{\rm rcl}(A)\to B$ with $\psi=\ol{\psi}\circ\gamma$. For example, the real closure of the polynomial ring $R[\x]:=R[\x_1,\ldots,\x_n]$ is the ring ${\mathcal S}(R^n,R)$ of semialgebraic functions on $R^n$. More generally, if $Z\subset R^n$ is an algebraic set, then ${\mathcal S}(Z,R)$ is the real closure of the ring $\psd(Z)$ of polynomial functions on $Z$. In particular, if $\varphi:\psd(Z)\to F$ is an $R$-homomorphism, then there exists a unique $R$-homomorphism $\ol{\varphi}:{\mathcal S}(Z,R)\to F$ such that $\varphi=\ol{\varphi}\circ\gamma$ where $\gamma:\psd(Z)\hookrightarrow{\mathcal S}(Z,R)$ is the natural inclusion. Unfortunately, ${\mathcal S}(M,R)$ is not the real closure of $\psd(M)$ for an arbitrary semialgebraic set $M$ in general. 
\begin{mexample}
Consider the semialgebraic subsets $M:=\{x^2+y^2<1\}$ and $K:=\{x^2+y^2\leq1\}$ of $R^2$. We have that $R[\x,\y]=\psd(R^2)=\psd(K)=\psd(M)$ and so the real closure of this ring is ${\mathcal S}(R^2,R)$ but in the first row of the following diagram
$$
\xymatrix{
{\mathcal S}(R^2,R)\ar@{->>}[r]&{\mathcal S}(K,R)\ar@{^(->}[r]&{\mathcal S}^*(M,R)\ar@{^(->}[r]&{\mathcal S}(M,R)\\
R[\x,\y]\ar@{^(->}[u]\ar@{=}[r]&\psd(K)\ar@{^(->}[u]\ar@{=}[r]&\psd(M)\ar@{^(->}[u]\ar@{^(->}[ru]
}
$$
none of the involved homomorphisms is bijective.\qed 
\end{mexample}
There are many `polynomial' subrings of ${\mathcal S}(M,R)$ corresponding to all semialgebraic embeddings of $M$ in some $R^n$. The choice of a suitable one will be crucial for our purposes.

\subsection*{Extension of coefficients} 
There exists a (unique) semialgebraic subset $M_{F}\subset F^n$ called the \em extension of $M$ to $F$ \em that satisfies $M=M_{F}\cap R^n$. The extension of semialgebraic sets depicts the natural expected behavior with respect to boolean operations, interiors, closures, boundedness, semialgebraically connected components, Transfer Principle, etc. (see \cite[\S5.1-3]{bcr}). Moreover, given another semialgebraic set $N\subset R^n$ and a semialgebraic map $f:M\to N$, there exists a semialgebraic map $f_{F}:M_{F}\to N_{F}$ called \em extension of $f$ to $F$ \em that fulfills $f_{F}|_M=f$. The extension of semialgebraic maps enjoys the natural expected behavior with respect to direct and inverse image, continuity, injectivity, surjectivity, bijectivity, etc. (see \cite[\S5.1-3]{bcr}). Summarizing: `Every property that can be expressed in the first-order language of ordered fields with parameters in $R$ can be transferred to $F$' (\cite[5.2.3]{bcr}). We refer the reader to \cite{dk0} and \cite[\S5]{bcr} for a complete study of the extension to $F$. By \cite[7.3.1]{bcr}, the extension of semialgebraic functions to $F$ induces a well-defined $R$-monomorphism 
$$
{\tt i}_{M,F}:{\mathcal S}(M,R)\hookrightarrow{\mathcal S}(M_{F},F),\ f\mapsto f_{F}. 
$$
This observation together with the evaluation homomorphism ${\rm ev}_{M_F,\tt p}:{\mathcal S}(M_F,F)\to F,\ g\mapsto g({\tt p})$ for ${\tt p}\in M_{F}$ provides the natural $R$-homomorphism $\psi_{\tt p}:={\rm ev}_{M_F,\tt p}\circ {\tt i}_{M,F}:{\mathcal S}(M,R)\to F$.

For each $i=1,\ldots,n$ consider the $i$th-projection $\pi_i:M\to R,\ x:=(x_1,\ldots,x_n)\mapsto x_i$. Given an $R$-homomorphism $\varphi:{\mathcal S}(M,R)\to F$, the point ${\tt p}_{\varphi}:=(\varphi(\pi_1),\ldots,\varphi(\pi_n))\in F^n$ will be called the \em core of $\varphi$\em. Our purpose is to understand under which conditions $\varphi$ is completely determined by its core or, in other words, when $\varphi$ coincides with $\psi_{\tt p_\varphi}$. We still need some extra terminology.

We say that a point ${\tt p}\in F^n$ is \em adjacent to \em $M$ if ${\tt p}\in N_{F}$ for each locally closed semialgebraic set $N\subset R^n$ that contains $M$. The pro-constructible set $\widehat{M_{F}}$ of points of $F^n$ that are adjacent to $M$ will be called the \em adjancence of $M$ in $F^n$\em; clearly, it holds $M_F\subset\widehat{M_{F}}\subset\cl(M)_F$. Of course, if $M$ is locally closed or if $R=F$, then $M_{F}=\widehat{M_{F}}$. We will see in Lemma \ref{adjacent} that the core of an $R$-homomorphism $\varphi:{\mathcal S}(M,R)\to F$ always belongs to $\widehat{M_{F}}$; hence, in case $M$ is locally closed, then ${\tt p}_\varphi\in M_F$. For further properties of the set $\widehat{M_{F}}$ see \cite[I.3.20]{s2}.

We will also need to measure `how much space' a point ${\tt p}\in M_F$ `fills' in $M$ (see Example \ref{space}) and to that end we define the \em semialgebraic depth \em of $\tt p$ as
\renewcommand{\theequation}{I.\arabic{equation}}
\begin{equation}\label{eqdM}
{\rm d}_M({\tt p}):=\min\{\dim(N):\,N\subset M\text{ is a closed semialgebraic subset of $M$ and }\ {\tt p}\in N_F\}.
\end{equation}\renewcommand{\theequation}{\thesection.\arabic{equation}}
As we state in (\ref{sdtd}.\ref{rmd}), if $M$ is closed and bounded in $R^n$, the semialgebraic depth ${\rm d}_M({\tt p})$ has a further algebraic meaning: \em ${\rm d}_M({\tt p})=\tr\deg_R(R({\tt p}))$ where $R({\tt p})$ denotes the smallest subfield of $F$ containing $R$ and the coordinates of the tuple ${\tt p}$\em.

\subsection*{Main results} 
It is clear that the core of an $R$-homomorphism ${\mathcal S}(M,R)\to F$ depends on how $M$ is embedded in $R^n$. Thus, a `good immersion' of $M$ in $R^n$ increases the possibilities that the core completely determines the corresponding $R$-homomorphism. Having this in mind, we introduce the following concept. A semialgebraic set $M\subset R^n$ is \em appropriately embedded \em if it is bounded and for each point $q\in\cl(M)\setminus M$ the germ $M_q$ is semialgebraically connected and either $\cl(M)_q=M_q$ or $\dim(\cl(M)_q\setminus M_q)=\dim(M_q)-1$. We will show in Theorem \ref{se0} that every semialgebraic subset of $R^n$ can be appropriately embedded in $R^n$. We expect that this type of embeddings has further applications than the ones described in this article. After this preliminary exposition, we are ready to state our main results:

\begin{mlem}[Substitution Lemma]\label{st1}
Let ${\tt p}\in F^n$ be adjacent to $M$.
\begin{itemize}

\item[(i)] If ${\tt p}\in M_{F}$, then $\psi_{\tt p}:={\rm ev}_{M_F,{\tt p}}\circ{\tt i}_{M,F}$ is the unique $R$-homomorphism ${\mathcal S}(M,R)\to F$ whose core is ${\tt p}$. In particular, if ${\rm d}_{\cl(M)}({\tt p})=\dim(M_{F,{\tt p}})$, then ${\tt p}\in M_{F}$.

\item[(ii)] If $M$ is appropriately embedded, ${\tt p}\not\in M_{F}$ and ${\rm d}_{\cl(M)}({\tt p})=\dim(M_{F,{\tt p}})-1$, there exists exactly one $R$-homomorphism $\psi_{\tt p}:{\mathcal S}(M,R)\to F$ whose core is ${\tt p}$. Moreover, for each $f\in{\mathcal S}(M,R)$ there exists a semialgebraic set $M\subset M_f\subset\cl(M)$ such that ${\tt p}\in M_{f,F}$, $f$ can be extended to $\widehat{f}\in{\mathcal S}(M_f,R)$ and $\psi_{\tt p}(f)=\widehat{f}_F({\tt p})$.

\item[(iii)] If ${\tt p}\not\in M_{F}$ and ${\rm d}_{\cl(M)}({\tt p})\leq\dim(M_{F,{\tt p}})-2$, there exist infinitely many $R$-homomorphisms ${\mathcal S}(M,R)\to F$ whose core is {\tt p}.
\end{itemize}
\end{mlem}

As $M_{F}=\widehat{M_{F}}$ if $M$ is locally closed, we deduce in this case, by Lemma \ref{st1} (i) and Lemma \ref{adjacent}, that the $R$-homomorphisms ${\mathcal S}(M,R)\to F$ are evaluation homomorphisms. Thus, there exists a bijection between $M_F$ and the set of $R$-homomorphisms ${\mathcal S}(M,R)\to F$. As one can expect, the most interesting case appears when $M$ is not locally closed. 

\begin{mthm}[Substitution Theorem I]\label{st3}
Let $\varphi:{\mathcal S}(M,R)\to F$ be an $R$-homomorphism. Then there exist
\begin{itemize}
\item[(i)] an appropriately embedded semialgebraic set $N\subset R^n$ and a semialgebraic embedding $h:N\hookrightarrow M$ such that $N_0:=h(N)$ is closed in $M$,
\item[(ii)] a point ${\tt p}\in F^n$ adjacent to $N$ such that either ${\tt p}\in N_F$ or ${\tt d}_{\cl(N)}({\tt p})=\dim(N_{F,{\tt p}})-1$,
\end{itemize}
satisfying that the following diagram is commutative
$$
\xymatrix@C=0.5pc@R=0.5pc{
f\ar@{|-{>}}[dd]&{\mathcal S}(M,R)\ar[rrr]^\varphi\ar[dd]_{{\tt j}^*}&&&F\\
&&&\\
f|_{N_0}&{\mathcal S}(N_0,R)\ar[rrr]^{h^{*}}&&&{\mathcal S}(N,R)\ar[uu]_{\psi_{\tt p}}\\
&g\ar@{|-{>}}[rrr]&&&g\circ h\ 
}
$$
where $\psi_{\tt p}:{\mathcal S}(N,R)\to F$ is the unique $R$-homomorphism from ${\mathcal S}(N,R)$ to $F$ whose core is ${\tt p}$.
\end{mthm}

Of course, the decomposition provided in Theorem \ref{st3} is not unique but this results shows that every $R$-homomorphism $\varphi:{\mathcal S}(M,R)\to F$ is essentially a restriction homomorphism composed with a point-evaluation homomorphism (via an intermediate $R$-isomorphism induced by a semialgebraic homeomorphism). Namely, $\varphi$ is determined by two objects (see the proof of Theorem \ref{st3} for further details):
\begin{itemize}
\item[\sqbullet] A fitting closed semialgebraic subset $N_0\subset M$ such that 
$$
I_M(N_0):=\{f\in{\mathcal S}(M,R):\,N_0\subset\{f=0\}\}\subset\ker\varphi.
$$
\item[\sqbullet] A finite family $\{g_1,\ldots,g_n\}\subset{\mathcal S}(N,R)$ such that ${\tt p}:=(\varphi(g_1),\ldots,\varphi(g_n))\in F^n$ is adjacent to $N$ and either ${\tt p}\in N_F$ or ${\tt d}_{\cl(N)}({\tt p})=\dim(N_{F,{\tt p}})-1$.
\end{itemize} 
This confirms the finitary nature of the $R$-homomorphisms $\varphi:{\mathcal S}(M,R)\to F$. 

\vspace{2mm}
\noindent{\em Ring of bounded semialgebraic functions}. The situation for the ring of bounded semialgebraic functions is quite different. Recall that a \em semialgebraic pseudo-compactification of $M$ \em is a pair $(X,{\tt j})$ that consists of a closed and bounded semialgebraic set $X\subset R^n$ and a semialgebraic embedding ${\tt j}:M\hookrightarrow X$ whose image is dense in $X$. A crucial fact is that ${\mathcal S}^*(M,R)$ is the direct limit of the family constituted by the rings of semialgebraic functions ${\mathcal S}(X,R)$ where $(X,{\tt j})$ runs on the semialgebraic pseudo-compactifications of $M$ (see \ref{cita}). We will use this together with Lemma \ref{st1}(i) applied to ${\mathcal S}(X,R)$ in order to understand the $R$-homomorphisms ${\mathcal S}^*(M,R)\to F$. A substantial difference using $R$-homomorphisms $\varphi:{\mathcal S}(M,R)\to F$ is that if $M$ is bounded, the core of an $R$-homomorphism ${\mathcal S}^*(M,R)\to F$ belongs to $\cl(M)_F$ with no further restrictions (compare Lemma \ref{adjacent} with Remark \ref{adjacentbounded}). In fact, if $(X,{\tt j})$ is a semialgebraic pseudo-compactification of $M$ and $\varphi:{\mathcal S}(X,R)\to F$ is an $R$-homomorphism, then it can be extended to the direct limit ${\mathcal S}^*(M,R)$ (of course this extension is rarely unique, see Lemma \ref{exsx} and Remark \ref{exsx2}). A crucial fact proved in \cite[Proof of Thm. 3. Step 2]{fg2} states: 
\begin{mlemdef}\label{bri}
Let $\gtp$ be a prime ideal of ${\mathcal S}^{\diam}(M,R)$. Then there exists a semialgebraic pseudo-compactification $(X,{\tt j})$ of $M$ such that 
$$
\qf({\mathcal S}(X,R)/(\gtp\cap{\mathcal S}(X,R)))=\qf({\mathcal S}^{\diam}(M,R)/\gtp).
$$
\em We say that $(X,{\tt j})$ is a \em brimming semialgebraic pseudo-compactification of $M$ for $\gtp$\em. 
\end{mlemdef}

Our main result for the bounded case, which also works for the ring ${\mathcal S}(M,R)$, is the following.

\begin{mthm}[Substitution Theorem II]\label{boundedcase}
Let $\varphi:{\mathcal S}^{\diam}(M,R)\to F$ be an $R$-homomorphism and $(X,{\tt j})$ a brimming semialgebraic pseudo-compactification of $M$ for $\gtp:=\ker\varphi$. Let ${\tt p}$ be the core of $\psi:=\varphi\circ{\tt j}^*$ and denote $A:={\mathcal S}(X,R)/(\gtp\cap{\mathcal S}(X,R))$. Then $\psi$ induces the homomorphism 
$$
\widehat{\psi}:\qf(A)=\qf({\mathcal S}^{\diam}(M,R)/\gtp)\hookrightarrow F,\ \tfrac{[a]}{[b]}\mapsto\tfrac{\psi(a)}{\psi(b)}=\tfrac{a_{1,F}({\tt p})}{a_{2,F}({\tt p})}
$$ 
and the following diagram is commutative
$$
\xymatrix{
\qf(A)\ar@{=}[r]&\qf({\mathcal S}^{\diam}(M,R)/\gtp)\ar@{^(->}[r]^(.7){\widehat{\psi}}&F\\
A\ar@{^(->}[u]\ar@{^(->}[r]^(.3){\ol{\tt j}^*}&{\mathcal S}^{\diam}(M,R)/\gtp\ar@{^(->}[u]\ar@{^(->}[ru]^{\ol{\varphi}}\\
{\mathcal S}(X,R)\ar@{^(->}[r]_{{\tt j}^*}\ar[u]\ar@/_{9mm}/[rruu]^(.375)\psi&{\mathcal S}^{\diam}(M,R)\ar[u]!D|\hole\ar@/_{5mm}/[uur]_{\varphi}}
$$
In particular, $\varphi$ is completely determined by ${\tt p}$: if $f\in{\mathcal S}^{\diam}(M,R)$, there exist $a_1,a_2\in{\mathcal S}(X,R)$ such that $a_{2,F}({\tt p})\neq0$, $a_2f-a_1\in\gtp$ and $\varphi(f)=\frac{a_{1,F}({\tt p})}{a_{2,F}({\tt p})}$.
\end{mthm}
\begin{proof}
By Lemma \ref{st1}, it holds $\psi={\rm ev}_{X_F,\tt p}\circ{\tt i}_{X,F}$. Note that $\widehat{\psi}:\qf(A)\to F$ is the unique homomorphism between the real closed fields $\qf(A)$ and $F$. On the other hand, $\ol{\varphi}:{\mathcal S}^{\diam}(M,R)\hookrightarrow F$ induces the homomorphism 
$$
\widehat{\ol{\varphi}}:\qf({\mathcal S}^{\diam}(M,R)/\gtp)\hookrightarrow F, \tfrac{[f]}{[g]}\mapsto\tfrac{\varphi(f)}{\varphi(g)}.
$$ 
Since $\qf(A)=\qf({\mathcal S}^{\diam}(M,R)/\gtp)$, we have $\widehat{\ol{\varphi}}=\widehat{\psi}$ and for each $f\in{\mathcal S}^{\diam}(M,R)$ there exist $a_1,a_2\in{\mathcal S}(X,R)$ such that $a_2\not\in\gtp\cap{\mathcal S}(X,R)$ and $a_2f-a_1\in\gtp$; in particular, $\varphi(a_i)=\psi(a_i)=a_{i,F}({\tt p})$ and $a_{2,F}({\tt p})=\varphi(a_2)\neq0$. Thus, 
$$
a_{2,F}({\tt p})\varphi(f)-a_{1,F}({\tt p})=\varphi(b)\varphi(f)-\varphi(a)=\varphi(bf-a)=0,
$$
that is, $\varphi(f)=\tfrac{a_{1,F}({\tt p})}{a_{2,F}({\tt p})}$, as required.
\end{proof}

\begin{mremark}
If $\varphi:{\mathcal S}^{\diam}(M)\to A$ is an $R$-homomorphism into an $R$-algebra $A$, the first isomorphism theorem provide the following canonical decomposition of $\varphi$:
$$
{\mathcal S}^{\diam}(M)\overset{\pi}{\longrightarrow}{\mathcal S}^{\diam}(M)/\ker\varphi\overset{\ol{\varphi}}{\cong}\im\varphi\hookrightarrow A.
$$
In particular, if $A$ is an integral domain, the ideal $\gtp:=\ker\varphi$ is prime and so $F:=\qf({\mathcal S}^{\diam}(M)/\gtp)$ is a real closed field that contains $R$ as a subfield. Now, the results stated above can be applied to the $R$-homomorphism $\psi:{\mathcal S}^{\diam}(M)\overset{\pi}{\longrightarrow}{\mathcal S}^{\diam}(M)/\gtp\hookrightarrow F$ induced by $\varphi$.
\end{mremark}

Theorem \ref{st3} is proved using Lemma \ref{st1}, Theorem \ref{boundedcase} and the theory of real closed rings \cite{s2}. The proof of Lemma \ref{st1} is based on a fruitful use of the real closure's universal properties of the ring of polynomials with coefficients in $R$ and the crucial fact that semialgebraic functions enjoy the following weak continuous extension property, which has its own interest.

\begin{mthm}[Weak continuous extension property]\label{ext2}
Suppose that the germ $M_q$ is semialgebraically connected for all $q\in\cl(M)$. Then for each $f\in{\mathcal S}(M,R)$ there exist an open semialgebraic neighborhood $V$ of $M$ in $\cl(M)$ and a semialgebraic set $Y\subset\cl(M)\setminus M$ such that $\dim(Y_q)\leq\dim(M_q)-2$ for all $q\in Y$ and $f$ can be extended continuously to $V\setminus Y$.
\end{mthm}

Of course, if $M$ has dimension $2$, we may take $Y=\varnothing$ after shrinking $V$. However, this possibility
can no longer be extended for dimension $\geq3$ (see Example \ref{ce}). Namely,
\begin{mcor}[Continuous extension property]
Assume that $M$ is $2$-dimensional and the germ $M_q$ is semialgebraically connected for all $q\in\cl(M)$. Then for each semialgebraic function $f\in{\mathcal S}(M,R)$ there exists an open semialgebraic neighborhood $V$ of $M$ in $\cl(M)$ such that $f$ can be extended continuously to $V$.
\end{mcor}

The article is organized as follows. In Section \ref{s2} we present all basic notions and notations used in this paper as well as some preliminary results. We give special attention to the semialgebraic pseudo-compactifications of a semialgebraic set that play a special role in this work. The reading can be started directly in Section \ref{s3} and referred to the Preliminaries only when needed. The aim of Section \ref{s3} is to prove Theorem \ref{ext2} (in \S\ref{4}) and the purpose of Section \ref{s4} is to show that each semialgebraic set can be appropriately embedded. In Section \ref{s5} we prove Lemma \ref{st1} (in \S\ref{1}) and Theorem \ref{st3} (in \S\ref{2}). 

The author is indebted with Prof. Schwartz and Prof. Tressl for encouraging him to use the theory of real closed rings and develop this work in the framework of general real closed fields and not only for the field $\R$ of real numbers, which has been the first approach. This more general procedure allows the author to present clearer and stronger results. The author is also very grateful to S. Schramm for a careful reading of the final version and for the suggestions to refine its redaction.

\section{Preliminaries on semialgebraic sets}\label{s2}

In this section we introduce some terminology, notations and preliminary results that are systematically used in this work. For each $f\in{\mathcal S}^{\diam}(M,R)$ and each semialgebraic subset $N\subset M$ we denote $Z_N(f):=\{x\in N:\, f(x)=0\}$. If $N=M$, we say that $Z(f):=Z_M(f)$ is the \em zeroset \em of $f$. We denote the open ball of $R^n$ with center $x$ and radius $\veps$ with $\Bb_n(x,\veps)$ and the corresponding closed ball with $\ol{\Bb}_n(x,\veps)$. In some cases it will be useful to assume that the semialgebraic set $M$ we are working with is bounded. Such assumption can be done without loss of generality. Namely, the semialgebraic homeomorphism
$$
h:\Bb_n(0,1)\to R^n,\ x\mapsto\frac{x}{\sqrt{1-\|x\|^2}}
$$
induces an $R$-isomorphism ${\mathcal S}(M,R)\to {\mathcal S}(h^{-1}(M),R),\,f\mapsto f\circ h$.

A crucial fact when dealing with the ring of semialgebraic functions on a semialgebraic set $M$ is that every closed semialgebraic subset $Z$ of $M$ is the zeroset $Z(h)$ of a (bounded) semialgebraic function $h$ on $M$; take for instance $h:=\min\{1,\dist(\cdot,Z)\}\in {\mathcal S}^*(M,R)$.

Local closedness has been revealed in the semialgebraic setting as an important property for the validity of results which are in the core of semialgebraic geometry. If $M\subset R^n$ is a semialgebraic set, then $\cl(M)$ and $U:=R^n\setminus(\cl(M)\setminus M)$ are also semialgebraic sets. In case $M$ is additionally locally closed, then $U$ is open in $R^n$ and so $M=\cl(M)\cap U$ can be written as the intersection of a closed and an open semialgebraic subset of $R^n$. Consider the subset
\begin{equation}
\rho(M):=\cl(\cl(M)\setminus M)\cap M.
\end{equation}
It follows from \cite[9.14-9.21]{dk2} that the semialgebraic set 
\begin{equation}\label{mlc}
M_{\lc}:=M\setminus\rho(M)=\cl(M)\setminus\cl(\cl(M)\setminus M)
\end{equation}
is the largest locally closed and dense subset of $M$ and it coincides with the set of points of $M$, which have a closed and bounded semialgebraic neighborhood in $M$.

An elementary but important fact states that the core of an $R$-homomorphism ${\mathcal S}(M,R)\to F$ is adjacent to $M$. Its proof is inspired partly by the proof of Efroymson's Substitution Theorem for Nash functions \cite[8.5.2]{bcr} (see also \cite[6.3]{fg} for a generalization).

\begin{lem}\label{adjacent}
Let $\varphi:{\mathcal S}(M,R)\to F$ be an $R$-homomorphism and ${\tt p}$ the core of $\varphi$. Then ${\tt p}$ is adjacent to $M$.
\end{lem}
\begin{proof}
Write $\cl(M)=\bigcup_{i=1}^r\{g_{i1}\geq0,\ldots,g_{is}\geq0\}$ for some $g_{ij}\in R[\x]$. Suppose first ${\tt p}\not\in\cl(M)_{F}$; we may assume $g_{i1,F}({\tt p})<0$ for $i=1,\ldots,r$. Consider the semialgebraic function $g:=\prod_{i=1}^r(g_{i1}-|g_{i1}|)^2$ and observe that $g|_M=0$; hence, $\varphi(g)=0$. However, since $\varphi$ is an $R$-homomorphism and $|g_{ij}|=\sqrt{g_{ij}^2}$, we have $\varphi(g_{ij})=g_{ij,F}({\tt p})$ and $\varphi(|g_{ij}|)=|g_{ij,F}({\tt p})|$. Thus, 
$$
0=\varphi(g)=\prod_{i=1}^r(g_{i1,F}({\tt p})-|g_{i1,F}({\tt p})|)^2>0
$$
because $g_{i1,F}({\tt p})<0$ for $i=1,\ldots,r$, which is a contradiction. We conclude ${\tt p}\in\cl(M)_{F}$. 

Now let $N$ be a locally closed semialgebraic set that contains $M$. Then $N=\cl(N)\cap U$ where $U$ is the open semialgebraic set $R^n\setminus(\cl(N)\setminus N)$; observe that $M\subset U$. As $\cl(M)\subset\cl(N)$ it is enough to check that ${\tt p}\in U_{F}$. 

Indeed, write $R^n\setminus U:=\bigcup_{i=1}^k\{h_{i1}\geq0,\ldots,h_{i\ell}\geq0\}$ and consider the semialgebraic function $h:=\prod_{i=1}^k\sum_{j=1}^\ell(h_{ij}-|h_{ij}|)^2$. Since $Z(h)\cap M=(R^n\setminus U)\cap M=\varnothing$, we have that $h$ is a unit in ${\mathcal S}(M,R)$ and so $0<\varphi(h)=h_F({\tt p})$; hence, ${\tt p}\in F^n\setminus Z(h_{F})=U_{F}$, as required.
\end{proof}

A crucial tool when dealing with rings of bounded semialgebraic functions is the use of semialgebraic pseudo-compactifications.

\subsection{Semialgebraic pseudo-compactifications of a semialgebraic set.}\label{cita}

Let $(X,{\tt j})$ be a semialgebraic pseudo-compactification of $M$. It holds that ${\mathcal S}(X,R)={\mathcal S}^*(X,R)$ since the image of a bounded and closed semialgebraic set under a semialgebraic function is again bounded and closed. The embedding ${\tt j}$ induces an $R$-monomorphism ${\tt j}^*:{\mathcal S}(X,R)\hookrightarrow{\mathcal S}^{\diam}(M,R),\ f\mapsto f\circ{\tt j}$ and we denote $\gta\cap{\mathcal S}(X,R):=({\tt j}^{*})^{-1}(\gta)$ for every ideal $\gta$ of ${\mathcal S}^{\diam}(M,R)$. The following properties showed in \cite[\S1]{fg2} are decisive:

\vspace{2mm}
\begin{substeps}{cita}\label{a}
\em
For each finite family $\{f_1,\ldots,f_r\}\subset{\mathcal S}^*(M,R)$ there exists a semialgebraic pseudo-compacti\-fication $(X,{\tt j})$ of $M$ and semialgebraic functions $F_1,\ldots,F_r\in{\mathcal S}(X,R)$ such that $f_i=F_i\circ{\tt j}$. 
\em
\end{substeps}

\vspace{2mm}
\begin{substeps}{cita}
Let ${\mathfrak F}_M$ be the collection of all semialgebraic pseudo-compactifications of $M$. Given $(X_1,{\tt j}_1),(X_2,{\rm j}_2)\in{\mathfrak F}_M$, we say that $(X_1,{\tt j}_1)\preccurlyeq(X_2,{\rm j}_2)$ if and only if there exists a (unique) continuous surjective map $\rho:X_2\to X_1$ such that $\rho\circ{\tt j}_2={\tt j}_1$; the uniqueness of $\rho$ follows because $\rho|_M={\tt j}_1\circ({\tt j}_2|_M)^{-1}$ and $M$ is dense in $X_i$. It holds that: \em $({\mathfrak F}_M,\preccurlyeq)$ is a directed set.\em
\end{substeps}

\vspace{2mm}
\begin{substeps}{cita}
We have a collection of rings $\{{\mathcal S}(X,R)\}_{(X,{\tt j})\in{\mathfrak F}_M}$ and $R$-monomorphisms 
$$
\rho_{X_1,X_2}^*:{\mathcal S}(X_1,R)\to{\mathcal S}(X_2,R),\ f\mapsto f\circ\rho
$$ 
for $(X_1,{\tt j}_1)\preccurlyeq(X_2,{\tt j}_2)$ such that 
\begin{itemize}
\item $\rho_{X_1,X_1}^*=\id$ and
\item $\rho_{X_1,X_3}^*=\rho_{X_2,X_3}^*\circ\rho_{X_1,X_2}^*$ if $(X_1,{\tt j}_1)\preccurlyeq(X_2,{\tt j}_2)\preccurlyeq(X_3,{\tt j}_3)$.
\end{itemize}
We conclude: \em The ring ${\mathcal S}^*(M,R)$ is the direct limit of the directed system $\qq{{\mathcal S}(X,R),\rho_{X_1,X_2}^*}$ together with the homomorphisms ${\tt j}^*:{\mathcal S}(X,R)\hookrightarrow{\mathcal S}^*(M,R)$ where $(X,{\tt j})\in{\mathfrak F}_M$\em. We write ${\mathcal S}^*(M,R)=\displaystyle\lim_{\longrightarrow}{\mathcal S}(X,R)$. 
\end{substeps}

\vspace{1mm}
\begin{substeps}{cita}\label{inequality}
On the other hand, \em the ring ${\mathcal S}(M,R)$ is the localization ${\mathcal S}^*(M,R)_{{\mathcal W}}$ of ${\mathcal S}^*(M,R)$ at the multiplicative set ${\mathcal W}$ of those functions $f\in{\mathcal S}^*(M,R)$ such that $Z_M(f)=\varnothing$\em. In particular, \em if $\gtp$ is a prime ideal of ${\mathcal S}^*(M,R)$ that does not meet ${\mathcal W}$, then $\qf({\mathcal S}^*(M,R)/\gtp)=\qf({\mathcal S}(M,R)/\gtp{\mathcal S}(M,R))$\em.
\end{substeps}

\begin{remark}\label{adjacentbounded}
Proceeding similarly to the first part of the proof of Lemma \ref{adjacent}, one may prove that if $M$ is bounded and $\varphi:{\mathcal S}^*(M,R)\to F$ is an $R$-homomorphism, then the core ${\tt p}:=(\varphi(\pi_1),\ldots,\varphi(\pi_n))$ of $\varphi$ belongs to $\cl(M)_F$. We need the boundness of $M$ to guarantee that the polynomial functions on $M$ are bounded functions, that is, $\psd(M)\hookrightarrow{\mathcal S}^*(M,R)$. 

On the other hand, we cannot adapt the second part of the proof of Lemma \ref{adjacent} to the ring ${\mathcal S}^*(M,R)$, that is, we cannot assure that ${\tt p}$ is adjacent to $M$, because a function $f\in{\mathcal S}^*(M,R)$ with empty zero set needs not to be a unit. Thus, we can only ensure that ${\tt p}\in\cl(M)_F$. 
\end{remark}

As a kind of converse of the previous remark, we propose the following result.

\begin{lem}\label{exsx}
Assume that $M$ is bounded and let $X:=\cl(M)$ and ${\tt p}\in X_F$. Then there exists an $R$-homomorphism $\varphi:{\mathcal S}^*(M,R)\to F$ whose core is ${\tt p}$.
\end{lem}
\begin{proof}
By the curve selection lemma \cite[2.5.5]{bcr} there exists a semialgebraic path $\alpha:[0,1]_F\to F^n$ such that $\alpha(0)={\tt p}$ and $\alpha((0,1]_F)\subset M_F$. Consider the map
$$
\varphi:{\mathcal S}^*(M,R)\to F,\ f\mapsto\lim_{t\to 0^+}(f_F\circ\alpha)(t).
$$
Once we guarantee that $\lim_{t\to 0^+}(f_F\circ\alpha)(t)$ exists, it is clear that $\varphi$ is an $R$-homomorphism whose core is ${\tt p}$.

Indeed, as the graph of $g:=f_F\circ\alpha$ is a $1$-dimensional semialgebraic subset of $F^2$, it is a finite union of singletons and 1-dimensional Nash manifolds (see \cite[2.9.10]{bcr}); hence, we may assume after shrinking the domain of $g$ that the restriction of $g$ to $(0,1)_F$ is Nash. If $g$ is constant on $(0,1)_F$, the existence of the limit is clear. Otherwise the zero set of $g'$ is finite and after shrinking the domain of $g$, we may assume that it is empty and without loss of generality that $g$ is decreasing on $(0,1)_F$. As $g$ is a bounded function, $g((0,1]_F)=(\lambda,g(1)]$ for some $\lambda\in F$ (use \cite[2.1.7]{bcr}) and so $\lim_{t\to 0^+}g(t)=\lambda$.
\end{proof}
\begin{remark}\label{exsx2}
Notice that each pair of essentially different semialgebraic paths $\alpha_1,\alpha_2$ in $M_F$ through the point ${\tt p}$ define different $R$-homomorphisms $\varphi_i:{\mathcal S}^*(M,R)\to F$ whose core is ${\tt p}$.
\end{remark}

\subsection{Semialgebraic depth and transcendence degree}\label{sdtd}
We recall a different point of view for the concept of semialgebraic depth devised in \cite[\S2]{fg2} (see also \cite{fe} for a deeper study of this invariant). Given a prime ideal $\gtp$ of ${\mathcal S}(M,R)$, we define the \em semialgebraic depth of $\gtp$ \em as
$$
{\rm d}_M(\gtp):=\min\{\dim(Z(f)):\, f\in\gtp\}.
$$
Now given a point ${\tt p}:=({\tt p}_1,\ldots,{\tt p}_n)\in M_F$, consider the prime ideal $\gtp({\tt p}):=\{f\in{\mathcal S}(M,R):\,f_F({\tt p})=0\}$. As one can check, it holds ${\tt d}_M({\tt p})={\tt d}_M(\gtp({\tt p}))$ (see \eqref{eqdM}). On the other hand, let $F$ be the real closed field $\qf({\mathcal S}(M,R)/\gtp)$ and ${\tt p}\in F^n$ the core of the $R$-homomorphism
$$
\varphi:{\mathcal S}(M,R)\to{\mathcal S}(M,R)/\gtp\hookrightarrow F.
$$ 
Then $\gtp=\gtp({\tt p})$ and so ${\tt d}_M({\tt p})={\tt d}_M(\gtp)$. 

Recall that an ideal $\gta$ of ${\mathcal S}(M,R)$ is a \em $z$-ideal \em if a function $f\in{\mathcal S}(M,R)$ belongs to $\gta$ whenever there exists $g\in\gta$ such that $Z(g)\subset Z(f)$. If $M$ is locally closed, all radical ideals of ${\mathcal S}(M,R)$ are $z$-ideals \cite[2.6.6]{bcr}. The following properties are proved in \cite[Thm.3]{fg2}:

\vspace{2mm}\setcounter{substep}{0}
\begin{substeps}{sdtd}\label{rmdc}
\em Let $\gtp$ be a prime $z$-ideal of ${\mathcal S}(M,R)$. Then ${\rm d}_M(\gtp)=\tr\deg_R({\mathcal S}(M,R)/\gtp)$.
\end{substeps}

\vspace{2mm}
\begin{substeps}{sdtd}\label{rmd}
\em Let $X$ be a closed and bounded semialgebraic set and $\gtp$ a prime ideal of ${\mathcal S}(X,R)$. Then
$$
{\rm d}_X(\gtp)=\tr\deg_R(R[\x]/(\gtp\cap R[\x]))=\tr\deg_R({\mathcal S}(X,R)/\gtp).
$$
In particular, if ${\tt p}\in X_F$ and $R({\tt p})$ is the smallest subfield of $F$ that contains $R$ and the coordinates of the tuple ${\tt p}$,
$$
{\rm d}_X({\tt p})=\tr\deg_R(R({\tt p}))=\tr\deg_R(\qf({\mathcal S}(X,R)/\gtp({\tt p}))).
$$

\em The following example illustrates the algebraic interpretation of the semialgebraic depth.
\end{substeps}

\begin{example}\label{space}
Let $F:=\R(\{\t^*\})$ be the field of meromorphic Puiseux series with coefficients in $\R$. Consider the point ${\tt p}:=(\t,e^\t)\in F^2$. Clearly, $\tr\deg_{\R}(\R(\t,e^\t))=2$ and let us check that ${\rm d}_{\R^2}({\tt p})=2$. Indeed, we have to prove that if $N\subset\R^2$ is a closed semialgebraic subset such that ${\tt p}\in N_F$, then it has dimension $2$. Otherwise, choose a $1$-dimensional closed semialgebraic set $N$ such that ${\tt p}\in N_F$. By \cite[2.9.10]{bcr}, we may assume that $N$ is the union of two points and a Nash manifold Nash diffeomorphic to the interval $(0,1)$. Thus, there exists a non zero polynomial $P\in\R[\t,\x]$ such that $N\subset Z(P)$; hence, $P|_{N_F}=0$ and so $P(\t,e^\t)=0$, which is a contradiction since $\t$ and $e^\t$ are algebraically independent over $\R$. 
\end{example}

\section{Weak continuous extension of a semialgebraic function}\label{s3}

The purpose of this section is to prove Theorem \ref{ext2}. We begin with some illustrating examples.

\begin{examples}
(i) Let $M:=(R^2\setminus\{y=0\})\cup\{(0,0)\}$ and $f:M\setminus\{(0,0)\}\to R$ the semialgebraic function that maps $\{y>0\}$ onto $1$ and $\{y<0\}$ onto $0$. Since $f$ is bounded on $M\setminus\{(0,0)\}$, the semialgebraic map $g:=(x^2+y^2)f$ can be extended continuously to the origin. However, $g$ cannot be extended continuously to any neighborhood of the origin in $ R^2=\cl(M)$.

(ii) Let $M':=\{(y-x)(y+x)>0\}\cup\{(0,0)\}$ and consider the bounded semialgebraic function
$$
f:R^2\setminus\{xy=0\}\to R,\ (x,y)\mapsto\frac{xy}{|xy|}.
$$
Clearly, $h:=xf\in{\mathcal S}(M,R)$ can be extended continuously to $\{x=0\}$. Consider the semialgebraic maps
$$
\varphi:M\to M',\ (x,y)\mapsto(x,y+h(x,y))\quad\text{and}\quad\psi:M'\to M,\ (x,y)\mapsto(x,y-h(x,y)),
$$
which are mutually inverse. Thus, $M$ and $M'$ are semialgebraically homeomorphic and so ${\mathcal S}(M,R)$ and ${\mathcal S}(M',R)$ are isomorphic. It follows from the following result that each $g\in{\mathcal S}(M',R)$ can be extended continuously to a neighborhood of the origin in $\cl(M)=\{(y-x)(y+x)\geq0\}$.
\end{examples}

\subsection*{Triangulation of semialgebraic sets} 
We will use often the following fact: \em Let $X\subset R^n$ be a bounded and closed semialgebraic set and $\{S_1,\ldots,S_r\}$ a family of semialgebraic subsets of $X$. Then there exists a finite simplicial complex $K$ and a semialgebraic homeomorphism $\Phi:|K|\to X$ such that the restriction $\Phi|_{\sigma^0}:\sigma^0\to R^n$ is a Nash embedding for each $\sigma\in K$ and each set $S_i$ is the union of finitely many open simplices $\sigma^0$ of $K$ \em (see \cite[9.2.1, 9.2.3]{bcr} for further details).

\subsection{Proof of Theorem \ref{ext2}}\label{4}

Consider the bounded semialgebraic function $g:=\frac{f}{1+|f|}$, which satisfies $|g(x)|<1$ for all $x\in M$. Note that since $f=\theta(g)$ where $\theta:(-1,1)\to R,\ t\mapsto \frac{t}{1-|t|}$, it is enough to prove the statement for $g$. 

\noindent{\em Strategy of the Proof:} First we construct the open neighborhood $V$ of $M$ in $\cl(M)$ quoted in the statement. Next we construct the set $Y$ in several steps by describing the different types of problematic points to extend $g$ continuously and proving that each of these sets $Y_i$ has local dimension upperly bounded by the local dimension of $M$ minus $2$. We also prove for technical reasons that $V\setminus(Y\cup M)$ is locally closed. Finally, we construct the continuous (semialgebraic) extension of $g$ to $V\setminus Y$.

\noindent{\bf Step 1.} \em General notations and construction of the open semialgebraic neighborhood $V$\em. We assume that $M$ is bounded. Notice that $M':=\gr(g)\subset M\times R$ is a bounded semialgebraic set and let $X:=\cl(M')$. Consider the projections 
$$
\pi:R^{n+1}\to R^n,\ (x,x_{n+1})\mapsto x\quad\text{and}\quad\pi_{n+1}:R^{n+1}\to R,\ (x_1,\ldots,x_n,x_{n+1})\mapsto x_{n+1}.
$$ 
Define $S:=\{x\in X:\,\pi_{n+1}(x)=\pm1\}$ and $\varrho:=\pi|_X$; observe that $M'\cap S=\varnothing$. 

\vspace{2mm}
\noindent We claim: \em The fiber $\varrho^{-1}(p)$ is a singleton for all $p\in M$\em.

Indeed, suppose by contradiction that there exists a point $(q,\lambda)\in X$ such that $q\in M$ and $g(q)\neq\lambda$. By the curve selection lemma \cite[2.5.5]{bcr} there exists a continuous semialgebraic path $\alpha:[0,1]\to R^{n+1}$ such that $\alpha(0)=(q,\lambda)$ and $\alpha((0,1])\subset M'$. Define $\beta:=\pi\circ\alpha:[0,1]\to M$ and observe that $\alpha|_{(0,1]}=(\beta|_{(0,1]},g\circ\beta|_{(0,1]})$, $\beta(0)=q$ and $\alpha(0)=(q,\lambda)$ where $\lambda\neq g(q)$, which is against the continuity of $g$ at $q$. 

Now since $S$ is a bounded and closed semialgebraic set, so is $C:=\varrho(S)$. Define $V:=\cl(M)\setminus C$, which is an open semialgebraic subset of $\cl(M)$ and observe that $M\subset V$ because $\varrho^{-1}(p)$ is a singleton and $|g(p)|<1$ for each $p\in M$.

\noindent{\bf Step 2.} \em Construction of the semialgebraic set $Y$\em. 

\noindent{\bf S2.a.} Define $Y_0:=\{p\in\cl(M):\, \dim(\varrho^{-1}(p))=1\}$. We claim: \em$\dim(Y_{0,p})\leq\dim(M_p)-2$ for all $p\in\cl(M)$ and the fiber $\varrho^{-1}(q)$ is a singleton for each $q\in\cl(M)\setminus Y_0$\em. 

Indeed, assume by contradiction that there exists $p\in\cl(M)$ such that $\dim(Y_{0,p})\geq\dim(M_p)-1$ and let $U$ be an open semialgebraic neighborhood of $p$ such that $\dim(Y_0\cap U)\geq\dim(M\cap U)-1$. Notice that since $\varrho^{-1}(x)$ has dimension $1$ for each $x\in Y_0\cap U$,
$$
\dim(\varrho^{-1}(Y_0\cap U))=\dim(Y_0\cap U)+1\geq\dim(M\cap U). 
$$
On the other hand, since $M\cap Y_0=\varnothing$ and $\varrho^{-1}(M)=M'$, 
$$
\varrho^{-1}(Y_0\cap U)\subset (X\setminus M')\cap\varrho^{-1}(U)\subset(\cl(M'\cap\varrho^{-1}(U))\setminus(M'\cap\varrho^{-1}(U)));
$$ 
hence, by \cite[2.8.13]{bcr},
$$
\dim(M\cap U)\leq\dim(\varrho^{-1}(Y_0\cap U))<\dim(M'\cap\varrho^{-1}(U))=\dim(M\cap U),
$$
which is a contradiction. Thus, $\dim(Y_{0,p})\leq\dim(M_p)-2$ for all $p\in\cl(M)$.

Fix a point $q\in\cl(M)\setminus Y_0$ and let us check that $\varrho^{-1}(q)$ is a singleton. Since $q\not\in Y_0$, we know that $\varrho^{-1}(q)=\{z_1,\ldots,z_s\}$ is a finite set. Choose pairwise disjoint bounded and closed semialgebraic sets $K_i\subset R^{n+1}$ such that $z_i\in\Int(K_i)$. Let $K:=X\setminus\bigcup_{i=1}^s\Int(K_i)$, which is a closed and bounded semialgebraic set and so it is also $\pi(K)$. As $\varrho^{-1}(q)=\{z_1,\ldots,z_s\}$, we deduce $q\not\in\pi(K)$. Define $W:=R^n\setminus\pi(K)$, which is an open semialgebraic subset of $R^n$ that contains $q$ and satisfies $\varrho^{-1}(W\cap\cl(M))\subset\bigcup_{i=1}^s X\cap K_i$; hence, 
\begin{equation}\label{union}
W\cap M\subset\bigcup_{i=1}^s\pi(X\cap K_i)\cap M.
\end{equation} 
We know that $\varrho^{-1}(y)$ is a singleton for each $y\in M$; hence, as $K_i\cap K_j=\varnothing$ if $i\neq j$, we deduce
$$
\pi(X\cap K_i)\cap\pi(X\cap K_j)\subset\cl(M)\setminus M\quad\text{ if $i\neq j$.} 
$$
Moreover, each set $\pi(X\cap K_i)\cap M$ is non empty and closed in $M$. Thus, since $M_q$ is semialgebraically connected, we deduce $s=1$ and so $\varrho^{-1}(q)$ is a singleton.

\noindent{\bf S2.b.} Let $(K,\Phi)$ be a semialgebraic triangulation of $X$ compatible with $M'$, $X\setminus M'$, $\varrho^{-1}(V\setminus M)$ and $\varrho^{-1}(Y_0)$ such that $\Phi|_{\sigma^0}:\sigma^0\to S^0:=\Phi(\sigma^0)\subset R^{n+1}$ is a Nash embedding for each $\sigma\in K$. Let ${\mathfrak F}$ be the collection of all simplices $\tau\in K$ such that: 
\begin{itemize}
\item $T^0:=\Phi(\tau^0)\subset\varrho^{-1}(V\setminus(M\cup Y_0))$ and 
\item there exists a point $x\in T^0$ satisfying $\dim(T^0_x)=\dim(M'_x)-1$. 
\end{itemize}
We claim: \em If $\tau\in{\mathfrak F}$, then $\dim(T^0_z)=\dim(M'_z)-1$ and $X_z\setminus M'_z=T^0_z$ for all $z\in T^0$\em.

Indeed, suppose there exists a point $z'\in T^0$ such that $\dim(T^0_{z'})\leq\dim(M'_{z'})-2$. Then there exists a simplex $\sigma'\in K$ such that $z'\in S':=\Phi(\sigma')$ and $\dim(\sigma')=\dim(M'_{z'})$ because $(K,\Phi)$ is compatible with $M'$ and $X$. As $z'\in T^0\cap S'$, we deduce that $\tau$ is a proper face of $\sigma'$ and so
\begin{multline*}
\dim(M'_{z'})=\dim(\sigma')\leq\dim(M'_x)=\dim(T^0_x)+1\\
=\dim(T^0_{z'})+1\leq\dim(M'_{z'})-2+1=\dim(M'_{z'})-1,
\end{multline*}
which is a contradiction. Thus, $\dim(T^0_z)=\dim(M'_z)-1$ for all $z\in T^0$.

Now let $z\in T^0$ and let $\sigma\in K$ be a simplex such that $z\in S:=\Phi(\sigma)$ and $S^0:=\Phi(\sigma^0)\subset M'$. Since $z\in T^0\cap S$, we deduce that $\tau\subset\sigma$ is a (proper) face of $\sigma$ and as $\dim(T^0_z)=\dim(M'_z)-1$ and $T^0\subset X\setminus M'$, we deduce $\dim(S^0_z)=\dim(M'_z)$. Observe that since $z\in T^0$ and $\dim(T^0_z)=\dim(M'_z)-1$, we have $\dim(T^0_z)\leq\dim(X_z\setminus M'_z)\leq\dim(M'_z)-1=\dim(T^0_z)$ and so $\dim(T^0_z)=\dim(X_z\setminus M'_z)$; hence, $T^0_z=X_z\setminus M'_z$ for all $z\in T^0$. 

\noindent{\bf S2.c.} Let ${\mathfrak G}$ be the collection of all simplices $\epsilon\in K$ such that $E^0:=\Phi(\epsilon^0)\subset\varrho^{-1}(V\setminus(M\cup Y_0))$ and define $Y_1:=\bigcup_{\epsilon\in{\mathfrak G}\setminus{\mathfrak F}}\varrho(E^0)$. We claim: \em $\dim(Y_{1,p})\leq\dim(M_p)-2$ for all $p\in\cl(M)$\em. 

Let $p\in\cl(Y_1)$ and suppose that $\dim(Y_{1,p})\geq\dim(M_p)-1$. Pick a point $q\in Y_1$ close to $p$ such that $\dim(Y_{1,q})=\dim(Y_{1,p})$ and $\dim(M_q)\leq\dim(M_p)$. Since $q\in Y_1\subset\cl(M)\setminus M$, we have $Y_{1,q}\subset\cl(M)_q\setminus M_q$; hence, by \cite[2.8.13]{bcr}, $\dim(Y_{1,q})\leq\dim(M_q)-1$. Thus,
$$
\dim(M_q)\leq\dim(M_p)\leq\dim(Y_{1,p})+1=\dim(Y_{1,q})+1\leq\dim(M_q)-1+1=\dim(M_q)
$$
and so $\dim(Y_{1,q})=\dim(M_q)-1$. As $Y_1\subset V\setminus Y_0$, we deduce that $\varrho^{-1}(q)$ is a singleton $\{z\}$. Let $\epsilon_1,\ldots,\epsilon_r$ be all simplices in ${\mathfrak G}\setminus{\mathfrak F}$ such that $z\in E_j=\Phi(\epsilon_j)$. Notice that $\dim(E^0_{j,z})\leq\dim(M'_{z})-2$ for each $j=1,\ldots,r$. Moreover, as$\varrho|_{M'}:M'\to M$ is a semialgebraic homeomorphism and $\varrho(M'_{z})\subset M_q$, we deduce $\dim(M'_{z})=\dim(\varrho(M'_{z}))\leq\dim(M_q)$. Thus,
$$
\dim(\pi(E^0_j)_q)\leq\dim(\pi(E^0_j))\leq\dim(E^0_j)=\dim(E^0_{j,z})\leq\dim(M'_{z})-2\leq\dim(M_q)-2
$$
and since $Y_{1,q}=\bigcup_{j=1}^r\pi(E^0_j)_q$, we deduce $\dim(M_q)-1=\dim(Y_{1,q})\leq\dim(M_q)-2$, which is a contradiction. Thus, $\dim(Y_{1,p})\leq\dim(M_p)-2$ for all $p\in\cl(M)$.

\noindent{\bf S2.d.} Define $Y:=\cl(Y_0\cup Y_1)\setminus M$ that satisfies $\dim(Y_{p})\leq\dim(M_p)-2$ for all $p\in\cl(M)$.

\noindent{\bf Step 3.} \em Construction of the continuous extension $G$ of $g$ to $V\setminus Y$\em. Write $\varrho^{-1}(q)=(q,G(q))$ for all $q\in V\setminus Y$ and observe that $G|_{M}=g$. We claim: \em $G:V\setminus Y\to R$ is continuous and consequently a semialgebraic function\em. 

\noindent{\bf S3.a.} Since $G|_{M_{\lc}}=g|_{M_{\lc}}$ is continuous and $M_{\lc}$ is open in $\cl(M)$, it is enough to check the continuity of $G$ at the points of $V\setminus(Y\cup M_{\lc})$.

\noindent{\bf S3.b.} \em We first check the continuity at the points $q\in V\setminus(Y\cup M)=V\setminus(\cl(Y)\cup M)$\em. 

Choose a point $q\in V\setminus(Y\cup M)$. We begin with the construction of fitting neighborhoods of $q$ in $\cl(M)$ and $\rho^{-1}(q)$ in $X$.

We know: $\varrho^{-1}(q)=\{x\}$ is a singleton and there exists a simplex $\tau\in K$ such that $x\in T^0$ and $T^0_x=X_x\setminus M'_x$; recall also $\dim(T^0_z)=\dim(M'_z)-1$ for all $z\in T^0$. Let $\sigma_1,\ldots,\sigma_m\in K$ be the simplices of $K$ such that $\dim(\sigma_i)=\dim(M'_x)$ and $\tau\subset\sigma_i$. The condition $\dim(T^0_z)=\dim(M'_z)-1$ for all $z\in T^0$ guarantees that $T^0\sqcup\bigsqcup_{j=1}^mS_i^0$ is an open neighborhood of $T^0$ in $X$. Define $C_1:=X\setminus(T^0\sqcup\bigsqcup_{j=1}^mS_i^0)$, which is a closed semialgebraic set and so $\varrho(C_1)$ is also a closed semialgebraic set. Since $\varrho^{-1}(q)=\{x\}$ and $x\not\in C_1$, we deduce $q\in X\setminus\varrho(C_1)$. Let $\ol{U}_1$ be a closed semialgebraic neighborhood of $q$ in $\cl(M)$ contained in $V\setminus(\cl(Y)\cup\varrho(C_1))$ such that $\ol{U}_1\cap M$ is semialgebraically connected (use that the germ $M_q$ is semialgebraically connected). Thus, $\ol{W}:=\varrho^{-1}(\ol{U}_1)$ is a closed semialgebraic neighborhood of $x$ in $X$ contained in $T^0\sqcup\bigsqcup_{j=1}^mS_i^0$. 

Notice that the restriction \em $\varrho|_{\ol{W}}:\ol{W}\to \ol{U}_1$ is a semialgebraic homeomorphism \em because $\ol{W}$ is closed and bounded and $\varrho|_{\ol{W}}$ is a bijective (continuous) semialgebraic map. Let us check now that \em $\ol{W}\setminus T^0$ is semialgebraically connected\em.

Indeed, since $T^0_z=X_z\setminus M'_z$ for each $z\in T^0$, we deduce $(X\setminus M')\cap\ol{W}=T^0\cap\ol{W}$ and so 
$$
\varrho(\ol{W}\setminus T^0)=\varrho(\ol{W}\setminus(X\setminus M'))=\varrho(\ol{W}\cap M')=\ol{U}_1\cap M.
$$
As $\ol{U}_1\cap M$ is semialgebraically connected, so is $\ol{W}\setminus T^0$. Since $S_i^0$ is open, $S_i^0\cap S_j^0=\varnothing$ if $i\neq j$ and $S_i^0\cap\ol{W}=\varnothing$, we deduce $m=1$ and $T^0\cup S_1^0$ is an open neighborhood of $x$ in $X$. As $\varrho|_{\ol{U}_1}^{-1}(z)=(z,G|_{\ol{U}_1}(z))$ and $\ol{U}_1$ is a semialgebraic neighborhood of $q$, we deduce that $G$ is continuous at $q$ because $\varrho|_{\ol{U}_1}^{-1}$ is continuous.

\noindent{\bf S3.c.} To finish let us prove that: \em $G$ is continuous at all points of $M\setminus M_{\lc}$\em. Suppose by contradiction that there exists a point $p\in M\setminus M_{\lc}$ such that $G$ is not continuous at $p$. Then we can find $\veps>0$ such that $p$ is adherent to the semialgebraic set $D:=\{y\in V\setminus Y:\,|G(p)-G(y)|>\veps\}$. By the curve selection lemma there exists a semialgebraic path $\alpha:[0,1]\to R^n$ such that $\alpha(0)=p$ and $\alpha((0,1])\subset D$. On the other hand, since $G|_M=g$ is continuous, there exists $\delta>0$ such that $|G(x)-G(p)|<\frac{\veps}{2}$ for each $x\in M\cap\Bb_n(p,\delta)$. Shrinking the domain of $\alpha$, we may assume that $\im\alpha\subset\Bb_n(p,\delta)$. Since $\im(\alpha)\subset\cl(M)\cap\Bb_n(p,\delta)$, there exists by the curve selection lemma a semialgebraic path $\beta:[0,1]\to R^n$ such that $\beta(0)=\alpha(1)$ and $\beta((0,1])\subset M\cap\Bb_n(p,\delta)$. Thus,
$$
\veps<|G(p)-G(\alpha(1))|=\lim_{t\to0^+}|G(p)-G(\beta(t))|\leq\tfrac{\veps}{2},
$$
which is a contradiction.

Hence, $g\in{\mathcal S}(M,R)$ can be extended to a semialgebraic function $G\in{\mathcal S}(V\setminus Y,R)$, as required.
\qed

The following example shows that the previous result is sharp.

\begin{example}\label{ce}
Consider the semialgebraic set $M:=\{x-y>0,y>0\}\cup\{(0,0,0)\}\subset R^3$ and the bounded semialgebraic function $f:M\setminus\{(0,0,0)\}\to R,\ (x,y,z)\mapsto\frac{x-y}{x}$. Observe that $M$ is appropriately embedded and $g:=zf$ can be extended continuously to the origin. However, $g$ cannot be extended to any neighborhood of the origin in $\cl(M)$ because such extension should value $z$ on the semialgebraic set $\cl(M)\cap\{y=0\}$ and $0$ on the semialgebraic set $\cl(M)\cap\{x-y=0\}$, which contradicts the continuity of $g$.
\end{example}

\section{Appropriate embedding of semialgebraic sets}\label{s4}

In the Introduction we presented the appropriately embedded semialgebraic sets and our aim in this section is to prove that every semialgebraic set $M\subset R^n$ can be appropriately embedded. The proof of this result requires some initial preparation. We denote by $\eta(M)$ the set of points $q\in\cl(M)\setminus M$ such that either $0\leq\dim(\cl(M)_q\setminus M_q)<\dim(M_q)-1$ or $M_q$ is not semialgebraically connected. Obviously, $M$ is appropriately embedded if and only if $\eta(M)=\varnothing$. Moreover, $\eta(M)$ is a semialgebraic set. Only the semialgebraicity of the set of all points $x\in\cl(M)\setminus M$ such that the germ $M_x$ is not semialgebraically connected requires some comment: This holds by \cite[4.2]{fgr} (although this result is stated for the case $R=\R$, the same proof works for an arbitrary real closed field $R$). 

\begin{thm}[Appropriate embedding]\label{se0}
Assume that $M\subset R^n$ is bounded. Then there exists a neighborhood $U$ of $\eta(M)$ such that $N:=M\setminus U$ is appropriately embedded as well as a surjective semialgebraic map $h:\cl(N)\to\cl(M)$ such that $h|_N:N\to M$ is a semialgebraic homeomorphism.
\end{thm}

We begin with the inspiring particular case when $\eta(M)$ is closed whose proof is strongly based on \cite[9.4.1, 9.4.4]{bcr}.

\begin{proof}[Proof of Theorem \em \ref{se0} \em when $\eta(M)$ is closed]
The proof is conducted in several steps.

\noindent{\em Initial preparation.}
Consider the semialgebraic sets $S_1:=\cl(M)\setminus\eta(M)$ and $S_2:=M$, which clearly have the same closure. Let $\ol{\Bb}$ be a closed ball centered at the origin of radius large enough to obtain $\cl(M)\subset\ol{\Bb}$. Consider the semialgebraic function $f:=\dist(\cdot,\eta(S_2))\in{\mathcal S}(\ol{\Bb},R)$ as well as the semialgebraic sets $S_1$ and $S_2$; since $\eta(S_2)$ is closed, we have $Z(f)=\eta(S_2)$. One can prove that: \em There exist $\veps>0$, semialgebraic sets $A_1,A_2\subset f^{-1}(\veps)$ and a semialgebraic map 
$$
\theta:[0,\veps]\times f^{-1}(\veps)\to f^{-1}([0,\veps])
$$ 
such that the restriction
$$
\theta|_{(0,\veps]}:(0,\veps]\times f^{-1}(\veps)\to f^{-1}((0,\veps])
$$
is a semialgebraic homeomorphism, the composition $\pi:=f\circ\theta:(0,\veps]\times f^{-1}(\veps)\to(0,\veps]$ is the projection onto the first factor, $\theta(\veps, x)=x$ for all $x\in f^{-1}(\veps)$ and $\theta((0,\veps]\times A_i)=S_i\cap f^{-1}((0,\veps])$\em. 

To that end, proceed similarly to the proof of \cite[9.4.4]{bcr} but taking care of the semialgebraic sets $S_i$ by using the triangulability of the semialgebraic function $f$ compatible with the semialgebraic sets $S_1$ and $S_2$ \cite[9.4.1]{bcr}. 

Let us identify the semialgebraic sets $A_i$. Given a semialgebraic set $T\subset R^n$ and $\delta>0$, we denote 
\begin{align*}
&T_{[0,\delta]}:=\{x\in T:\,\dist(x,\eta(S_2))\leq\delta\}=T\cap f^{-1}([0,\delta]),\\
&T_{\{\delta\}}:=\{x\in T:\,\dist(x,\eta(S_2))=\delta\}=T\cap f^{-1}(\delta). 
\end{align*}
We write 
$$
\zeta:=\theta|_{(0,\veps]}^{-1}:f^{-1}(0,\veps]\to(0,\veps]\times f^{-1}(\veps),\ x\mapsto(\zeta_1(x),\zeta_2(x)).
$$
Now given $0<\delta\leq\veps$, we have
$$
\{\delta\}\times f^{-1}(\veps)=\pi^{-1}(\delta)=\theta^{-1}(f^{-1}(\delta))\quad\text{and so}\quad f^{-1}(\delta)=\theta(\{\delta\}\times f^{-1}(\veps)).
$$
Thus, $f^{-1}(\veps)$ is semialgebraically homeomorphic to $f^{-1}(\delta)$ and it holds $\zeta_1(f^{-1}(\delta))=\{\delta\}$; in particular, $\zeta_1(x)=\dist(x,\eta(S_2))$ for all $x\in f^{-1}((0,\veps])$. Moreover,
\begin{multline*}
A_i=\theta(\{\veps\}\times A_i)=\theta((\{\veps\}\times f^{-1}(\veps))\cap((0,\veps]\times A_i))\\
=f^{-1}(\veps)\cap S_i\cap f^{-1}((0,\veps])=S_i\cap f^{-1}(\veps)=S_{i,\{\veps\}};
\end{multline*}
hence, $A_i=S_{i,\{\veps\}}$ and $A_1=\cl(A_2)$. So the semialgebraic sets $(0,\veps]\times S_{i,\{\veps\}}$ and $S_{i,[0,\veps]}$ are semialgebaically homeomorphic.

\noindent{\bf Step 1. }{\em Construction of the semialgebraic homeomorphisms.} Consider the semialgebraic maps
$$
\begin{array}{l}
g:S_1\to S_1\setminus S_{1,[0,\frac{\veps}{2}]},\ x\mapsto\begin{cases}
x&\text{ if $x\in S_1\setminus S_{1,[0,\veps]},$}\\
\theta(\frac{\veps}{2}+\frac{\dist(x,\eta(S_2))}{2},\zeta_2(x))&\text{ if $x\in S_{1,[0,\veps]},$}
\end{cases}\\[20pt]
h:\cl(S_1\setminus S_{1,[0,\frac{\veps}{2}]})\to\cl(S_1),\ x\mapsto\begin{cases}
x&\text{if $x\in\cl(S_1\setminus S_{1,[0,\veps]}),$}\\
\theta(2\dist(x,\eta(S_2))-\veps,\zeta_2(x))&\text{if $x\in\cl(S_{1,[0,\veps]}\setminus S_{1,[0,\frac{\veps}{2}]}).$}
\end{cases}
\end{array}
$$
As $\theta(\veps, x)=x$ for all $x\in f^{-1}(\veps)$, the semialgebraic maps $g,h$ are well defined at the `conflictive points' of the set $S_{1,\{\veps\}}$. Notice that $h\circ g=\id_{S_1}$ and $g\circ h|_{S_1\setminus S_{1,[0,\frac{\veps}{2}}]}=\id_{S_1\setminus S_{1,[0,\frac{\veps}{2}}]}$; hence, $S_i$ and $S_i\setminus S_{i,[0,\frac{\veps}{2}]}$ are semialgebraically homeomorphic. Moreover, as $\cl(S_1\setminus S_{1,[0,\frac{\veps}{2}]})$ is bounded and closed, so is $h(\cl(S_1\setminus S_{1,[0,\frac{\veps}{2}]}))$ and since $h(S_1\setminus S_{1,[0,\frac{\veps}{2}]})=S_1$, we conclude that $h$ is surjective.

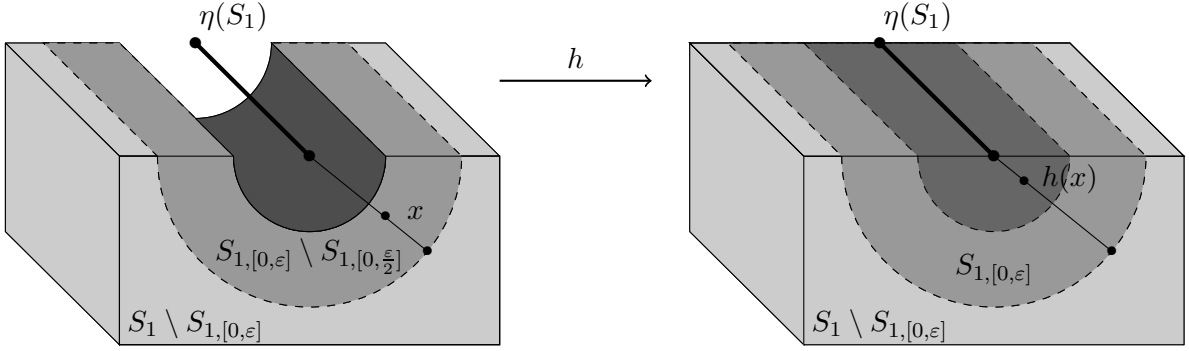
\begin{figure}[ht]
\centering
\begin{tikzpicture}

\draw (0,4) -- (1.5,2.5) -- (3,2.5) -- (1.5,4) -- (0,4);
\draw[fill=black!20!white] (0,4) -- (1.5,2.5) -- (3,2.5) -- (1.5,4) -- (0,4);

\draw (6.5,2.5) -- (5,4) -- (3.5,4) -- (5,2.5) -- (6.5,2.5);
\draw[fill=black!20!white] (6.5,2.5) -- (5,4) -- (3.5,4) -- (5,2.5) -- (6.5,2.5);

\draw (0,4) -- (0,1.5) -- (1.5,0) -- (1.5,2.5) -- (0,4);
\draw[fill=black!20!white] (0,4) -- (0,1.5) -- (1.5,0) -- (1.5,2.5) -- (0,4);

\draw (1.5,2.5) -- (1.5,0) -- (6.5,0) -- (6.5,2.5) -- (5,2.5) arc (360:180:1cm) -- (1.5,2.5);
\draw[fill=black!20!white] (1.5,2.5) -- (1.5,0) -- (6.5,0) -- (6.5,2.5) -- (5,2.5) arc (360:180:1cm) -- (1.5,2.5);

\draw[dashed] (2,2.5) arc (180:360:2cm) -- (4.5,4) -- (3.5,4) -- (5,2.5) arc (360:180:1cm) -- (1.5,4) -- (0.5,4) -- (2,2.5);
\draw[fill=black!40!white,dashed] (2,2.5) arc (180:360:2cm) -- (4.5,4) -- (3.5,4) -- (5,2.5) arc (360:180:1cm) -- (1.5,4) -- (0.5,4) -- (2,2.5);

\draw (1.5,2.5) -- (1.5,0) -- (6.5,0) -- (6.5,2.5) -- (5,2.5) arc (360:180:1cm) -- (1.5,2.5);

\draw (9,1.5) -- (10.5,0) -- (10.5,2.5) -- (9,4) -- (9,1.5);
\draw[fill=black!20!white] (9,1.5) -- (10.5,0) -- (10.5,2.5) -- (9,4) -- (9,1.5);

\draw (10.5,0) -- (15.5,0) -- (15.5,2.5) -- (10.5,2.5) -- (10.5,0);
\draw[fill=black!20!white] (10.5,0) -- (15.5,0) -- (15.5,2.5) -- (10.5,2.5) -- (10.5,0);

\draw (10.5,2.5) -- (15.5,2.5) -- (14,4) -- (9,4) -- (10.5,2.5);
\draw[fill=black!20!white] (10.5,2.5) -- (15.5,2.5) -- (14,4) -- (9,4) -- (10.5,2.5);

\draw[dashed] (11,2.5) arc (180:360:2cm) -- (13.5,4) -- (9.5,4) -- (11,2.5);
\draw[fill=black!40!white,dashed] (11,2.5) arc (180:360:2cm) -- (13.5,4) -- (9.5,4) -- (11,2.5);
\draw[fill=black!60!white,dashed] (10.5,4) -- (12,2.5) arc (180:360:1cm) -- (14,2.5) -- (12.5,4) -- (10.5,4);

\draw (10.5,2.5) -- (15.5,2.5) -- (14,4) -- (9,4) -- (10.5,2.5);

\draw (3,2.5) arc (180:360:1cm);

\draw (2.5,3) arc (270:360:1cm);

\draw[fill=black!70!white] (2.5,3) -- (3,2.5) arc (180:360:1cm) -- (5,2.5) -- (3.5,4) arc (360:270:1cm);

\draw[ultra thick] (2.5,4) -- (4,2.5);
\draw[ultra thick] (11.5,4) -- (13,2.5);

\draw[->,thick] (6.5,3.5) -- (8.5,3.5);

\draw (7.5,3.8) node{$h$};

\draw (4,2.5) node{$\bullet$};
\draw (2.5,4) node{$\bullet$};

\draw (13,2.5) node{$\bullet$};
\draw (11.5,4) node{$\bullet$};

\draw (4,1.175) node{$S_{1,[0,\veps]}\setminus S_{1,[0,\frac{\veps}{2}]}$};
\draw (2.5,0.25) node{$S_1\setminus S_{1,[0,\veps]}$};

\draw (13,1) node{$S_{1,[0,\veps]}$};
\draw (11.5,0.25) node{$S_1\setminus S_{1,[0,\veps]}$};

\draw (4,2.5) -- (5.55,1.24);

\draw (5.55,1.24) node{\scriptsize$\displaystyle\bullet$};
\draw (5,1.7) node{\scriptsize$\displaystyle\bullet$};
\draw (5.395,1.75) node{$x$};

\draw (13,2.5) -- (14.55,1.24);

\draw (14.55,1.24) node{\scriptsize$\displaystyle\bullet$};
\draw (13.4,2.175) node{\scriptsize$\displaystyle\bullet$};
\draw (14,2.2) node{$h(x)$};
\draw (12,4.35) node{$\eta(S_1)$};
\draw (3,4.35) node{$\eta(S_1)$};

\end{tikzpicture}
\caption{Action of the semialgebraic map $h:\cl(S_1\setminus S_{1,[0,\frac{\veps}{2}]})\to S_1$.}
\end{figure}

Write $N:=S_2\setminus S_{2,[0,\frac{\veps}{2}]}=M\setminus f^{-1}([0,\frac{\veps}{2}])$. Consider the open cover of $\cl(N)$
$$
\cl(N)=(\cl(N)\setminus f^{-1}([0,\tfrac{\veps}{2}]))\cup(\cl(N)\cap f^{-1}([0,\veps))))
$$ 

\noindent{\bf Step 2. }{\em For each point $q\in\cl(N)\cap f^{-1}([0,\veps))$, the germ $N_{q}$ is semialgebraically connected and either $\cl(N)_q=N_q$ or $\dim(\cl(N)_q\setminus N_q)=\dim(N_q)-1$.}

Indeed, fix $q\in\cl(N)\cap f^{-1}([0,\veps))$ and observe that $\frac{\veps}{2}\leq\dist(q,\eta(S_2))<\veps$. Clearly, the semialgebraic set $N_{[0,\veps]}=S_{2,[0,\veps]}\setminus S_{2,[0,\frac{\veps}{2}]}$ satisfies $N_{q}=N_{[0,\veps],q}$ and we obtain:
$$
\zeta|_{\cl(N_{[0,\veps]})}:\cl(N_{[0,\veps]})\to[\tfrac{\veps}{2},\veps]\times A_1,\ x\mapsto(\dist(x,\eta(S_2)),\zeta_2(x))
$$
is a semialgebraic homeomorphism; furthermore, $\zeta(N_{[0,\veps]})=(\frac{\veps}{2},\veps]\times A_2$. Let $U$ be a neighborhood of $q$ in $\cl(N_{[0,\veps]})$. Then there exists a semialgebraically connected open semialgebraic subset $V_1\subset A_1$ and $\frac{\veps}{2}\leq\delta_1<\delta_2<\veps$ such that $\theta([\delta_1,\delta_2)\times V_1)\subset U$ is a neighborhood of $q$ in $\cl(N_{[0,\veps]})$. Moreover, since for each $p\in\cl(S_2)\setminus\eta(S_2)$ the germ $S_{2,p}$ is semialgebraically connected, we may assume that $V_2:=V_1\cap A_2$ is also semialgebraically connected and since for each $p\in\cl(S_2)\setminus\eta(S_2)$ either $\cl(S_2)_p=S_{2,p}$ or $\dim(\cl(S_2)_p\setminus S_{2,p})=\dim(S_{2,p})-1$, we deduce that either $A_{1,z}=A_{2,z}$ or $\dim(A_{2,z}\setminus A_{1,z})=\dim(A_{1,z})-1$ for each $z\in A_2\setminus A_1$; hence, we may assume that either $V_1=V_2$ or $\dim(V_1\setminus V_2)=\dim(V_1)-1$. Observe that $\dim(V_i)=\dim(N_{q})-1$.

Moreover, $W_1:=\theta([\delta_1,\delta_2)\times V_1)$ is a semialgebraic neighborhood of $q$ in $\cl(N_{[0,\veps]})$ and 
$$
N\cap W_1=N_{[0,\veps]}\cap W_1=\begin{cases}
\theta([\delta_1,\delta_2)\times V_2)&\text{if $\delta_1>\frac{\veps}{2}$},\\
\theta((\frac{\veps}{2},\delta_2)\times V_2)&\text{if $\delta_1=\frac{\veps}{2}$}
\end{cases}
$$ 
because $\theta(\{\frac{\veps}{2}\}\times A_2)=S_{2,\{\frac{\veps}{2}\}}$. As $\theta((\delta_1,\delta_2)\times V_2)$ is semialgebraically connected, we deduce that $N_{q}$ is also semialgebraically connected. Moreover, we get 
\begin{multline*}
(\cl(N)\setminus N)\cap W_1=(\cl(N_{[0,\veps]})\setminus N_{[0,\veps]})\cap W_1=W_1\setminus(N_{[0,\veps]}\cap W_1)\\
=\begin{cases}
\theta([\delta_1,\delta_2)\times(V_1\setminus V_2))&\text{if $\frac{\veps}{2}<\delta_1$},\\
\theta(\{\frac{\veps}{2}\}\times V_1)\cup\theta([\delta_1,\delta_2)\times(V_1\setminus V_2))&\text{if $\frac{\veps}{2}=\delta_1$},
\end{cases}
\end{multline*}
which is either empty or has dimension $\dim(N_{q})-1$.

\noindent{\bf Step 3. }{\em For each $q\in\cl(N)\setminus f^{-1}([0,\frac{\veps}{2}])$ the germ $N_{q}$ is semialgebraically connected and either $\cl(N)_q=N_q$ or $\dim(\cl(N)_q\setminus N_q)=\dim(N_q)-1$.} 

Indeed, 
$$
q\in\cl(N)\setminus f^{-1}([0,\tfrac{\veps}{2}])\subset\cl(S_2)\setminus f^{-1}([0,\tfrac{\veps}{2}])\subset\cl(S_2)\setminus\eta(S_2);
$$ 
hence, $S_{2,q}=(M\setminus f^{-1}([0,\frac{\veps}{2}]))_q=N_{q}$ and so $\cl(S_2)_q=\cl(N_2)_q$. But since $S_{2,q}$ is semialgebraically connected if $q\in\cl(S_2)\setminus\eta(S_2)$, the same happens to $N_{q}$. Again $q\in\cl(S_2)\setminus\eta(S_2)$ implies that  either $\cl(S_2)_q=S_{2,q}$ (and so $\cl(N)_q=N_{q}$) or $\dim(\cl(S_2)_q\setminus S_{2,q})=\dim(S_{2,q})-1$ (and so $\dim(\cl(N)_q\setminus N_{q})=\dim(N_{q})-1$), as required. 
\end{proof}

\subsection{Basics on tubular neighborhoods for open simplices}

The proof of Theorem \ref{se0} in the general case when $\eta(M)$ is not necessarily closed is harder and requires the use of a suitable triangulation of $M$ as well as fitting tubular neighborhoods of some of its open simplices. To construct them, we present some preliminary results.

\begin{lemdef}\label{border}
Let $\tau\subset R^d$ be a $d$-dimensional simplex. Denote the faces of $\tau$ of dimension $d-1$ with $\vartheta_1,\ldots,\vartheta_{d+1}$. Then there exists a unique point $p_\tau\in\tau$ such that $\dist(p,\vartheta_1)=\dist(p,\vartheta_i)$ for $i=1,\ldots,d+1$. Moreover, $d(p_\tau,\partial\tau)=\max_{x\in\tau}\{d(x,\partial\tau)\}$ and this maximum is just attained at $p_\tau$\em. We call $p_\tau$ the \em incenter \em of $\tau$. 
\end{lemdef}
\begin{proof}
Let $p_1,\ldots,p_d,p_{d+1}$ be the vertices of $\tau$; we may assume $p_{d+1}=0$. Consider a linear change of coordinates $f:=(f_1,\ldots,f_d):R^d\to R^d$ that transforms $p_i$ into the point $(0,\ldots,0,1^{(i)},0,\ldots,0)$ for $i=1,\ldots,d$ (and $p_{d+1}=0$ into $0$). Then 
$$
f(\tau)=\{x_1\geq0,\ldots,x_d\geq0,1-\sum_{i=1}^dx_i\geq0\}\ \text{ and so }\ \tau=\Big\{f_1\geq0,\ldots,f_d\geq0,1-\sum_{i=1}^df_i\geq0\Big\}.
$$
Write $f_i(x):=\qq{u_i,x}$ for linearly independent $u_1,\ldots,u_d\in R^d$. We may assume 
$$
\vartheta_i=\tau\cap\{\qq{u_i,x}=0\}\ \forall i=1,\ldots,d\quad\text{ and }\quad\vartheta_{d+1}=\tau\cap\{1-\qq{u_1+\cdots+u_d,x}=0\}. 
$$
Observe that for each $x\in\tau$
$$
\dist(x,\vartheta_i)=\frac{\qq{u_i,x}}{\|u_i\|}\ \forall i=1,\ldots,d\quad \text{and}\quad \dist(x,\vartheta_{d+1})=\frac{1-\qq{u_1+\cdots+u_d,x}}{\|u_1+\cdots+u_d\|}.
$$
Consider now the system of linear equations
$$
\left\{\begin{array}{r}
\frac{\qq{u_1,x}}{\|u_1\|}-x_{n+1}=0,\\
\vdots\hspace{1.5cm}\\
\frac{\qq{u_d,x}}{\|u_d\|}-x_{n+1}=0,\\
\frac{1-\qq{u_1+\cdots+u_d,x}}{\|u_1+\cdots+u_d\|}-x_{n+1}=0,
\end{array}\right.\ \leadsto\ 
\left\{\begin{array}{r}
\qq{u_1,x}-\|u_1\|x_{n+1}=0,\\
\vdots\hspace{2cm}\\
\qq{u_d,x}-\|u_d\|x_{n+1}=0,\\
(\|u_1+\cdots+u_d\|+\sum_{i=1}^d\|u_i\|)x_{n+1}=1.
\end{array}\right.
$$
The unique solution of the previous linear system is $p_\tau$.

We now prove the second part. Denote $u_{d+1}:=u_1+\cdots+u_d$ and observe 
\begin{equation}\label{suma}
\sum_{i=1}^{d+1}\|u_i\|d(x,\vartheta_i)=1 
\end{equation}
for each $x\in\tau$. Notice that $d(x,\partial\tau)=\min_{i=1,\ldots,d+1}\{d(x,\vartheta_i)\}$ and so 
$$
d(p_\tau,\partial\tau)=d(p_\tau,\vartheta_i)=\frac{1}{\sum_{i=1}^{d+1}\|u_i\|}=:\lambda\quad\text{ for $i=1,\ldots,d+1$.}
$$
Assume that $x\in\tau$ satisfies $d(x,\partial\tau)\geq d(p_\tau,\partial\tau)$. Then 
$$
\min_{i=1,\ldots,d+1}\{d(x,\vartheta_i)\}\geq d(p_\tau,\vartheta_i)\quad\text{ for $i=1,\ldots,d+1$}
$$
and so $d(x,\vartheta_i)\geq d(p_\tau,\vartheta_i)=\lambda$ for $i=1,\ldots,d+1$. Using equation \eqref{suma}, we deduce
$$
d(x,\vartheta_i)=\lambda=d(p_\tau,\vartheta_i)\quad\text{ for $i=1,\ldots,d+1$}
$$
and this implies $x=p_\tau$, as required.
\end{proof}

\subsubsection{Construction of tubular neighborhoods}\label{ctn01}
Let $\tau\subset R^d$ be a $d$-dimensional simplex with incenter $p_\tau$ and $0<\veps<1$. We denote the simplex obtained as the cone of base $\tau\times\{0\}$ and vertex $(p_\tau,\veps^*\dist(p_\tau,\partial\tau))$ where $\veps^*:=\frac{\veps}{\sqrt{1-\veps^2}}$ by $\widehat{\tau}_\veps$. Let $m\geq1$, denote $n:=d+m$ and let $\pi:R^n\equiv R^d\times R^{n-d}\to R^n,\ (x,y)\mapsto (x,0)$ be the projection onto the first factor. For each $0<\veps<1$ consider the semialgebraic sets
$$
\begin{array}{l}
U_{\tau,\veps}:=\{(x,y)\in R^n:\,\dist((x,y),\tau\times\{0\})<\veps\dist((x,y),\partial\tau\times\{0\})\},\\[4pt]
\ol{U}_{\tau,\veps}:=\{(x,y)\in R^n:\,\dist((x,y),\tau\times\{0\})\leq\veps\dist((x,y),\partial\tau\times\{0\})\}.
\end{array}
$$
We prove next that $(U_{\tau,\veps},\pi)$ is a tubular neighborhood of $\tau^0$ for each $0<\veps<1$.

\begin{lem}\label{simplex}
The semialgebraic set $\ol{U}_{\tau,\veps}$ equals $\{(x,y)\in R^n:\,(x,\|y\|)\in\widehat{\tau}_\veps\}$ and $(U_{\tau,\veps},\pi)$ is a tubular neighborhood of $\tau^0$.
\end{lem}

\begin{figure}[ht]
\centering
\begin{tikzpicture}

\draw [fill=black!50!white, very thick](0,1.5) -- (8.25,3) -- (5.25,1.5) -- (0,1.5);

\draw [fill=black!60,fill opacity=0.2](0,1.5) -- (4.5,0) -- (5.25,1.5) -- (0,1.5);
\draw [fill=black!60,fill opacity=0.2](0,1.5) -- (5.25,1.5) -- (4.5,3.75) -- (0,1.5);
\draw [fill=black!60,fill opacity=0.2](4.5,0) -- (8.25,3) -- (5.25,1.5) -- (4.5,0);
\draw [fill=black!60,fill opacity=0.2](4.5,3.75) -- (8.25,3) -- (5.25,1.5) -- (4.5,3.75);

\draw[dashed] (4.5,3.75) -- (4.5,0);
\draw (4.5,1.875) node{\scriptsize$\displaystyle\bullet$};
\draw (0,1.5) node{\scriptsize$\displaystyle\bullet$};
\draw (8.25,3) node{\scriptsize$\displaystyle\bullet$};
\draw (5.25,1.5) node{\scriptsize$\displaystyle\bullet$};
\draw (11,1) node{\scriptsize$\displaystyle\bullet$};
\draw (15,3) node{\scriptsize$\displaystyle\bullet$};

\draw (6,2.2) node{$\tau$};
\draw (4.15,1.875) node{$p_\tau$};

\draw (12.1,1.8) node{$\tau$};
\draw (13.35,1.875) node{$p_\tau$};

\draw (1.5,0.5) node{$U_{\tau,\veps}$};
\draw (14.5,0.5) node{$U_{\tau,\veps}$};

\draw (0.6,3.25) node{$\tau\subset R^3$};

\draw (0.8,2.75) node{$\dim(\tau)=2$};

\draw (10.1,3.25) node{$\tau\subset R^3$};

\draw (10.3,2.75) node{$\dim(\tau)=1$};

\draw [fill=black!30,fill opacity=0.2,dashed] (13,0) to [out=180,in=180] (13,4) to [out=180,in=60] (12.1,3.346) -- (11,1) -- (12.6,0.10) [out=152,in=0] (13,0);

\draw [fill=black!30,fill opacity=0.2,draw=none] (13.86,0.58) -- (15,3) -- (13.5,3.845) to [out=148,in=0] (13,4) to [out=180,in=180] (13,0) to [out=0,in=-121.5] (13.86,0.58);

\draw [fill=black!60,fill opacity=0.2] (13,4) to [out=0,in=0] (13,0);
\draw [fill=black!60,fill opacity=0.2,dashed] (13,0) to [out=180,in=180] (13,4);

\draw (13,4) to [out=180,in=60] (12.1,3.346);
\draw (13,0) to [out=180,in=-28] (12.6,0.10);

\draw [ultra thick] (11,1) -- (13,2);
\draw [ultra thick,dashed] (13,2) -- (15,3);
\draw [ultra thick] (14.1,2.55) -- (15,3);

\draw (12.1,3.346) -- (11,1);
\draw (12.6,0.10) -- (11,1);
\draw (13.5,3.845) -- (15,3);
\draw (13.86,0.58) -- (15,3);
\draw (12.1,3.346) -- (13,2);

\draw (13,2) node{\scriptsize$\displaystyle\bullet$};

\draw (4.5,4) node{\scriptsize$\displaystyle(p_\tau,\veps^*\dist(p_\tau,\partial\tau))$};
\draw (4.5,-0.25) node{\scriptsize$\displaystyle(p_\tau,-\veps^*\dist(p_\tau,\partial\tau))$};
\draw (13.15,3.30) node{\tiny$\displaystyle\veps^*\dist(p_\tau,\partial\tau)$};
\draw[thick,dashed] (9,0) -- (9,4.25); 
\end{tikzpicture}
\caption{Examples of the semialgebraic set $\ol{U}_{\tau,\veps}$.}
\end{figure}
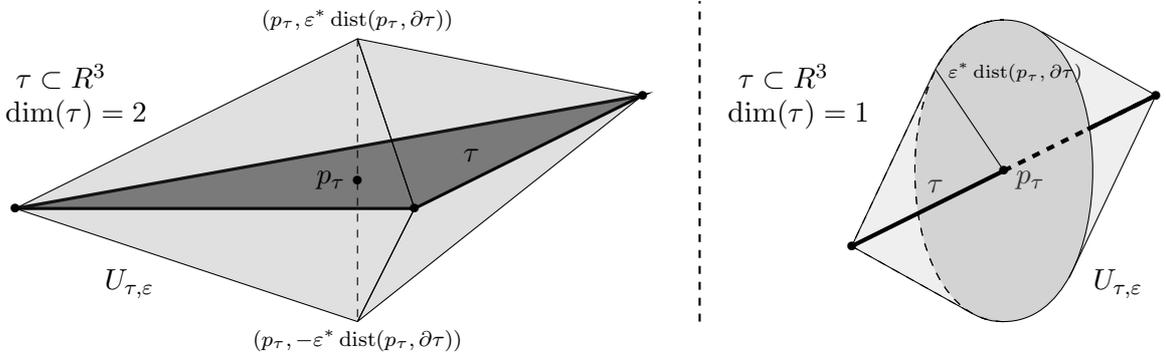
\begin{proof}
Observe first that $\dist((x,y),\tau\times\{0\})\leq\dist((x,y),\partial\tau\times\{0\})$ for all $(x,y)\in R^n$ and
$$
\{(x,y)\in R^d\times R^m:\,\dist((x,y),\tau\times\{0\})<\dist((x,y),\partial\tau\times\{0\})\}=\tau^0\times R^m.
$$
The last equality holds because the distance of a point $p$ to the simplex $\tau\times\{0\}$ equals the distance of $p$ to the affine subspace $W$ generated by one of its faces $\sigma\times\{0\}$. Even more there exists a point $q\in\sigma^0\times\{0\}$ such that 
$$
\dist(p,\tau\times\{0\})=\dist(p,\sigma^0\times\{0\})=\dist(p,W)=\dist(p,q).
$$

A straightforward computation shows that for each $p:=(x,y)\in\tau\times R^m$
$$
\dist((x,y),\partial\tau\times\{0\})^2=\dist(x,\partial\tau)^2+\|y\|^2\quad \text{and}\quad \dist((x,y),\tau\times\{0\})=\|y\|.
$$
Thus, as $0<\veps<1$,
\begin{multline*}
\ol{U}_{\tau,\veps}=\{(x,y)\in\tau\times R^m:\,\|y\|^2\leq\veps^2(\dist(x,\partial\tau)^2+\|y\|^2)\}=\\
\{(x,y)\in\tau\times R^m:\,\|y\|\leq\veps^*\dist(x,\partial\tau)\}.
\end{multline*}

Consider the semialgebraic set
\begin{equation}\label{s}
S:=\{(x,t)\in\tau\times R:\,0\leq t\leq\veps^*\dist(x,\partial\tau)\}.
\end{equation}
For our purposes it is enough that $S=\widehat{\tau}_\veps$. Of course, it holds $\dist(x,\partial\tau)=\min\{\dist(x,\vartheta_i):\,i=1,\ldots,d+1\}$ where $\vartheta_1,\ldots,\vartheta_{d+1}$ are the faces of dimension $d-1$ of $\tau$.

As we have seen in the proof of Lemma \ref{border} there exist independent vectors $u_1,\ldots,u_d\in R^d$ such that $\tau=\{f_1:=\qq{u_1,x}\geq0,\ldots,f_d:=\qq{u_d,x}\geq0,f_{d+1}:=1-\sum_{i=1}^df_i\geq0\}$. If we denote $u_{d+1}:=u_1+\cdots+u_d$, we have
$$
\dist(x,\vartheta_i)=\frac{f_i(x)}{\|u_i\|},\ \forall i=1,\ldots,d+1
$$
for each $x\in\tau$. Now observe that $(x,t)\in S$ if and only if
\begin{multline*}
0\leq t\leq \veps^*\dist(x,\partial\tau)=\veps^*\min\{\dist(x,\vartheta_i):\,i=1,\ldots,d+1\}\\ 
\iff\ \min\{\veps^*\dist(x,\vartheta_i)-t:\,i=1,\ldots,d+1\}\geq0,\ t\geq0.
\end{multline*}
Since $\tau=\{x\in R^d:\,f_i(x)\geq0:\,i=1,\ldots,d+1\}$, we conclude
$$
S=\{(x,t)\in R^{d+1}:\,\veps^*f_i(x)-t\|u_i\|\geq0, t\geq0:\,i=1,\ldots,d+1\},
$$
which is a convex polyhedron of $R^{d+1}$. As $\tau\times\{0\}\subset S$ and $(p_\tau,\veps^*\dist(p_\tau,\partial\tau))\in S$, we deduce that $S$ contains the cone of vertex $(p_\tau,\veps^*\dist(p_\tau,\partial\tau))$ and basis $\tau\times\{0\}$, that is, the simplex $\widehat{\tau}_\veps$. As we have seen in Lemma \ref{border}, $\dist(p_\tau,\partial\tau)=\max_{x\in\tau}\{d(x,\partial\tau)\}$ and this maximum is only attained at $p_\tau$. Thus, $(p_\tau,\veps^*\dist(p_\tau,\partial\tau))$ is a vertex of $S$ and in particular, $S$ is contained in $\tau\times[0,\veps^*\dist(x,\partial\tau)]$. Moreover, in view of the definition \eqref{s} of $S$, all vertices of $\tau\times\{0\}$ are also vertices of $S$, which are exactly $d+1$ ones. Now as $S$ is described by $d+2$ equations and has dimension $d+1$, it must be a simplex and has exactly $d+2$ vertices. These $d+2$ vertices are those of the simplex $\widehat{\tau}_\veps$, so we conclude $S=\widehat{\tau}_\veps$, as required.
\end{proof}

\subsubsection{Cross-sections of tubular neighborhoods} We analyze the structure of cross-sections of the semialgebraic set $\ol{U}_{\tau,\veps}$ by affine subspaces of dimension $d+1$ containing the $d$-dimensional simplex $\tau$. We begin with a straightforward consequence of Lemma \ref{simplex}.

\begin{cor}\label{intsimplex0}
Let $L$ be the affine subspace generated by a $d$-dimensional simplex $\tau\subset R^n$, let $p\in R^n\setminus L$ and let $\cc{L,p}$ be the affine subspace generated by $p$ and $L$. Denote the half-space of $\cc{L,p}$ determined by $L$ and that contains $p$ by $\cc{L,p}^+$. Let $q$ be the unique point of $\cc{L,p}^+$ such that $\dist(q,L)=\dist(q,p_\tau)=\veps^*\dist(p_\tau,\partial\tau)$. Then $\cc{L,p}^+\cap\ol{U}_{\tau,\veps}$ is the simplex obtained by considering the convex hull of the set $\tau\cup\{q\}$.
\end{cor}

\begin{cor}\label{smallveps}
Let $\sigma\subset R^n$ be a simplex and let $\tau$ be one of its faces. Then $\ol{U}_{\tau,\veps}\setminus\partial\tau$ only meets the faces of $\sigma$ that contain $\tau$ for each $\veps>0$ small enough. 
\end{cor}
\begin{proof}
Let $m:=\dim(\sigma)$ and $\beta_1,\ldots,\beta_r$ the $(m-1)$-dimensional faces of $\sigma$ that do not contain $\tau$. Let $\veps>0$ be such that $\veps^*d(p_\tau,\partial\tau)<\min_{i=1,\ldots,r}\{d(p_\tau,\beta_i)\}$. Denote the affine subspace generated by $\sigma$ with $W$ and the hyperplane of $W$ generated by $\beta_i$ with $H_i$. Notice that each face of $\sigma$ that does not contain $\tau$ is contained in some face $\beta_i$ and therefore in some $H_i$ for $i=1,\ldots,r$.

Let $\veps>0$ be such that
$$
\veps^*d(p_\tau,\partial\tau)<\min_{i=1,\ldots,r}\{d(p_\tau,H_i)\} 
$$
and $H_i':=\{h_i'=0\}$ the hyperplane of $R^n$ orthogonal to $W$ such that $W\cap H_i'=H_i$. Notice that the ball $\Bb_n(p_\tau,\veps^*d(p_\tau,\partial\tau))$ does not meet $H_i'$ for $i=1,\ldots,r$. We may assume that it is contained in the strict half space ${H_i'}^{>}:=\{h_i'>0\}$. Since $\tau\not\subset\beta_i$, we deduce $\tau\not\subset H_i'$ and so $\tau^0\cap H_i'=\varnothing$. As $p_\tau\in{H_i'}^{>}$, we conclude $\tau^0\subset {H_i'}^{>}$.

Let $L$ be the affine subspace generated by $\tau$. By Corollary \ref{intsimplex0}
$$
\ol{U}_{\tau,\veps}=\bigcup_q\cc{L,q}^+\cap\ol{U}_{\tau,\veps}
$$
where $q\in\Bb_n(p_\tau,\veps^*d(p_\tau,\partial\tau))\cap(p_\tau+L^\bot)$ and $\cc{L,q}^+\cap\ol{U}_{\tau,\veps}$ is the simplex obtained by considering the convex hull of the set $\tau\cup\{q\}$. Thus, $(\cc{L,q}^+\cap\ol{U}_{\tau,\veps})\setminus\partial\tau\subset{H_i'}^{>}$ for $i=1,\ldots,r$. Hence, $\ol{U}_{\tau,\veps}\setminus\partial\tau\subset{H_i'}^{>}$ for $i=1,\ldots,r$ and so $\ol{U}_{\tau,\veps}\setminus\partial\tau$ does not meet the faces of $\sigma$ that do not contain the simplex $\tau$.
\end{proof}

\begin{lem}\label{intsimplex}
Let $\sigma\subset R^n$ be an $n$-dimensional simplex such that the simplex $\tau$ is one of its faces. Denote the affine subspace generated by $\tau$ with $L$ and its dimension with $d$. Let $\veps>0$ be small enough such that $\ol{U}_{\tau,\veps}\setminus\partial\tau$ only meets the faces of $\sigma$ that contain $\tau$. Then $\sigma\cap\cc{L,p}\cap\ol{U}_{\tau,\veps}=\cc{L,p}^+\cap\ol{U}_{\tau,\veps}$ for each $p\in\sigma\setminus\tau$.
\end{lem}
\begin{proof}
By Corollary \ref{intsimplex0} there exists a point $q\in\cc{L,p}^+$ such that $\cc{L,p}^+\cap\ol{U}_{\tau,\veps}$ is the convex hull of $\tau\cup\{q\}$. Thus, since $\sigma$ is convex and $\tau$ is a face of $\sigma$, it is enough to check that $q$ belongs to $\sigma$. Assume by contradiction that $q$ does not belong to $\sigma$. Observe that $\pol:=\sigma\cap\cc{L,p}$ is a convex polyhedron of dimension $d+1$ because it contains $\{p\}\cup\tau$. As $\ol{U}_{\tau,\veps}$ is a neighborhood of $\tau^0$, the segment connecting $p$ with $p_\tau$ meets $\ol{U}_{\tau,\veps}$ in a (maybe smaller) segment; hence, we may assume $p\in\ol{U}_{\tau,\veps}\cap\Int(\pol)$. As $\cc{L,p}^+\cap\ol{U}_{\tau,\veps}$ is a simplex and $q\not\in\sigma$, the segment that connects $q$ and $p$ meets $\partial\pol$ in a point $p_0\in\partial\pol$. Observe that $p_0\not\in L$ because both $q$ and $p$ belong to the convex set $\cc{L,p}^+\setminus L$; hence, $p_0$ belongs to a $d$-dimensional face of $\pol$ different from $\tau$. This $d$-dimensional face corresponds to the intersection of $\cc{L,p}$ with a face of $\sigma$ that does not contain $\tau$. We conclude that $\ol{U}_{\tau,\veps}$ meets a face of $\sigma$ that does not contain $\tau$, which contradicts our choice of $\veps$.
\end{proof}

\subsubsection{Appropriate embedment of some essential differences of semialgebraic sets} 
Let $S$ be the difference between a simplex $\sigma$ and the closure $\ol{U}_{\tau,\veps}$ of a suitable tubular neighborhood of the interior of one of its faces $\tau$. We continue proving that the obstructing set $\eta(S)$ is contained in $\partial\tau$. In fact, $S$ is appropriately embedded but it is cumbersome to prove. For our purposes the following is enough:

\begin{lem}\label{lcd}
Let $\sigma\subset R^n$ be a simplex and $\tau$ a face of $\sigma$. Let $\veps>0$ be small enough such that $\ol{U}_{\tau,\veps}\setminus\partial\tau$ only meets the faces of $\sigma$ that contain $\tau$. Consider the semialgebraic sets $S:=\sigma^0\setminus\ol{U}_{\tau,\veps}$ and $\widehat{S}:=\sigma\setminus\ol{U}_{\tau,\veps}$. Then the sets $\eta(S)$ and $\eta(\widehat{S})$ are contained in $\partial\tau$.
\end{lem}
\begin{figure}[ht]
\centering
\begin{tikzpicture}
\draw[fill=black!20!white] (0,2) -- (5,0) -- (4,4) -- (0,2);
\draw[fill=black!20!white] (4,4) -- (5,0) -- (6,2) -- (4,4);
\draw[fill=black!40] (0,2) -- (3,2.5) -- (5,0) -- (0,2);
\draw[ultra thick](0,2) -- (5,0);

\draw[fill=black!20!white] (8,2) -- (11,2.5) -- (13,0) -- (12,4) -- (8,2);
\draw[fill=black!20!white] (12,4) -- (13,0) -- (14,2) -- (12,4);
\draw[fill=black!40] (8,2) -- (11,1.5) arc(-9.462322208025803:9.462322208025803:3.04138126514911cm) -- (8,2);
\draw[fill=black!40] (13,0) -- (11,2.5) arc(9.462322208025803:-9.462322208025803:3.04138126514911cm) -- (13,0);
\draw[ultra thick,dashed](8,2) -- (13,0);

\draw (2,3.5) node{$\sigma$};
\draw (2.25,0.75) node{$\tau$};
\draw (2.75,1.75) node{$\ol{U}_{\tau,\veps}\cap\sigma$};
\draw (9.5,3.75) node{$\widehat{S}:=\sigma\setminus\ol{U}_{\tau,\veps}$};
\draw (10.25,0.75) node{$\tau$};
\draw (7,2) node{\Large$\displaystyle\leadsto$};

\end{tikzpicture}
\caption{Appropriate embedment of the differences $S:=\sigma^0\setminus\ol{U}_{\tau,\veps}$ and $\widehat{S}:=\sigma\setminus\ol{U}_{\tau,\veps}$.}
\end{figure}
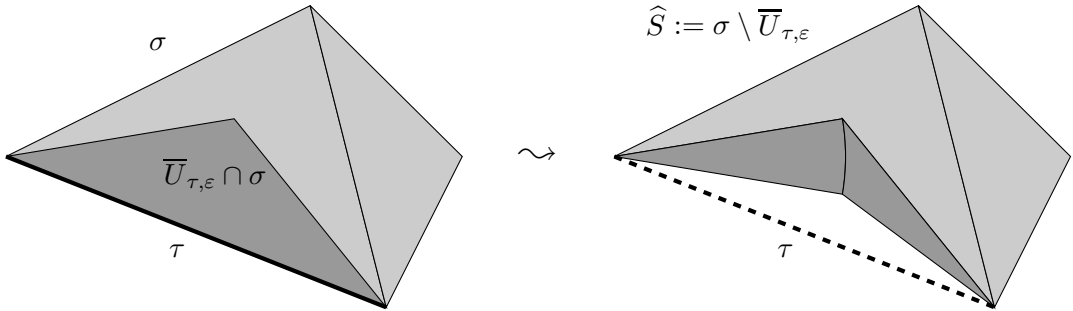
\begin{proof}
As $\sigma$ and $\sigma^0$ are appropriately embedded, there is nothing to prove for points in the set $\cl(S)\setminus\ol{U}_{\tau,\veps}=\cl(\widehat{S})\setminus\ol{U}_{\tau,\veps}$. Thus, it only remains to check that: \em For each $q\in\cl(S)\cap(\ol{U}_{\tau,\veps}\setminus\partial\tau)$, the germs $S_q$ and $\widehat{S}_q$ are semialgebraically connected, $\dim(\cl(S)_q\setminus S_q)=\dim(S_q)-1$ and $\dim(\cl(\widehat{S})_q\setminus\widehat{S}_q)=\dim(\widehat{S}_q)-1$\em.

First, we may assume that $\sigma$ has dimension $n$. Let $W$ be the affine subspace generated by $\tau$ and $\pi:R^n\to W$ the orthogonal projection onto $W$; recall that by Lemma \ref{simplex}, $\ol{U}_{\tau,\veps}=\{x\in R^n:\,(\pi(x),\|x-\pi(x)\|)\in\widehat{\tau}_\veps\}$. Using the fact that $\ol{U}_{\tau,\veps}$ is a bounded and closed convex set,  we construct next \em a semialgebraic homeomorphism $g_1:R^n\to R^n$ such that $g_1(\ol{U}_{\tau,\veps})=\Bb_n(0,1)$ and all lines through the origin are kept invariant\em.

After a change of coordinates, we may assume that $p_\tau$ is the origin. For each $x\in R^n$ let $\ell_x$ be the half-line from the origin passing through $x$. Since $\ol{U}_{\tau,\veps}$ is a closed and bounded convex set and the origin is an interior point of $\ol{U}_{\tau,\veps}$, the set $\ell_x\cap \ol{U}_{\tau,\veps}$ is a (non trivial) closed bounded interval; denote the endpoint of $\ell_x\cap \ol{U}_{\tau,\veps}$ different from the origin with $f(x)$. It holds that $f(x)\in\partial \ol{U}_{\tau,\veps}$ by \cite[11.2.4]{ber1} and the map $\delta:R^n\setminus\{0\}\to[0,+\infty),\ x\mapsto \|f(x)\|$ is continuous by \cite[11.3.1.2]{ber1}; hence, it is semialgebraic. Observe $f(x)=\delta(x)\frac{x}{\|x\|}$ for all $x\in R^n\setminus\{0\}$ and that $\delta$ is constant on every half-line from the origin. Consider now the semialgebraic maps
$$
g_1:R^n\to R^n,\, x\mapsto\begin{cases}
0&\text{if $x=0$,}\\
\frac{x}{\delta(x)}&\text{if $x\neq0$}
\end{cases}\qquad\text{and}\qquad g_2:R^n\to R^n,\, x\mapsto\begin{cases}
0&\text{if $x=0$,}\\
x\delta(x)&\text{if $x\neq0$.}
\end{cases}
$$
The continuity of such maps follows from the following fact: Since $0$ is an interior point of $\ol{U}_{\tau,\veps}$ and $\partial \ol{U}_{\tau,\veps}$ is closed and bounded, there exist $M,m>0$ such that $m<\delta(x)<M$ for all $x\in R^n\setminus\{0\}$. A straightforward computation shows in addition that $g_1\circ g_2=g_2\circ g_1=\id_{R^n}$ and therefore both are semialgebraic homeomorphisms. Notice moreover that $g_1(\ol{U}_{\tau,\veps})=\Bb_n(0,1)$. 

After this preparation, we are ready to prove the statement. Fix a point $q\in\cl(S)\cap\ol{U}_{\tau,\veps}\setminus\partial\tau$. Let $H_1,\ldots,H_r$ be the hyperplanes of $R^n$ generated by those respective $(n-1)$-dimensional faces of $\sigma$ that contain $\tau$; clearly $W=H_1\cap\ldots\cap H_r$. Write $H_i:=\{h_i=0\}$ and assume $\sigma\subset\{h_1\geq0,\ldots,h_r\geq0\}$. Denote $T:=\{h_1>0,\ldots,h_r>0\}\setminus\ol{U}_{\tau,\veps}$ and $\widehat{T}:=\{h_1\geq0,\ldots,h_r\geq0\}\setminus\ol{U}_{\tau,\veps}$ and notice the following: Since $\ol{U}_{\tau,\veps}\setminus\partial\tau$ only meets the faces of $\sigma$ that contain $\tau$, we have \em $S_q=T_q$ and $\widehat{S}_q=\widehat{T}_q$ for each $q\in\cl(S)\cap(\ol{U}_{\tau,\veps}\setminus\partial\tau)$\em. 

On the other hand, as $g_1$ keeps all lines through the origin invariant, we conclude:
\begin{align*}
&T':=g_1(T)=\{h_1>0,\ldots,h_r>0\}\setminus\Bb_n(0,1),\\ 
&\widehat{T}':=g_1(\widehat{T})=\{h_1\geq0,\ldots,h_r\geq0\}\setminus\Bb_n(0,1),\\
&g_1(\cl(S)\cap(\ol{U}_{\tau,\veps}\setminus\partial\tau))\subset\partial(\Bb_n(0,1)\cap\{h_1\geq0,\ldots,h_r\geq0\}).
\end{align*}
For each point $z\in\partial(\Bb_n(0,1)\cap\{x_1\geq0,\ldots,x_r\geq0\})$ the germs $T'_z$ and $\widehat{T}'_z$ are semialgebraically connected and in addition $\dim(\cl(T')_z\setminus T'_z)=\dim(T'_z)-1$ and $\dim(\cl(\widehat{T}')_z\setminus\widehat{T}'_z)=\dim(\widehat{T}'_z)-1$. Of course, the same holds for the germs $S_q=T_q$ and $\widehat{S}_q=\widehat{T}_q$ for each $q\in\cl(S)\cap(\ol{U}_{\tau,\veps}\setminus\partial\tau)$, as required.
\end{proof}

\subsubsection{Separation of tubular neighborhoods of two different simplices}
We prove now that given two simplices whose intersection is a common face, we can find disjoint tubular neighborhoods for their relative interiors. We need a preliminary result concerning strict separation of simplices meeting just in a face that we include for the sake of completeness. Namely,

\begin{lem}[Strict separation of simplices]\label{sp}
Let $\tau_1,\tau_2\subset R^n$ be two simplices such that $\tau_1\cap\tau_2$ is a common face $\vartheta$. Then there exists a hyperplane $H:=\{h=0\}$ of $R^n$ such that $\tau_i\cap H=\vartheta$ and $\tau_i\setminus\vartheta\subset\{(-1)^ih>0\}$.
\end{lem}
\begin{proof}
We analyze the case that \em $\tau_1\cap\tau_2$ is a common vertex of the simplices $\tau_1$ and $\tau_2$ \em first.

Indeed, we may assume $\dim(\tau_1)=n$ and $\dim(\tau_2)=m\leq n$; in fact, we suppose $\tau_1$ is the simplex whose vertices are the origin ${\bf0}$ and points $e_i:=(0,\ldots,0,1^{(i)},0,\ldots,0)$ whose coordinates are all zero except for the $i$th, which is equal to one. After reordering the vertices of $\tau_1$, we assume that $\vartheta$ is the simplex of vertices ${\bf0},e_1,\ldots,e_d$. Let $u_{d+1},\ldots,u_m$ be the remaining vertices of $\tau_2$ and consider the projection 
$$
\pi:R^n\to R^{n-d},\ x:=(x_1,\ldots,x_n)\mapsto x':=(x_{d+1},\ldots,x_n).
$$
Let $\tau_i':=\pi(\tau_i)$, denote $e_k':=\pi(e_k)$ for $k=d+1,\ldots,n$ and $u_j':=\pi(u_j)$ for $j=d+1,\ldots,n$. Of course, $\tau_1'$ is the simplex of vertices ${\bf0},e_{d+1}',\ldots,e_n'$ and $\tau_2'$ is the simplex of vertices ${\bf0},u_{d+1}',\ldots,u_m'$. Let us check that: $\tau_1'\cap\tau_2'=\{{\bf0}\}$. 

Pick a point $p'\in\tau_1'\cap\tau_2'$ and write 
$$
(p_{d+1},\ldots,p_n)=p'=p_{d+1}e_{d+1}'+\cdots+p_ne_n'=\mu_{d+1}u_{d+1}'+\cdots+\mu_mu_m'
$$
where $p_k,\mu_j\geq0$, $\sum_{k=d+1}^np_k\leq1$ and $\sum_{j=d+1}^m\mu_j\leq1$. Denote $p:=(p_1,\ldots,p_n):=\sum_{j=d+1}^m\mu_ju_j$ and let $q:=(q_1,\ldots,q_n):=\frac{1}{2n(\|p\|+1)}p+\frac{1}{2n}\sum_{\ell=1}^de_{\ell}$. Notice that 
$$
q_\ell=\begin{cases}
\frac{1}{2n}(\frac{p_\ell}{(\|p\|+1)}+1)\geq0&\text{ for $\ell=1,\ldots,d$},\\ 
p_\ell\frac{1}{2n(\|p\|+1)}\geq0&\text{ for $\ell=d+1,\ldots,n$}
\end{cases}\quad\text{and}\quad\sum_{\ell=1}^nq_\ell=\frac{d}{2n}+\frac{1}{2n}\sum_{\ell=1}^n\frac{p_\ell}{(\|p\|+1)}<1,
$$
so $q\in\tau_1$. On the other hand,
$$
q=\sum_{\ell=1}^d\frac{1}{2n}e_{\ell}+\sum_{j=d+1}^n\frac{1}{2n(\|p\|+1)}\mu_ju_j\quad\text{and}\quad \frac{d}{2n}+\sum_{j=d+1}^n\frac{\mu_j}{2n(\|p\|+1)}\leq1
$$
so $q\in\tau_2$. As $\tau_1\cap\tau_2=\vartheta$, we deduce $q\in\vartheta$ and so $\mu_j=0$ for $j=d+1,\ldots,n$; hence, $p'={\bf0}$, that is, $\tau_1'\cap\tau_2'=\{{\bf0}\}$.

Therefore it is enough to separate $\tau_1'$ and $\tau_2'$ by a hyperplane of $R^{n-d}$ and we can assume from the beginning that $\vartheta=\tau_1\cap\tau_2$ is just a vertex (the origin). Let us construct two disjoint open convex neighborhoods $A_1$ and $A_2$ of $\tau_1\setminus\{0\}$ and $\tau_2\setminus\{0\}$. Then we have reduced everything to the separation of these two sets using Hahn-Banach's Theorem \cite[11.4.5]{ber1}.

\noindent\begin{minipage}[t]{0.555 \textwidth}
Indeed, let $\eta_i$ be the face of $\tau_i$ such that $\tau_i$ is the bounded cone of basis $\eta_i$ and vertex ${\bf0}$. Consider the infinite convex cone 
$$
C_i:=\{tx:\,x\in\tau_i, t\in[0,+\infty)\}
$$ 
and the closed semialgebraic set $T_i:=\{\frac{x}{\|x\|}:\, x\in\tau_i\}$ contained in the sphere $S^n:=\{x\in R^n:\,\|x\|=1\}$. Notice that $T_1\cap T_2=\varnothing$ and $C_i=\{ty:\, y\in T_i\}$. Let $V_1,V_2$ be two disjoint open subsets of $S^n$ such that $T_i\subset V_i$. Notice that 
$$
C_i\setminus\{{\bf0}\}\subset A_i:=\{tx:\,x\in V_i, t\in(0,+\infty)\}
$$
and $A_1\cap A_2=\varnothing$. After shrinking $A_i$ if necessary, we may assume that $A_i$ is an infinite convex cone: Define $\delta_i:=\dist(\eta_i,R^n\setminus A_i)$, which is strictly positive because $\eta_i$ is closed and bounded, $R^n\setminus A_i$ is closed and $\eta_i\cap(R^n\setminus A_i)=\varnothing$. 
\end{minipage}
\hfill 
\begin{minipage}[t]{0.44\textwidth}
\centering
\raisebox{-15.5em}{\begin{tikzpicture}
\draw[fill=black!30!white,draw=none,opacity=0.5](1,2) -- (1,6) -- (6,6) -- (1,2);
\draw[fill=black!40!white,draw=none,opacity=0.5](1,2) -- (1.5,6) -- (5.5,6) -- (1,2);
\draw[fill=black!30!white,draw=none,opacity=0.5](1,2) -- (6,3) -- (6,0) -- (4,0) -- (1,2);
\draw[fill=black!40!white,draw=none,opacity=0.5](1,2) -- (6,2.5) -- (6,0) -- (4.5,0) -- (1,2);
\draw[fill=black!40!white,draw=none](1,2) -- (1.8,6) -- (5,6) -- (1,2);
\draw[fill=black!40!white,draw=none](1,2) -- (5,0) -- (6,0) -- (6,2) -- (1,2);
\draw[fill=black!60!white,very thick](1,2) -- (3.5,4.5) -- (1.5,4.5) -- (1,2);
\draw[fill=black!60!white,very thick](1,2) -- (4,2) -- (4,0.5) -- (1,2);
\draw (1,2) -- (1.8,6);
\draw (1,2) -- (5,6);
\draw (1,2) -- (5,0);
\draw (1,2) -- (6,2);
\draw[dashed] (1,2) -- (1,6);
\draw[dashed] (1,2) -- (6,6);
\draw[dashed] (1,2) -- (1.5,6);
\draw[dashed] (1,2) -- (5.5,6);
\draw[dashed] (1,2) -- (6,2.5);
\draw[dashed] (1,2) -- (4.5,0);
\draw[dashed] (1,2) -- (6,3);
\draw[dashed] (1,2) -- (4,0);
\draw[dashed,thick](1,0.5) arc (270:460:1.5cm);
\draw[very thick](2.060660172,3.060660172) arc (45:79:1.5cm);
\draw[very thick](2.5,2) arc (0:-26:1.5cm);

\draw[ultra thick] (0,1.5) -- (6,4.5);

\draw (1.55,3) node{$T_1$};
\draw (2.125,3.75) node{$\tau_1$};
\draw (2.5,4.75) node{$\eta_1$};
\draw (3,5.5) node{$C_1$};
\draw (0.65,5.5) node{$A_1$};
\draw (0.5,4.75) node{$A_1'$};
\draw[->,thick] (0.75,4.75) -- (1.5,4.75);

\draw (5,4.5) node{$H$};

\draw (2.1,1.75) node{$T_2$};
\draw (3.5,1.35) node{$\tau_2$};
\draw (4.35,1.35) node{$\eta_2$};
\draw (5.25,0.9) node{$C_2$};
\draw (3.25,0.15) node{$A_2$};
\draw (2.5,0.5) node{$A_2'$};
\draw[->,thick] (2.75,0.5) -- (3.75,0.5);

\draw (1.25,0.9) node{$S^n$};

\draw (1,2) node{\scriptsize$\displaystyle\bullet$};
\end{tikzpicture}}

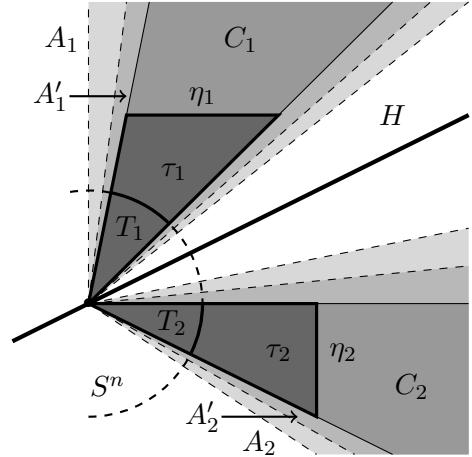
\captionof{figure}{Separation technique}
\end{minipage}

\noindent We claim: \em The open semialgebraic set $W_i:=\{x\in R^n:\,\dist(x,\eta_i)<\frac{\delta_i}{2}\}\subset A_i$ is convex\em.

Indeed, if $x,y\in W_i$, there exist $x_0,y_0\in\eta_i$ such that $\|x-x_0\|<\frac{\delta_i}{2}$ and $\|y-y_0\|<\frac{\delta_i}{2}$. Fix $\lambda\in[0,1]$ and let us see that $\lambda x+(1-\lambda)y\in W_i$. As $\eta_i$ is convex, $\lambda x_0+(1-\lambda)y_0\in\eta_i$. Moreover,
$$
\|\lambda x+(1-\lambda)y-(\lambda x_0+(1-\lambda)y_0)\|\leq\lambda\|x-x_0\|+(1-\lambda)\|y-y_0\|<\tfrac{\delta_i}{2};
$$
hence, $\lambda x+(1-\lambda)y\in W_i$, that is, $W_i$ is convex.

Notice that the open convex sets $A_i':=\{tx:\,x\in W_i,\ t\in(0,+\infty)\}\subset A_i$ contain $C_i\setminus\{{\bf0}\}$ and so $\tau_i\setminus\{0\}$, as required.
\end{proof}

\begin{lem}[Small intersection]\label{smallint}
Let $\tau_1,\tau_2\subset R^n$ be two simplices such that the intersection $\vartheta:=\tau_1\cap\tau_2=\partial\tau_1\cap\partial\tau_2$ is either empty or a common face. Then $\ol{U}_{\tau_1,\veps}\cap\ol{U}_{\tau_2,\veps}=\vartheta$ and in particular $U_{\tau_1,\veps}\cap U_{\tau_2,\veps}=\varnothing$ for $\veps>0$ small enough.
\end{lem}
\begin{proof}
Let $W_i\subset R^n$ be the affine subspace generated by $\tau_i$ and $\pi_i:R^n\to R^n$ the orthogonal projection onto $W_i$. By Lemma \ref{simplex} we know
\begin{equation}\label{rot}
\ol{U}_{\tau_i,\veps}=\{p\in R^n:\,(\pi({\tt p}),\|p-\pi({\tt p})\|)\in\widehat{\tau}_{i,\veps}\}
\end{equation}
where $\widehat{\tau}_{i,\veps}\subset W_i\times R$ is the $(d+1)$-dimensional simplex obtained as the (convex) cone whose basis is $\tau_i\times\{0\}$ and whose vertex is the point $v_{\tau_i}:=(p_{\tau_i},\veps^*\dist(p_{\tau_i},\partial\tau_i))$, see Lemma \ref{simplex}.

By Lemma \ref{sp} there exists a hyperplane $H:=\{h=0\}$ of $R^n$ such that $\tau_i\cap H=\vartheta$ and $\tau_i\setminus\vartheta\subset\{(-1)^ih>0\}$. Thus, $(-1)^ih(p_{\tau_i})>0$ and we pick $\veps_0>0$ such that $\Bb_n(p_{\tau_i},\veps_0)\subset\{(-1)^ih>0\}$. Choose $\veps>0$ small enough such that $\veps^*\dist(p_{\tau_i},\partial\tau_i)<\veps_0$ (see \S\ref{ctn01}). Recall that $\ol{U}_{\tau_i,\veps}$ is by \eqref{rot} the union of all cones of basis $\tau_i$ and vertices contained in the set 
$$
\Vv_i:=\{q\in R^n:\,\pi_i(q)=p_{\tau_i}\ \text{and}\ \|q-p_{\tau_i}\|=\veps^*\dist(p_{\tau_i},\partial\tau_i)\}.
$$
As $\tau_i\setminus\vartheta$ and $\Vv_i\subset\{(-1)^ih>0\}$, we conclude $\ol{U}_{\tau_i,\veps}\setminus\vartheta\subset\{(-1)^ih>0\}$ and so $\ol{U}_{\tau_1,\veps}\cap\ol{U}_{\tau_2,\veps}=\vartheta$, as required.
\end{proof}

\subsection{Proof of Theorem \ref{se0} in the general case}
We proceed by induction on the dimension of the semialgebraic set $\eta(M)$. 

\noindent{\bf Step 1.} {\em First step and formulation of induction hypothesis.} The first step of induction $\dim(\eta(M))=0$ follows from the case when $\eta(M)$ is closed. Assume that the result is true if $\dim(\eta(M))\leq d-1$ and let us check that it holds if $\dim(\eta(M))=d$. 

\noindent{\bf Step 2.} {\em Reduction of the problem to the piecewise linear case.} Let $(K,\Phi)$ be a semialgebraic triangulation of $\cl(M)$ compatible with $M$, $\cl(M)\setminus M$ and $\eta(M)$ where $K$ is a finite simplicial complex and $\Phi:|K|\to\cl(M)$ is a semialgebraic homeomorphism. To simplify notations, we identify $\cl(M)$ with $|K|$, $M$ with $\Phi^{-1}(M)$ and $\eta(M)$ with $\Phi^{-1}(\eta(M))$. Let $\tau_1,\ldots,\tau_s$ be the simplices of $K$ of dimension $d$ such that $\tau_i^0$ is contained in $\eta(M)$; observe that $\eta(M)':=\eta(M)\setminus\bigcup_{i=1}^s\tau_i^0$ has dimension $\leq d-1$. 

\noindent{\bf Step 3.} {\em Local construction of the appropriate embedding.} 
For each $i=1,\ldots,s$ denote the affine subspace generated by $\tau_i$ with $L_i$ and the orthogonal projection onto $L_i$ with $\pi_i:R^n\to L_i$. We use freely that $\dist(x,\tau_i)=\|x-\pi_i(x)\|$ if $x\in\pi_i^{-1}(\tau_i)$. To simplify the forthcoming work, we prove the following identity
\begin{equation}\label{clue0}
\dist(x,\partial\tau_i)^2=\|x-\pi_i(x)\|^2+\dist(\pi_i(x),\partial\tau_i)^2\qquad\forall x\in\pi_i^{-1}(\tau_i).
\end{equation}

Indeed, let $y,z\in\partial\tau_i\subset L_i$ such that $\dist(x,\partial\tau_i)=\|x-y\|$ and $\dist(\pi_i(x),\partial\tau_i)=\|\pi_i(x)-z\|$. Then, using Pythagoras Theorem,
\begin{multline*}
\dist(x,\partial\tau_i)^2=\|x-\pi_i(x)\|^2+\|\pi_i(x)-y\|^2\geq\|x-\pi_i(x)\|^2+\dist(\pi_i(x),\partial\tau_i)^2\\
=\|x-\pi_i(x)\|^2+\|\pi_i(x)-z\|^2=\|x-z\|^2\geq\dist(x,\partial\tau_i)^2
\end{multline*}
and equality \eqref{clue0} holds. 

We deduce furthermore that for each $0<\veps<1$ and $x\in\pi_i^{-1}(\tau_i)$ the equivalence
\begin{equation}\label{clue00}
\dist(x,\tau_i)\leq\veps\dist(x,\partial\tau_i)\ \iff\ \|x-\pi_i(x)\|\leq\veps^*\dist(\pi_i(x),\partial\tau_i)
\end{equation}
holds where $\veps^*:=\frac{\veps}{\sqrt{1-\veps^2}}$. Rewrite the semialgebraic neighborhoods of $\tau_i^0$ provided in Lemma \ref{simplex}, by means of \eqref{clue00}, as follows 
\begin{align*}
&U_{i,\veps}:=U_{\tau_i,\veps}=\{x\in\pi_i^{-1}(\tau_i):\,\|x-\pi_i(x)\|<\veps^*\dist(\pi_i(x),\partial\tau_i)\},\\
&\ol{U}_{i,\veps}:=\ol{U}_{\tau_i,\veps}=\{x\in\pi_i^{-1}(\tau_i):\,\|x-\pi_i(x)\|\leq\veps^*\dist(\pi_i(x),\partial\tau_i)\}.
\end{align*}
By Corollary \ref{smallveps} and Lemma \ref{smallint} we may choose $\veps>0$ such that $\ol{U}_{i,\veps}\setminus\partial\tau_i$ only meets the simplices of $K$ that contain $\tau_i$ and $\ol{U}_{i,\veps}\cap\ol{U}_{j,\veps}=\tau_i\cap\tau_j$ if $i\neq j$. 

Now define $M_\delta:=M\setminus\bigcup_{i=1}^s(\ol{U}_{i,\delta}\setminus\partial\tau_i)$ for each $\delta>0$ and observe that $\bigcup_{i=1}^s(\ol{U}_{i,\delta}\setminus\partial\tau_i)$ is a (tubular) neighborhood of $\bigcup_{i=1}^s\tau_i^0$. Consider the semialgebraic maps
$$
\begin{array}{rll}
g_{i}:V_i:=\ol{U}_{i,\veps}\setminus\tau_i^0&\!\!\!\to\!\!\!&W_i:=\ol{U}_{i,\veps}\setminus(\ol{U}_{i,\frac{\veps}{2}}\setminus\partial\tau_i)\\
x&\!\!\!\mapsto\!\!\!&
\begin{cases}
\pi_i(x)+\frac{x-\pi_i(x)}{\|x-\pi_i(x)\|}(a_1\dist(\pi_i(x),\partial\tau_i)+a_2\|x-\pi_i(x)\|)&\text{if $x\not\in\partial\tau_i$,}\\
x&\text{if $x\in\partial\tau_i$,}
\end{cases}
\end{array}
$$
$$
\begin{array}{rll}
h_{i}:\ol{W}_i:=\ol{U}_{i,\veps}\setminus U_{i,\frac{\veps}{2}}&\!\!\!\to\!\!\!&\ol{V}_i:=\ol{U}_{i,\veps}\\
x&\!\!\!\mapsto\!\!\!&
\begin{cases}
\pi_i(x)+\frac{x-\pi_i(x)}{\|x-\pi_i(x)\|}(b_1\|x-\pi_i(x)\|+b_2\dist(\pi_i(x),\partial\tau_i))&\text{if $x\not\in\partial\tau_i$,}\\
x&\text{if $x\in\partial\tau_i$}
\end{cases}
\end{array}
$$
where $(a_1,a_2),(b_1,b_2)\in R^2$ are the respective solutions of the system of linear equations
$$
\left\{\begin{array}{l}
a_1=(\frac{\veps}{2})^*,\\[4pt]
a_1+\veps^*a_2=\veps^*
\end{array}\right|\leadsto
\left\{\begin{array}{l}
a_1=(\frac{\veps}{2})^*,\\[4pt]
a_2=\frac{\veps^*-(\frac{\veps}{2})^*}{\veps^*}
\end{array}\right.
\ \text{and}\ 
\left\{\begin{array}{l}
(\frac{\veps}{2})^*b_1+b_2=0,\\[4pt]
\veps^*b_1+b_2=\veps^*,
\end{array}\right|\leadsto
\left\{\begin{array}{l}
b_1=\frac{\veps^*}{\veps^*-(\frac{\veps}{2})^*},\\[4pt]
b_2=-(\frac{\veps}{2})^*\frac{\veps^*}{\veps^*-(\frac{\veps}{2})^*}.
\end{array}\right.
$$
Let us motivate the formulas of $g_i$ and $h_i$. Consider the segment $\Ss:=\{\lambda y+(1-\lambda)\pi_i(y):\,\lambda\in(0,1]\}$ that connects a point $y\in\pi_i^{-1}(\tau_i)$ such that $\|y-\pi_i(y)\|=\veps^*\dist(\pi_i(y),\partial\tau_i)$ with the point $\pi_i(y)$. Define also the segment $\Ss':=\{\lambda y+(1-\lambda) \pi_i(y):\,\lambda\in(\mu,1]\}$ that connects $y$ with the point $z\in\Ss$ such that $\|z-\pi_i(z)\|=(\frac{\veps}{2})^*\dist(\pi_i(z),\partial\tau_i)$. A straightforward computation shows that $\mu=\frac{(\frac{\veps}{2})^*}{\veps^*}$. The semialgebraic map $g_i$ arises when one `linearly' transforms the segment $\Ss$ onto $\Ss'$. The semialgebraic map $h_i$ appears to perform reversely. 

The continuity of $g_i$ and $h_i$ at $\partial\tau_i$ requires further comments. We analyze as an example what happens with $h_i$, as the behavior of $g_i$ is analogous. Pick a point $y\in\partial\tau_i$ and let $x\in\ol{W_i}\setminus\partial\tau_i$ be close to $y$. Then
\begin{equation*}
\begin{split}
\|h_i(x)-h_i(y)\|&=\|h_i(x)-y\|\leq\|h_i(x)-\pi_i(x)\|+\|\pi_i(x)-y\|\\
&=|b_2\dist(\pi_i(x),\partial\tau_i)+b_1\|x-\pi_i(x)\||+\|\pi_i(x)-y\|\\
&\leq (|b_2|+1)\|\pi_i(x)-y\|+b_1(\|x-y\|+\|y-\pi_i(x)\|)\\
&\leq (2b_1+|b_2|+1)\|x-y\|,
\end{split}
\end{equation*}
so $h_i$ is continuous at $y$.

Using the identities $a_1+a_2b_2=0$ and $a_2b_1=1$, straightforward computations show
$$
g_{i}(V_i)=W_i,\ h_{i}(\ol{W}_i)=\ol{V}_i,\ g_{i}\circ h_{i}|_{W_i}=\id_{W_i}\quad\text{and}\quad(h_{i}|_{W_i})\circ g_{i}=\id_{V_i},
$$
so both $h_i|_{W_i}$ and $g_i$ are semialgebraic homeomorphisms. Moreover, by Lemma \ref{intsimplex} and taking into account that $g_i$ and $h_i$ are invariant over the segments (of suitable length) orthogonal to $\tau_i$, it holds $g_{i}(\sigma\cap V_i)=\sigma\cap W_i$ and $h_{i}(\sigma\cap\ol{W}_i)=\sigma\cap\ol{V}_i$ for each simplex $\sigma$ such that $\tau_i\subset\sigma$. 

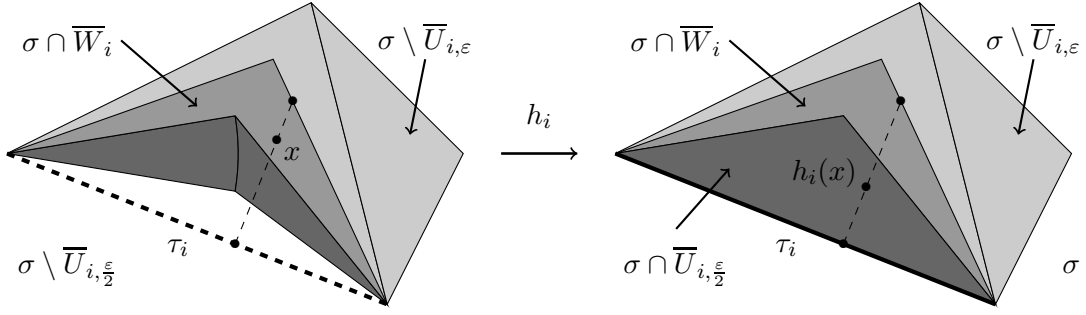
\begin{figure}[ht]
\centering
\begin{tikzpicture}
\draw[fill=black!20!white] (8,2) -- (13,0) -- (12,4) -- (8,2);
\draw[fill=black!20!white] (12,4) -- (13,0) -- (14,2) -- (12,4);
\draw[fill=black!40] (8,2) -- (11.5,3.25) -- (13,0) -- (11,2.5) -- (8,2);
\draw[fill=black!60] (8,2) -- (11,2.5) -- (13,0) -- (8,2);
\draw[ultra thick](8,2) -- (13,0);

\draw[fill=black!20!white] (0,2) -- (3,2.5) -- (5,0) -- (4,4) -- (0,2);
\draw[fill=black!20!white] (4,4) -- (5,0) -- (6,2) -- (4,4);
\draw[fill=black!60] (0,2) -- (3,1.5) arc(-9.462322208025803:9.462322208025803:3.04138126514911cm) -- (0,2);
\draw[fill=black!60] (5,0) -- (3,2.5) arc(9.462322208025803:-9.462322208025803:3.04138126514911cm) -- (5,0);
\draw[ultra thick,dashed](0,2) -- (5,0);
\draw[fill=black!40] (0,2) -- (3.5,3.25) -- (5,0) -- (3,2.5) -- (0,2);

\draw (14,0.5) node{$\sigma$};
\draw (10.25,0.75) node{$\tau_i$};
\draw (0.8,0.5) node{$\sigma\setminus\ol{U}_{i,\frac{\veps}{2}}$};
\draw (0.8,3.5) node{$\sigma\cap\ol{W}_i$};
\draw (8.8,3.5) node{$\sigma\cap\ol{W}_i$};
\draw (8.8,0.5) node{$\sigma\cap\ol{U}_{i,\frac{\veps}{2}}$};
\draw (5.5,3.5) node{$\sigma\setminus\ol{U}_{i,\veps}$};
\draw[->,thick] (5.5,3.25) -- (5.3,2.25);
\draw[->,thick] (1.5,3.5) -- (2.5,2.625);
\draw[->,thick] (9.5,3.5) -- (10.5,2.625);
\draw[->,thick] (8.8,1) -- (9.5,1.8); 
\draw[->,thick] (13.5,3.25) -- (13.3,2.25);
\draw (13.5,3.5) node{$\sigma\setminus\ol{U}_{i,\veps}$};
\draw (2.25,0.75) node{$\tau_i$};
\draw[->,thick] (6.5,2) -- (7.5,2);
\draw[thick] (6.5,2) -- (7.5,2);
\draw(7,2.5) node{$h_{i}$};;

\draw[dashed](3,0.8) -- (3.75714285714286,2.69285714285714);
\draw[dashed](11,0.8) -- (11.75714285714286,2.69285714285714);

\draw (3,0.8) node{\scriptsize$\displaystyle\bullet$};
\draw (3.75714285714286,2.69285714285714) node{\scriptsize$\displaystyle\bullet$};

\draw (3.55,2.175) node{\scriptsize$\displaystyle\bullet$};
\draw (3.75,2) node{$x$};

\draw (11.3,1.55) node{\scriptsize$\displaystyle\bullet$};
\draw (10.75,1.75) node{$h_i(x)$};

\draw (11,0.8) node{\scriptsize$\displaystyle\bullet$};
\draw (11.75714285714286,2.69285714285714) node{\scriptsize$\displaystyle\bullet$};

\end{tikzpicture}
\caption{Action of the map $h_i:\sigma\setminus U_{\tau_i,\frac{\veps}{2}}\to\sigma$.}
\end{figure}

\noindent{\bf Step 4.} {\em Global construction of the appropriate embedding outside a semialgebraic set of low dimension}. As one can check straightforwardly, the previous semialgebraic maps $g_i$ and $h_i$ glue together to provide the following semialgebraic maps
$$
\begin{array}{rll}
g:M&\to&N:=M_{\frac{\veps}{2}},\\ 
x&\mapsto&
\begin{cases}
x&\text{if $x\in M_{\veps}$,}\\[4pt]
g_{i}(x)&\text{if $x\in M\cap V_i$, $i=1,\ldots,s$,}
\end{cases}
\end{array}
$$
$$
\begin{array}{rll}
h:\cl(N)&\to&\cl(M),\\ 
x&\mapsto&
\begin{cases}
x&\text{if $x\in\cl(M)_\veps$,}\\[4pt]
h_{i}(x)&\text{if $x\in\cl(M)\cap\ol{W}_i$, $i=1,\ldots,s$.}
\end{cases}
\end{array}
$$
Again, routinary computations show $g\circ h|_{N}=\id_{N}$ and $h|_{N}\circ g=\id_{M}$; hence, $g$ and $h|_N$ are both semialgebraic homeomorphisms.

\begin{figure}[ht]
\centering
\begin{tikzpicture}

\draw[fill=black!20!white] (0,6) -- (4,6) -- (4,8) -- (0,8) -- (0,6);
\draw[fill=black!20!white] (0,8) -- (4,8) -- (4,10) -- (0,10) -- (0,8);
\draw[fill=black!20!white] (4,6) -- (6,6) -- (6,8) -- (4,8) -- (4,6);
\draw[fill=black!20!white] (4,8) -- (6,8) -- (6,10) -- (4,10) -- (4,8);

\draw (4,8) node{$\displaystyle\bullet$};








\draw[fill=black!20!white] (8,8) -- (10,8) -- (8,10) -- (8,8);
\draw[fill=black!20!white] (10,10) -- (10,8) -- (8,10) -- (10,10);
\draw[fill=black!20!white] (10,8) -- (12,8) -- (12,10) -- (10,8);
\draw[fill=black!20!white] (10,8) -- (10,10) -- (12,10) -- (10,8);
\draw[fill=black!20!white] (12,10) -- (14,10) -- (14,8) -- (12,10);
\draw[fill=black!20!white] (12,8) -- (12,10) -- (14,8) -- (12,8);

\draw[fill=black!20!white] (8,8) -- (10,8) -- (8,6) -- (8,8);
\draw[fill=black!20!white] (10,6) -- (10,8) -- (8,6) -- (10,6);
\draw[fill=black!20!white] (10,8) -- (12,8) -- (12,6) -- (10,8);
\draw[fill=black!20!white] (10,8) -- (10,6) -- (12,6) -- (10,8);
\draw[fill=black!20!white] (12,6) -- (14,6) -- (14,8) -- (12,6);
\draw[fill=black!20!white] (12,8) -- (12,6) -- (14,8) -- (12,8);

\draw (12,8) node{$\displaystyle\bullet$};


\draw[fill=black!20!white] (8,2) -- (9,2.5) -- (10,2) -- (8,4) -- (8,2);
\draw[fill=black!20!white] (10,4) -- (10,2) -- (8,4) -- (10,4);
\draw[fill=black!20!white] (10,2) -- (11,2.5) -- (12,2) -- (11.5,3) -- (12,4) -- (10,2);
\draw[fill=black!20!white] (10,2) -- (10,4) -- (12,4) -- (10,2);
\draw[fill=black!20!white] (12,4) -- (14,4) -- (14,2) -- (12,4);
\draw[fill=black!20!white] (12,2) -- (12.5,3) -- (12,4) -- (14,2) -- (13,2.5) -- (12,2);

\draw[fill=black!20!white] (8,2) -- (9,1.5) -- (10,2) -- (8,0) -- (8,2);
\draw[fill=black!20!white] (10,0) -- (10,2) -- (8,0) -- (10,0);
\draw[fill=black!20!white] (10,2) -- (11,1.5) -- (12,2) -- (11.5,1) -- (12,0) -- (10,2);
\draw[fill=black!20!white] (10,2) -- (10,0) -- (12,0) -- (10,2);
\draw[fill=black!20!white] (12,0) -- (14,0) -- (14,2) -- (12,0);
\draw[fill=black!20!white] (12,2) -- (12.5,1) -- (12,0) -- (14,2) -- (13,1.5) -- (12,2);

\draw (12,2) node{$\displaystyle\bullet$};


\draw[fill=black!20!white] (0,2.5) arc (90:26.56505117707799:0.5cm) -- (1,2.5) -- (1.55278640450004,2.22360679774998) arc (153.43494882292201:135:0.5cm) -- (0,4) -- (0,2.5);
\draw[fill=black!20!white] (2,4) -- (2,2.5) arc (90:135:0.5cm) -- (0,4) -- (2,4);
\draw[fill=black!20!white] (2.35355339059328,2.35355339059328) arc (45:26.56505117707799:0.5cm) -- (3,2.5) -- (4,2) -- (3.5,3) -- (3.77639320225002,3.55278640450004) arc (243.43494882292201:225:0.5cm) -- (2.35355339059328,2.35355339059328);
\draw[fill=black!20!white] (2.35355339059328,2.35355339059328) arc (45:90:0.5cm) -- (2,4) -- (3.5,4) arc (-180:-135:0.5cm) -- (2.35355339059328,2.35355339059328);
\draw[fill=black!20!white] (6,4) -- (6,2.5) arc (90:135:0.5cm) -- (4.35355339059328,3.64644660940672) arc (315:360:0.5cm) -- (6,4);
\draw[fill=black!20!white] (4,2) -- (4.5,3) -- (4.22360679774998,3.55278640450004) arc (296.56505117707799:315:0.5cm) -- (5.64644660940673,2.35355339059327) arc (135:153.43494882292201:0.5cm)-- (5,2.5) -- (4,2);

\draw[fill=black!20!white] (0,1.5) arc (-90:-26.56505117707799:0.5cm) -- (1,1.5) -- (1.55278640450004,1.77639320225002) arc (206.56505117707799:225:0.5cm) -- (0,0) -- (0,1.5);
\draw[fill=black!20!white] (0,0) -- (2,0) -- (2,1.5) arc (270:225:0.5cm) -- (0,0);
\draw[fill=black!20!white] (3,1.5) -- (4,2) -- (3.5,1) -- (3.77639320225002,0.44721359549996) arc (116.56505117707799:135:0.5cm) -- (2.35355339059327,1.64644660940673) arc (315:333.43494882292201:0.5cm) -- (3,1.5);
\draw[fill=black!20!white] (2,1.5) arc (270:315:0.5cm) -- (3.64644660940672,0.35355339059328) arc (135:180:0.5cm) -- (2,0) -- (2,1.5);
\draw[fill=black!20!white] (6,0) -- (6,1.5) arc (270:225:0.5cm) -- (4.35355339059328,0.35355339059328) arc (45:0:0.5cm) -- (6,0);
\draw[fill=black!20!white] (4,2) -- (5,1.5) -- (5.55278640450004,1.77639320225002) arc (206.56505117707799:225:0.5cm) -- (4.35355339059327,0.35355339059327) arc (45:63.43494882292202:0.5cm) -- (4.5,1) -- (4,2); 

\draw (4,2) node{$\displaystyle\bullet$};

\draw (3,5.5) node{\scriptsize$\displaystyle M:=\big(([-2,1]\times[-1,1])\setminus\{xy=0\}\big)\cup\{(0,0)\}$};
\draw (3,5) node{\scriptsize$\displaystyle \eta(M):=(M\cap\{xy=0\})\setminus\{(0,0)\}$};
\draw (7,8.5) node{\small$\Phi$};
\draw (7,7.5) node{\scriptsize semialgebraic};
\draw (7,7) node{\scriptsize triangulation};
\draw (7,6.25) node{\small{Stage 0}};
\draw (11.5,5) node{\small$h^{(1)}$};
\draw (10,5) node{\small{Stage 1}};
\draw (7,2.5) node{\small$h^{(2)}$};
\draw (7,1.5) node{\small{Stage 2}};

\draw[->,thick] (7.5,8) -- (6.5,8);
\draw[->,thick] (11,4.5) -- (11,5.5);
\draw[->,thick] (6.5,2) -- (7.5,2);

\draw[dashed] (0,2) -- (6,2);
\draw[dashed] (4,0) -- (4,4);

\draw[dashed] (8,2) -- (14,2);
\draw[dashed] (12,0) -- (12,4);
\end{tikzpicture}
\caption{Walkthrough of the proof.}
\end{figure}
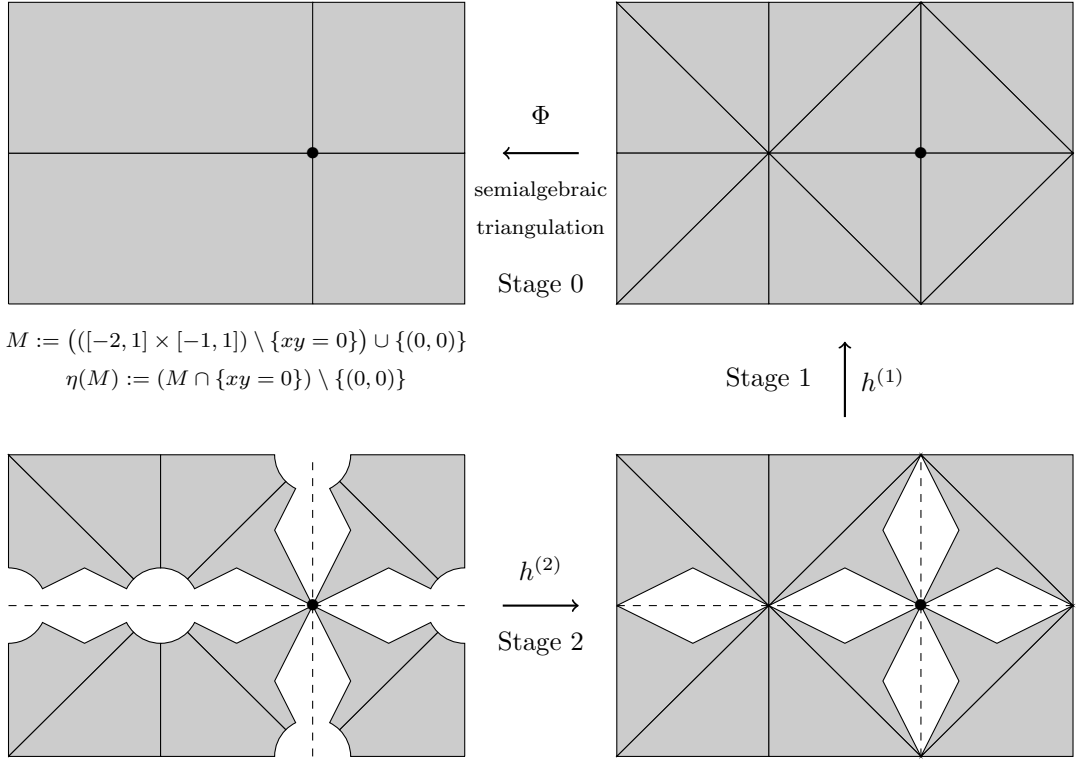 

\noindent{\bf Step 5.} {\em The semialgebraic set $N$ is appropriately embedded outside a semialgebraic set of low dimension.} Let us check now that for each $q\in\cl(N)\setminus\eta(M)'$ the germ $N_q$ is semialgebraically connected and either $\cl(N)_q=N_q$ or $\dim(\cl(N)_q\setminus N_q)=\dim(N_q)-1$; once this is done and taking into account $\dim(\eta(M)')=d-1$, the statement holds by induction hypothesis.

Indeed, let $q\in\cl(N)\setminus \eta(M)'$. Observe that if $q\not\in\bigcup_{i=1}^s\ol{U}_{i,\frac{\veps}{2}}$, then
$$
M_q=\Big(M\setminus\bigcup_{i=1}^s\ol{U}_{i,\frac{\veps}{2}}\Big)_q=N_q
$$
and $q\in\cl(M)\setminus\eta(M)$. Thus, $N_q$ is semialgebraically connected and either $\cl(N)_q=N_q$ or $\dim(\cl(N)_q\setminus N_q)=\dim(N_q)-1$. Therefore we may assume $q\in\ol{U}_{i,\frac{\veps}{2}}\setminus\partial\tau_i$ and $\dist(q,\tau_i)=\frac{\veps}{2}\dist(q,\partial\tau_i)$ for some $i=1,\ldots,s$. Let $\sigma\in K$ be the simplex such that $q\in\sigma^0$. We claim: $\tau_i\subset\sigma$. 

Indeed, $h_i(q)=\pi_i(q)\in\tau_i^0$ because by \eqref{clue0}, 
$$
0<\dist(q,\tau_i)=\|q-\pi_i(q)\|=(\tfrac{\veps}{2})^*\dist(\pi_i(q),\partial\tau_i).
$$
Thus, since $h_{i}(\sigma\cap\ol{W}_i)=\sigma\cap\ol{V}_i$, we deduce $h_i(q)\in\sigma\cap\tau_i^0\neq\varnothing$ and so $\tau_i\subset\sigma$.

Let $\sigma_1,\ldots,\sigma_r$ be the collection of all simplices of $K$ such that $\sigma\subset\sigma_j$ and $\sigma_j^0\subset M$ and let us assume that $\sigma_1=\sigma$. The union $\bigcup_{j=1}^r\sigma_j^0$ is an open neighborhood of $q$ in $M$. Denote $S_j:=\sigma_j^0\setminus\ol{U}_{i,\frac{\veps}{2}}$ and $\widehat{S}_j:=\sigma_j\setminus\ol{U}_{i,\frac{\veps}{2}}$ and let us check that: \em The germ
\begin{equation}\label{nq}
N_q=M_q\setminus\ol{U}_{i,\frac{\veps}{2},q}=\Big(\bigcup_{k=1}^r\sigma_{k,q}^0\Big)\setminus\ol{U}_{i,\frac{\veps}{2},q}=\bigcup_{k=1}^r\sigma_{k,q}^0\setminus\ol{U}_{i,\frac{\veps}{2},q}=\bigcup_{k=1}^rS_{k,q}
\end{equation}
is semialgebraically connected and $\dim(\cl(N)_q\setminus N_q)=\dim(N_q)-1$\em.

By Lemma \ref{lcd}, the germ $S_{j,q}$ is semialgebraically connected and 
$$
\dim(\cl(\widehat{S}_{j,q})\setminus\widehat{S}_{j,q})=\dim(S_{j,q})-1
$$ 
for each $j=1,\ldots,r$. As $S_1\subset\widehat{S}_j$, we deduce that the germ $S_{1,q}\cup S_{j,q}$ is connected for $j=1,\ldots,r$ and so it is also $N_q=\bigcup_{k=1}^rS_{k,q}$.

Next, let us check $\cl(S_k)_q\setminus N_{q}\supset\cl(\widehat{S}_k)_q\setminus\widehat{S}_{k,q}$. It is enough to show $\cl(S_k)\setminus\bigcup_{j=1}^rS_j\supset\cl(\widehat{S}_k)\setminus\widehat{S}_k$. Indeed,
$$
\cl(S_k)\setminus\bigcup_{j=1}^rS_j=\cl(S_k)\setminus\bigcup_{j=1}^r((\sigma_j^0\cap\sigma_k)\setminus\ol{U}_{i,\frac{\veps}{2}})\supset\cl(S_k)\setminus(\sigma_k\setminus\ol{U}_{i,\frac{\veps}{2}})=\cl(\widehat{S}_k)\setminus\widehat{S}_k.
$$
Thus, $\cl(N)_q\setminus N_q\supset\bigcup_{k=1}^q\cl(\widehat{S}_k)_q\setminus\widehat{S}_{k,q}$ and using \cite[2.8.13]{bcr}, we conclude
\begin{multline*}
\dim(N_q)-1\geq\dim(\cl(N)_q\setminus N_q)=\max\{\dim(\cl(\widehat{S}_k)_q\setminus\widehat{S}_{k,q}):\,k=1,\ldots,r\}\\
=\max\{\dim(S_{k,q}):\,k=1,\ldots,r\}-1=\dim(N_q)-1,
\end{multline*}
as required.\qed

\section{Proofs of Lemma \ref{st1} and Theorem \ref{st3}}\label{s5}

In this section we conduct the proofs of Lemma \ref{st1} and Theorem \ref{st3} stated in the Introduction. 

\subsection{Proof of Lemma \ref{st1}}\label{1}
(i) Let $\varphi$ be a homomorphism whose core is ${\tt p}$. At least the homomorphism $\psi_p:={\rm ev}_{M_F,{\tt p}}\circ{\tt i}_{M,F}$ fits this condition. We develop the proof in several steps: 

\noindent{\bf Step 1.} \em Assume that $M$ is closed in $R^n$\em. Consider the evaluation $R$-homomorphism $\psi:R[\x]:=R[\x_1,\ldots,\x_n]\to F,\ Q\mapsto Q({\tt p})$. Since $F$ is a real closed field and ${\mathcal S}(R^n,R)$ is the real closure of $R[\x]$, there exists a unique homomorphism $\Psi:{\mathcal S}(R^n,R)\to F$ such that $\Psi|_{R[\x]}=\psi$; of course, $\Psi={\rm ev}_{F^n,{\tt p}}\circ{\tt i}_{R^n,F}$. Since $M$ is a closed semialgebraic subset of $R^n$, the $R$-homomorphism 
$$
\theta:{\mathcal S}(R^n,R)\to{\mathcal S}(M,R),\ f\mapsto f|_{M}
$$ 
is by \cite{dk} surjective and the uniqueness of $\Psi$ guarantees $\Psi=\varphi\circ\theta$. Thus, if $f\in{\mathcal S}(M,R)$, we pick a semialgebraic extension $\widehat{f}\in{\mathcal S}(R^n,R)$ and deduce
$$
\varphi(f)=\Psi(\widehat{f})=\widehat{f}_F({\tt p})=f_F({\tt p}).
$$
The last equality holds because ${\tt p}\in M_F$. We conclude $\varphi=\psi_p$.

\noindent{\bf Step 2.} \em The general case\em. Let $(X,{\tt j}_X)$ be a semialgebraic pseudo-compact\-ifica\-tion of $M$. We know by Step 1 that $\varphi\circ{\tt j}_X^*={\rm ev}_{X_F,{\tt j}_{X,F}({\tt p})}\circ{\tt i}_{X,F}$ and so $\varphi\circ{\tt j}_X^*(h)=h_F({\tt j}_{X,F}({\tt p}))=(h\circ{\tt j}_X)_F({\tt p})$ for all $h\in{\mathcal S}(X,R)$. Now since ${\mathcal S}^*(M,R)=\displaystyle\lim_{\longrightarrow}{\mathcal S}(X,R)$ (see \ref{cita}), we deduce $\varphi(h)=h_F({\tt p})$ for all $h\in{\mathcal S}^*(M,R)$. Next, if $g\in{\mathcal S}(M,R)$, we write $g=\frac{g}{1+|g|}/\frac{1}{1+|g|}$ where $\frac{g}{1+|g|},\frac{1}{1+|g|}\in{\mathcal S}^*(M,R)$ and $\frac{1}{1+|g|}$ is a unit in ${\mathcal S}(M,R)$. Hence,
$$
\varphi(g)=\frac{\varphi(\frac{g}{1+|g|})}{\varphi(\frac{1}{1+|g|})}=\frac{\frac{g_F({\tt p})}{1+|g_F({\tt p})|}}{\frac{1}{1+|g_F({\tt p})|}}=g_F({\tt p}),
$$
that is, $\varphi={\rm ev}_{M_F,{\tt p}}\circ{\tt i}_{M,F}$.

\noindent{\bf Step 3.} Let us prove now: \em If ${\rm d}_{\cl(M)}({\tt p})=\dim(M_{F,{\tt p}})$, then ${\tt p}\in M_{F}$.\em

Suppose ${\tt p}\in\cl(M_F)\setminus M_{\lc,F}=\cl(\cl(M_F)\setminus M_F)$. As usual, we assume that $M$ is bounded and let $(K,\Phi)$ be a semialgebraic triangulation of $X:=\cl(M)$ compatible with $M$. Consider the open semialgebraic subset $U:=\bigcup_{{\tt p}\in\Phi(\sigma)_F}\Phi(\sigma^0)$ of $X$; clearly, ${\tt p}\in U_F$. Notice that $\dim(M\cap U)=\dim(M_{F,{\tt p}})$. We have 
$$
{\tt p}\in\cl(X_F\setminus M_F)\cap U_F=\cl((X_F\cap U_F)\setminus(M_F\cap U_F))\cap U_F\subset\cl(\cl(M\cap U)\setminus(M\cap U))_F;
$$
hence, by \cite[2.8.13]{bcr} we conclude
\begin{multline*}
\dim(M_{F,{\tt p}})={\tt d}_X({\tt p})\leq\dim(\cl(\cl(M\cap U)\setminus(M\cap U)))\\
=\dim(\cl(M\cap U)\setminus(M\cap U))<\dim(M\cap U)=\dim(M_{F,{\tt p}}), 
\end{multline*}
which is a contradiction. Thus, ${\tt p}\in M_{\lc,F}\subset M_{F}$, as required.

\vspace{2mm}
(ii) Let $\varphi$ be a homomorphism whose core is ${\tt p}$ and fix $f\in{\mathcal S}(M,R)$. Since $M$ is appropriately embedded, there exists by Theorem \ref{ext2} an open semialgebraic neighborhood $V$ of $M$ in $X:=\cl(M)$, a semialgebraic set $Y\subset V\setminus M$ such that $\dim(\cl(Y)_q)\leq\dim(M_q)-2$ for all $q\in X$ and a semialgebraic extension $\widehat{f}\in{\mathcal S}(V\setminus Y,R)$. We claim: ${\tt p}\in(V\setminus\cl(Y))_F$.

Indeed, since $V$ is locally closed and ${\tt p}$ is adjacent to $M$, we have ${\tt p}\in V_F$. Let us check next that ${\tt p}\not\in\cl(Y)_F$. Let $(K,\Phi)$ be a semialgebraic triangulation of $X$ compatible with $M$ and $\cl(Y)$ where $K$ is a finite simplicial complex and $\Phi:|K|\to X$ a semialgebraic homeomorphism. For the sake of simplicity, we identify $X$ with $|K|$ and the involved objects $M$ and $\cl(Y)$ with their inverse images under $\Phi$. 

Assume by contradiction that ${\tt p}\in\cl(Y)_F$. As $(K,\Phi)$ is compatible with $\cl(Y)$, there exists $\tau\in K$ such that ${\tt p}\in \tau^0_F\subset\cl(Y)_F$. Define ${\mathfrak G}:=\{\sigma\in K:\,\tau\subset\sigma\}$ and consider the open semialgebraic neighborhood $U:={\rm St}(\tau^0)=\bigcup_{\sigma\in{\mathfrak G}}\sigma^0$ of $\tau^0$ in $X$. Let ${\mathfrak F}:=\{\sigma\in{\mathfrak G}:\,\sigma^0\subset\cl(Y)\}$ and $d:=\max\{\dim(\sigma):\,\sigma\in{\mathfrak F}\}$. Since $\dim(\cl(Y)_q)\leq\dim(M_q)-2$ for all $q\in X$, there exists a simplex $\varsigma\in{\mathfrak G}$ such that $d\leq\dim(\varsigma)-2$ and $\varsigma^0\subset M$. Thus, we have ${\tt p}\in\tau^0_F\subset\varsigma_F$ and 
\begin{equation}\label{sigmam}
\dim(\varsigma^0)=\dim(\varsigma^0_F)=\dim(\varsigma^0_{F,{\tt p}})\leq\dim(M_{F,{\tt p}})
\end{equation} 
because $\varsigma^0_F$ is pure dimensional. As ${\tt p}\in U_F\cap\cl(Y)_F=\bigcup_{\sigma\in{\mathfrak F}}\sigma^0_F$, we deduce using \eqref{sigmam}
\begin{multline*}
{\rm d}_X({\tt p})\leq\dim(\cl(U_F\cap\cl(Y)_F))=\dim(U_F\cap\cl(Y)_F)\\
=d\leq\dim(\varsigma^0_F)-2\leq\dim(M_{F,{\tt p}})-2<{\rm d}_X({\tt p}),
\end{multline*}
which is a contradiction. Therefore ${\tt p}\in V_F\setminus\cl(Y)_F=(V\setminus\cl(Y))_F$.

Consider now the closed semialgebraic set $(X\setminus V)\cup\cl(Y)$ and let $G\in{\mathcal S}(X,R)$ be a semialgebraic function such that $Z(G)=(X\setminus V)\cup\cl(Y)$; clearly $G_F({\tt p})\neq0$. Denote the natural inclusion induced by the restriction to $M$ with $\theta:{\mathcal S}(X,R)\hookrightarrow{\mathcal S}(M,R)$ and let $g=\theta(G)$. The semialgebraic functions $h_1:=\frac{f}{1+|f|}g$ and $h_2:=\frac{g}{1+|f|}$ can be extended by zero to respective semialgebraic functions $H_1,H_2\in{\mathcal S}(X,R)$, that is, $\theta(H_i)=h_i$. By (i) and using ${\tt p}\in X_F$, we have $\varphi\circ\theta={\rm ev}_{X_F,{\tt p}}\circ{\tt i}_{M,F}$; hence, $\varphi(h_i)=\varphi\circ\theta(H_i)=H_{i,F}({\tt p})$ and $\varphi(g)=\varphi\circ\theta(G)=G_F({\tt p})$. As ${\tt p}\in(V\setminus\cl(Y))_F$ and $f$ can be extended to $\widehat{f}\in{\mathcal S}(V\setminus Y,R)$, we deduce 
\begin{equation*}
\begin{split}
\frac{\varphi(f)}{1+|\varphi(f)|}\varphi(g)=\varphi(h_1)=H_{1,F}({\tt p})=\frac{\widehat{f}_F({\tt p})}{1+|\widehat{f}_F({\tt p})|}G_F({\tt p})=\frac{\widehat{f}_F({\tt p})}{1+|\widehat{f}_F({\tt p})|}\varphi(g),\\
\frac{1}{1+|\varphi(f)|}\varphi(g)=\varphi(h_2)=H_{2,F}({\tt p})=\frac{1}{1+|\widehat{f}_F({\tt p})|}G_F({\tt p})=\frac{1}{1+|\widehat{f}_F({\tt p})|}\varphi(g).
\end{split}
\end{equation*}
Since $\frac{1}{1+|\varphi(f)|}\varphi(g)\neq0$, we conclude $\varphi(f)=\widehat{f}_F({\tt p})$ after dividing the previous equalities.

We proved the uniqueness as well as the existence of the homomorphism $\varphi$.

\vspace{2mm}
(iii) Assume that $M$ is bounded and let $C\subset X:=\cl(M)$ be a closed semialgebraic set of dimension ${\rm d}_X({\tt p})$ such that ${\tt p}\in C_F$. Let $(K,\Phi)$ be a semialgebraic triangulation of $X$ compatible with $M$, $C$ and $X\setminus M$ where $K$ is a finite simplicial complex and $\Phi:|K|\to X$ is a semialgebraic homeomorphism. For the sake of simplicity, we identify $X$ with $|K|$ and the involved objects $M$, $\cl(Y)$ and $X\setminus M$ with their inverse images under $\Phi$. 

Let $\tau\in K$ be a simplex such that ${\tt p}\in\tau^0_F$ and let us check: \em There exists a simplex $\sigma\in K$ of dimension $\dim(\sigma)\geq\dim(\tau)+2$ such that $\tau\subset\sigma$ and $\sigma^0\subset M$\em. 

Indeed, define ${\mathfrak G}:=\{\sigma\in K:\,\tau\subset\sigma\}$ and consider the open semialgebraic neighborhood $U:={\rm St}(\tau^0)=\bigcup_{\sigma\in{\mathfrak G}}\sigma^0$ of $\tau^0$ in $X$. Let ${\mathfrak F}:=\{\sigma\in{\mathfrak G}:\,\sigma^0\subset C\}$ and notice the following: Since ${\tt p}\in\tau^0_F$ and $(K,\Phi)$ is compatible with $C$, we have 
$$
{\rm d}_X({\tt p})\leq\dim(\tau)=\dim(\tau^0)\leq\dim(C)={\rm d}_X({\tt p}),
$$
that is, ${\rm d}_X({\tt p})=\dim(\tau)$. Since ${\rm d}_X({\tt p})\leq\dim(M_{F,{\tt p}})-2$ and 
$$
\dim(M_{F,{\tt p}})=\max\{\dim(\sigma):\,\sigma\in{\mathfrak G}\ \text{and}\ \sigma^0\subset M\},
$$ 
there exists a simplex $\sigma\in{\mathfrak G}$ such that $\dim(\sigma)\geq{\rm d}_X({\tt p})+2=\dim(\tau)+2$ and $\sigma^0\subset M$. 

As ${\tt p}\not\in M_{F}$, we have $\tau^0 \cap M=\varnothing$. Let $b\in\sigma^0$ be the barycenter of $\sigma$ and $\epsilon$ a face of $\sigma$ of dimension $\dim(\sigma)-1$ such that $\tau\subsetneq\epsilon$. Let $\epsilon(b)$ be the convex hull of $\epsilon\cup\{b\}$, which is a simplex such that $\tau\subset\epsilon(b)$, $\dim(\sigma)=\dim(\epsilon(b))$ and $\epsilon(b)\setminus M\subset\epsilon$. Since $\dim(\epsilon(b))=\dim(\epsilon)+1\geq\dim(\tau)+2$, there exists a vertex $v$ of $\epsilon(b)$ such that $v\not\in\tau$. Let $\eta$ be the segment connecting $v$ and $b$, which is a face of $\epsilon(b)$ that does not meet $\tau$. Moreover, $\eta\setminus M\subset\{v\}$ and so the interior of any segment connecting a point of $\eta^0$ with a point of $\tau$ is contained in $M$.

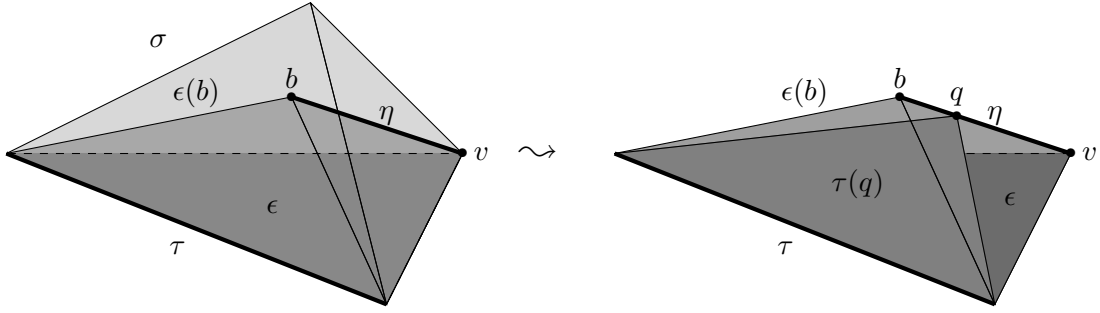
\begin{figure}[ht]
\centering
\begin{tikzpicture}

\draw[fill=black] (0,2) -- (5,0) -- (6,2) -- (0,2);
\draw[fill=black!50!white,opacity=0.75,draw=none] (0,2) -- (5,0) -- (3.75,2.75) -- (0,2);
\draw[fill=black!50!white,opacity=0.75,draw=none] (5,0) -- (6,2) -- (3.75,2.75) -- (5,0);

\draw[fill=black!30!white,opacity=0.5,draw=none] (0,2) -- (5,0) -- (4,4) -- (0,2);
\draw[fill=black!30!white,opacity=0.5,draw=none] (4,4) -- (5,0) -- (6,2) -- (4,4);

\draw (0,2) -- (5,0) -- (4,4) -- (0,2);
\draw (4,4) -- (5,0) -- (6,2) -- (4,4);
\draw (0,2) -- (5,0) -- (3.75,2.75) -- (0,2);
\draw (5,0) -- (6,2) -- (3.75,2.75) -- (5,0);
\draw[dashed] (0,2) -- (5,0) -- (6,2) -- (0,2);
\draw[ultra thick](6,2) -- (3.75,2.75);
\draw[ultra thick](0,2) -- (5,0);

\draw[fill=black!80!white] (8,2) -- (13,0) -- (14,2) -- (8,2);
\draw[fill=black!50!white] (8,2) -- (13,0) -- (12.5,2.5) -- (8,2);
\draw[fill=black!50!white,opacity=0.75,draw=none] (8,2) -- (13,0) -- (11.75,2.75) -- (8,2);
\draw[fill=black!50!white,opacity=0.75,draw=none] (13,0) -- (14,2) -- (11.75,2.75) -- (13,0);

\draw (8,2) -- (13,0) -- (11.75,2.75) -- (8,2);
\draw (13,0) -- (14,2) -- (11.75,2.75) -- (13,0);
\draw(8,2) -- (13,0) -- (12.5,2.5) -- (8,2);
\draw[dashed] (14,2) -- (12.6,2);
\draw[ultra thick](14,2) -- (11.75,2.75);
\draw[ultra thick](8,2) -- (13,0);

\draw (2,3.5) node{$\sigma$};
\draw (2.5,2.8) node{$\epsilon(b)$};
\draw (3.5,1.3) node{$\epsilon$};
\draw (2.25,0.75) node{$\tau$};
\draw (3.75,3) node{$b$};
\draw (3.75,2.75) node{\scriptsize$\displaystyle\bullet$};
\draw (6.25,2) node{$v$};
\draw (6,2) node{\scriptsize$\displaystyle\bullet$};
\draw (5,2.5) node{$\eta$};

\draw (10.5,2.8) node{$\epsilon(b)$};
\draw (13.2,1.4) node{$\epsilon$};
\draw (11.2,1.6) node{$\tau(q)$};
\draw (10.25,0.75) node{$\tau$};
\draw (11.75,3) node{$b$};
\draw (11.75,2.75) node{\scriptsize$\displaystyle\bullet$};
\draw (14.25,2) node{$v$};
\draw (14,2) node{\scriptsize$\displaystyle\bullet$};
\draw (12.5,2.75) node{$q$};
\draw (12.5,2.5) node{\scriptsize$\displaystyle\bullet$};
\draw (13,2.5) node{$\eta$};

\draw (7,2) node{\Large$\displaystyle\leadsto$};

\end{tikzpicture}
\caption{Construction of $\epsilon(b)$ and $\tau(q)$ for $q\in\eta^0$.}
\end{figure}

For each $q\in\eta^0$ denote the convex hull of $\tau\cup\{q\}$ with $\tau(q)$, which is a simplex contained in $\epsilon(b)$ and consider the closed semialgebraic subset $T(q):=\tau(q)\cap M$ of $M$, which satisfies $\cl(T(q))=\tau(q)$ and $\tau(q)\setminus T(q)\subset\tau$. By \cite{dk} the homomorphism 
\begin{equation}\label{phiq}
\phi_q:{\mathcal S}(M,R)\to{\mathcal S}(T(q),R),\ f\mapsto f|_{T(q)}
\end{equation}
is surjective for each $q\in\eta^0$. Since ${\tt p}\not\in M_{F}$, we have $\tau^0 \cap M=\varnothing$ and so 
\begin{equation}\label{inclusions}
\tau^0\subset\tau(q)\setminus T(q)\subset\tau. 
\end{equation}
As $\dim(T(q))=\dim(\tau^0)+1$ and 
$$
{\rm d}_{M}({\tt p})\leq{\rm d}_{T(q)}({\tt p})\leq\dim(\tau^0)={\rm d}_{M}({\tt p}),
$$ 
we obtain ${\rm d}_{T(q)}({\tt p})=\dim(T(q))-1=\dim(T(q)_{F,{\tt p}})-1$ because $T(q)_F$ is pure dimensional. Moreover, observe that $T(q)$ is appropriately embedded because it is bounded, $T(q)_x$ is semialgebraically connected and by \eqref{inclusions} either $\tau(q)_x=T(q)_x$ or $\dim(\tau(q)_x\setminus T(q)_x)=\dim(\tau)=\dim(T(q)_x)-1$ for all $x\in\tau(q)$. Thus, by (ii) there exists a unique homomorphism 
\begin{equation}\label{varphiq}
\varphi_{q}:{\mathcal S}(T(q),R)\to F,\ h\mapsto\widehat{h}({\tt p})
\end{equation}
for each $q\in\eta^0$. In each case $\widehat{h}$ is the continuous extension of $h$ to a fitting semialgebraic set $T(q)\subset\widehat{T}(q)\subset\cl(T(q))$ that contains ${\tt p}$. In view of \eqref{phiq} and \eqref{varphiq}, the core of the homomorphism $\psi_q:=\varphi_q\circ\phi_q$ for each $q\in\eta^0$ is ${\tt p}$.

To finish let us prove the following: \em Given two different points $q_1,q_2\in\eta^0$, there exists a semialgebraic function $f\in{\mathcal S}(M,R)$ such that $\phi_{q_1}(f)\neq\phi_{q_2}(f)$\em. 

Indeed, as $T(q_1)\cap T(q_2)\subset\tau\cap M$, the closed semialgebraic subsets $C_i:=T(q_i)\setminus\tau$ of $M\setminus\tau$ are disjoint. Thus, there exists by \cite{dk} a bounded semialgebraic function $f_0\in{\mathcal S}(M\setminus\tau,R)$ such that $f_0|_{C_i}=i-1$. On the other hand, $\tau\setminus\tau^0$ is a closed semialgebraic set in $R^n$ and we choose a bounded semialgebraic function $g\in{\mathcal S}(R^n,R)$ such that $Z(G)=\tau\setminus\tau^0$; denote $g:=G|_{M\setminus\tau}$. Clearly, the semialgebraic function $f_0g\in{\mathcal S}(M\setminus\tau,R)$ can be extended by zero to a semialgebraic function $f\in{\mathcal S}(M,R)$. Observe that $f|_{T(q_1)}=G|_{T(q_1)}f_0|_{T(q_1)}=0$ and $f|_{T(q_2)}=G|_{T(q_2)}f_0|_{T(q_2)}=G|_{T(q_2)}$. Thus, $\phi_{q_1}(f)=0$ and $\phi_{q_2}(f)=G_F({\tt p})\neq0$ because $Z(G_F)=\tau_F\setminus\tau^0_F$ and ${\tt p}\in\tau^0_F$.
\qed

\subsection{Basics on real spectra} The proof of Theorem \ref{st3} requires some preliminary definitions and notations concerning real spectra that we summarize here. Denote with $A$ either ${\mathcal P}(M)$ or ${\mathcal S}(M,R)$ and with $\Sper(A)$ the \em real spectrum of $A$\em; we refer the reader to \cite[\S7]{bcr} for the definition, notations and main properties of the real spectrum of a unital commutative ring. The points of $\Sper(A)$ are called \em prime cones\em. The \em support $\gtP_{\alpha}:=\alpha\cap(-\alpha)$ of a prime cone $\alpha\in\Sper(A)$ \em is a prime ideal of $A$. As usual, given $f\in A$ and $\alpha\in\Sper(A)$, we write $f(\alpha)\geq 0$ if $f\in\alpha$ and $f(\alpha)>0$ if $f\in\alpha\setminus(-\alpha)$; otherwise, $f(\alpha)<0$. Consider the collection of sets $\{\alpha\in\Sper(A):\, g_1(\alpha)>0,\ldots,g_r(\alpha)>0\}$ where $g_1,\ldots,g_r\in A$. Such collection constitutes a basis of the spectral topology of $\Sper(A)$. Recall that if $A=\psd(M)$, the \em zero set \em of $\gtP_{\alpha}$ is
$$
Z(\gtP_{\alpha}):=\{x\in\ol{M}^{\zar}:\, h(x)=0\ \forall h\in\gtP_{\alpha}\}
$$
where $\ol{M}^{\zar}$ is the \em Zariski closure of $M$ in $R^n$\em. Moreover, if $N:=\bigcup_{i=1}^r\{g_{i1}>0,\ldots,g_{is}>0,f_i=0\}\subset M$ is a semialgebraic set with $g_{ij},f_i\in R[\x]$, then $\widetilde{N}$ is the set of all prime cones $\alpha\in\Sper(A)$ satisfying the existence of an index $1\leq i\leq r$ such that $g_{i1}(\alpha)>0,\ldots,g_{is}(\alpha)>0,f_i(\alpha)=0$; of course, $\widetilde{N}$ does not depend on the description of $N$. Now if $\alpha\in\Sper(A)$, we define
\begin{equation}\label{def}
\dim(\alpha):=\dim({\mathcal P}(M)/\gtP_\alpha)\quad\text{and}\quad\dim_{\alpha}(\widetilde{M}):=\sup\{\dim(\beta):\,\beta\subset\alpha\ \text{and}\ \beta\in\widetilde{M}\}.
\end{equation}
The Zariski spectrum $\Spec(A)$ of $A$ is the collection of all prime ideals of $A$ and its Zariski topology has as a basis the collection of sets ${\mathcal D}(g):=\{\gtp\in\Spec(A):\,g\not\in\gtp\}$ where $g\in A$. If $A={\mathcal S}(M,R)$, it is well-known that $\Sper({\mathcal S}(M,R))$ is homeomorphic to $\Spec({\mathcal S}(M,R))$ via the support map: $\alpha\mapsto\gtP_\alpha$ (see for instance \cite{s2}).

\subsection{Proof of Theorem \ref{st3}}\label{2}
We divide the proof into several steps.

\noindent{\bf Step 1.} {\em Construction of the semialgebraic set $N_0$ in the statement}. Let $\gtp_0:=\ker\varphi$ be the kernel of $\varphi$ and $(X\subset R^n,{\tt j})$ a brimming semialgebraic pseudo-compactification of $M$ for $\gtp_0$. To simplify notations, we identify $M$ with ${\tt j}(M)$. The challenge of this proof is to find a chain of prime $z$-ideals $\gtp_d\subset\cdots\subset\gtp_1\subset\gtp_0$ of maximal length in ${\mathcal S}(M,R)$ such that ${\tt d}_M(\gtp_i)=\tr\deg_R(\qf({\mathcal S}(M,R)/\gtp_0))+i$ for $i=1,\ldots,d$. The construction of such a chain is quite cumbersome and involves a strong use of real spectra.

\noindent{\bf S1.a.} {\em Initial preparation}. Let $Z_0$ be the Zariski closure of $X$ and $\phi:\psd(Z_0)\hookrightarrow{\mathcal S}(X,R)$ the inclusion homomorphism. Consider the commutative diagram
$$
\xymatrix{
\Sper({\mathcal S}(X,R))\ar[rr]^{\Sper(\phi)}\ar[d]_{\supp}^{\cong}&&\widetilde{X}\subset\Sper({\mathcal P}(Z_0))\ar[d]^{\supp}\\
\Spec({\mathcal S}(X,R))\ar[rr]^{\Spec(\phi)}&&\Spec({\mathcal P}(Z_0))
}
$$
where the rows and columns are continuous maps (with respect to the respective spectral and Zariski topologies), $\Sper(\phi)(\alpha)=\phi^{-1}(\alpha)$ for each $\alpha\in\Sper({\mathcal S}(X,R))$ and $\Spec(\phi)(\gtq)=\phi^{-1}(\gtq)$ if $\gtq\in\Spec({\mathcal S}(X,R))$ (see \cite[7.1.7-8]{bcr}). As it is well-known, the map 
$$
\mu:=\Sper(\phi)\circ\supp^{-1}:\Spec({\mathcal S}(X,R))\to\widetilde{X}
$$ 
is a homeomorphism (see \cite[\S3]{cc}) and for every chain of prime cones $\widetilde{X}\ni\alpha_0\subsetneq\cdots\subsetneq\alpha_r$ in ${\mathcal P}(Z_0)$, $\mu^{-1}(\alpha_0)\subsetneq\cdots\subsetneq\mu^{-1}(\alpha_r)$ is a chain of prime ideals in ${\mathcal S}(X,R)$. Moreover, for each $\alpha\in\widetilde{X}$ it holds $\gtP_\alpha=\mu^{-1}(\alpha)\cap\psd(Z_0)$ and 
\begin{equation}\label{pol}
\dgt_M(\mu^{-1}(\alpha))=\dim(Z(\gtP_{\alpha}))
\end{equation}
because by \cite[2.6.6]{bcr} we have $Z(\gtP_{\alpha})=\bigcap_{f\in\mu^{-1}(\alpha)}\ol{Z_X(f)}^{\zar}$.

We construct now with the aid of $\mu$ a chain of prime ideals of maximal length contained in $\gtp_0$ such that the image of each term of the chain under $\mu$ belongs to $\widetilde{M}_{\lc}$. To that end, let $\beta_0:=\mu(\gtp_0\cap{\mathcal S}(X,R))\in\widetilde{X}$. As $M_{\lc}$ is dense in $X$, so is $\widetilde{M}_{\lc}$ in $\widetilde{X}$ (by Artin-Lang's Theorem); hence, $\beta_0$ is adherent to the constructible set $\widetilde{M}_{\lc}$. By \cite[Thm. I]{rz} it holds 
\begin{equation}\label{d}
d:=\dim_{\beta_0}(\widetilde{M}_{\lc})-\dim(\beta_0)\geq0
\end{equation} 
(see \eqref{def} for the definitions of the previous dimensions) and $\widetilde{M}_{\lc}$ contains $d$ points $\beta_1,\ldots,\beta_d$ such that $\beta_d\subsetneq\cdots\subsetneq\beta_1\subsetneq\beta_0$.

\noindent{\bf S1.b.} {\em Reduction to the case $d\geq1$}. If $d=0$, then $\beta_0\in\widetilde{M}_{\lc}$. We have the following diagram
$$
\xymatrix{
\psd(Z_0)/\gtP_{\beta_0}\ar@{^(->}[r]&{\mathcal S}(X,R)/(\gtp_0\cap{\mathcal S}(X,R))\ar@{^(->}[r]&{\mathcal S}(M,R)/\gtp_0\ar@{^(->}[r]^(.65){\ol{\varphi}}&F\\
\psd(Z_0)\ar@{->>}[u]\ar@{^(->}[rr]&&{\mathcal S}(M,R)\ar@{->>}[u]\ar[ur]^{\varphi}
}
$$
where the arrows in the first row preserve orderings. In particular, the prime cone $\beta_0$ is the inverse image of the cone $F^2$. As $\beta_0\in\widetilde{M}_{\lc}$, there exist $g_1,\ldots,g_r,f\in R[\x]:=R[\x_1,\ldots,\x_n]$ such that $\{g_1>0,\ldots,g_r>0,f=0\}\subset\widetilde{M}_{\lc}$ and $g_1(\beta_0)>0,\ldots,g_f(\beta_0)>0,f(\beta_0)=0$. Thus, $\varphi(g_1)>0,\ldots,\varphi(g_r)>0,\varphi(f)=0$; hence, the core ${\tt p}:=(\varphi(\pi_1),\ldots,\varphi(\pi_m))$ of $\varphi$ satisfies $g_1({\tt p})>0,\ldots,g_r({\tt p})>0,f({\tt p})=0$ and we conclude ${\tt p}\in M_{\lc,F}\subset M_F$. By Lemma \ref{st1}, $\varphi=\psi_{\tt p}:{\mathcal S}(M,R)\to F$ is the unique $R$-homomorphism from ${\mathcal S}(M,R)$ to $F$ whose core is ${\tt p}$. Thus, we assume in the following $d\geq1$.

\noindent{\bf S1.c.} {\em Construction of the chain of prime $z$-ideals $\gtp_d\subsetneq\cdots\subsetneq\gtp_1\subsetneq\gtp_0$ in ${\mathcal S}(M,R)$ such that ${\tt d}_M(\gtp_i)=\tr\deg_R(\qf({\mathcal S}(M,R)/\gtp_0))+i$ for $i=1,\ldots,d$}. Consider the chain of prime ideals
$$
\mu^{-1}(\beta_d)\subsetneq\cdots\subsetneq\mu^{-1}(\beta_1)\subsetneq\mu^{-1}(\beta_0)=\gtp_0\cap{\mathcal S}(X,R)
$$
of ${\mathcal S}(X,R)$. We claim:\setcounter{substep}{0} 

\begin{substeps}{2}
\em For $i=1,\ldots,d$ the ideal $\gtp_i:=\mu^{-1}(\beta_i){\mathcal S}(M_{\lc},R)\cap{\mathcal S}(M,R)$ is a prime $z$-ideal of ${\mathcal S}(M,R)$ contained in $\gtp_0$ such that $\gtp_i\cap{\mathcal S}(X,R)=\mu^{-1}(\beta_i)$\em.
\end{substeps}

The last part of the claim shows in particular that each $\gtp_i$ is a proper ideal. Fix an equation $G\in{\mathcal S}(X,R)$ of the closed semialgebraic set $X\setminus M_{\lc}$ and notice $G\not\in\mu^{-1}(\beta_i)$ for $i=1,\ldots,d$ because $\beta_i\in\widetilde{M}_{\lc}$; denote $g:=G|_M$. Let us begin by proving the following: 

\begin{substeps}{2}
\em The ideal $\gtp_i':=\mu^{-1}(\beta_i){\mathcal S}(M_{\lc},R)$ of ${\mathcal S}(M_{\lc},R)$ satisfies $\gtp_i'\cap{\mathcal S}(X,R)=\mu^{-1}(\beta_i)$ for $i=1,\ldots,d$\em.
\end{substeps}

Indeed, fix $H\in\gtp_i'\cap{\mathcal S}(X,R)$. There exist $H_j\in\mu^{-1}(\beta_i)$ and $g_j\in{\mathcal S}(M_{\lc},R)$ for $j=1,\ldots,r$ satisfying $H=H_1g_1+\ldots+H_rg_r$. The semialgebraic functions
$$
g':=\frac{g}{1+|g_1|+\cdots+|g_r|}\quad\text{and}\quad g_i':=\frac{g_ig}{1+|g_1|+\cdots+|g_r|}
$$
can be extended by zero to respective semialgebraic functions $G',G_i'\in{\mathcal S}(X,R)$, so
$$
G'H=H_1G'_1+\ldots+H_rG'_r\in\mu^{-1}(\beta_i).
$$
Since $Z(G')=X\setminus M_{\lc}$, we conclude $G'\not\in\mu^{-1}(\beta_i)$ and therefore $H\in\mu^{-1}(\beta_i)$. 

\begin{substeps}{2}
We show now: \em $\gtp_i'$ is a prime ideal of ${\mathcal S}(M_{\lc},R)$ for $i=1,\ldots,d$\em. Consequently, $\gtp_i=\gtp_i'\cap{\mathcal S}(M,R)$ is a prime ideal of ${\mathcal S}(M,R)$.
\end{substeps}

Indeed, let $a_1,a_2\in{\mathcal S}(M_{\lc},R)$ be such that $a_1a_2\in\gtp_i'$. Observe that $\frac{a_j}{1+|a_j|}g$ can be extended by zero to a semialgebraic function $A_i\in{\mathcal S}(X,R)$, so
$$
A_1A_2\in\mu^{-1}(\beta_i){\mathcal S}(M,R)\cap{\mathcal S}(X,R)=\mu^{-1}(\beta_i).
$$
Thus, we may assume $A_1\in\mu^{-1}(\beta_i)$ and therefore $a_1\frac{g}{1+|a_1|}\in\gtp_i'$. As $\frac{g}{1+|a_1|}$ is a unit in ${\mathcal S}(M_{\lc},R)$, we deduce $a_1\in\gtp_i'$; hence, $\gtp_i'$ is a prime ideal.

\begin{substeps}{2}
Now we prove: \em $\gtp_i$ is a prime $z$-ideal of ${\mathcal S}(M,R)$ for $i=1,\ldots,d$\em. \end{substeps}

Indeed, let $f_1\in\gtp_i=\gtp_i'\cap{\mathcal S}(M,R)$ and $f_2\in{\mathcal S}(M,R)$ be such that $Z_M(f_1)\subset Z_M(f_2)$; hence, $Z_{M_{\lc}}(f_1)\subset Z_{M_{\lc}}(f_2)$. As $\gtp_i'$ is a prime $z$-ideal (see \S\ref{sdtd}) because $M_{\lc}$ is locally closed, $f_2|_{M_{\lc}}\in\gtp_i'$ and so $f_2\in\gtp_i'\cap{\mathcal S}(M,R)=\gtp_i$.

\begin{substeps}{2}
Let us show next: \em $\gtp_i\subset\gtp_0$ for $i=1,\ldots,d$\em.
\end{substeps}

Indeed, let $h\in\gtp_i$, $H_j\in\mu^{-1}(\beta_i)$ and $g_j\in{\mathcal S}(M_{\lc})$ for $j=1,\ldots,r$ such that 
$$
h=H_1g_1+\cdots+H_rg_r. 
$$
As $M_{\lc}$ is open in $X$, the semialgebraic set $Z(H_1^2+\cdots+H_r^2)\setminus M_{\lc}$ is closed in $X$. Pick $B_1,B_2\in{\mathcal S}(X,R)$ such that $Z(B_1)=Z(H_1^2+\cdots+H_r^2)\setminus M_{\lc}$ and $Z(B_2)=\cl_X(Z(H_1^2+\cdots+H_r^2)\cap M_{\lc})$. Observe $Z(B_1B_2)=Z(H_1^2+\cdots+H_r^2)$ and as $\gtp_i$ is a prime $z$-ideal and $(H_1|_M)^2+\cdots+(H_r|_M)^2\in\gtp_i$, we conclude $B_1|_MB_2|_M\in\gtp_i$. Note that $B_1|_{M_{\lc}}$ is a unit in ${\mathcal S}(M_{\lc})$; hence, $B_1|_M\not\in\gtp_i$ and so $B_2|_M\in\gtp_i$. Thus, $B_2\in\gtp_i\cap{\mathcal S}(X,R)=\mu^{-1}(\beta_i)$. By Lemma \ref{st1} we know $\varphi\circ{\tt j}^*={\rm ev}_{X_F,{\tt j}_F({\tt p}_0)}\circ{\tt i}_{X,F}$ and we deduce $B_{2,F}({\tt p}_0)=\psi(B_2)=0$, as $B_2\in\mu^{-1}(\beta_i)\subset\gtp_0\cap{\mathcal S}(X,R)\subset\ker\varphi$. Therefore, ${\tt p}_0\in\cl(Z(H_1^2+\cdots+H_r^2)_F\cap M_{\lc,F})$ and by the curve selection lemma there exists a semialgebraic path $\alpha:[0,1]\to F^n$ such that 
\begin{equation}\label{mlc0}
\alpha((0,1])\subset Z(H_1^2+\cdots+H_r^2)_F\cap M_{\lc,F}=M_{\lc,F}\cap\bigcap_{j=1}^rZ(H_{j,F})
\end{equation}
and $\alpha(0)={\tt p}_0$. By Theorem \ref{boundedcase}, there exist $A_1,A_2\in{\mathcal S}(X,R)$ such that $A_2h-A_1\in\gtp_0$, $A_{2,F}({\tt p}_0)$ and $\varphi(h)=\tfrac{A_{1,F}({\tt p}_0)}{A_{2,F}({\tt p}_0)}$. Thus,
$$
\varphi(A_2|_Mh)=A_{2,F}({\tt p}_0)\varphi(h)=A_{1,F}({\tt p}_0).
$$
By \eqref{mlc0} the semialgebraic functions $H_{j,F}\circ\alpha$ are identically zero and the semialgebraic functions $g_{j,F}\circ\alpha$ are well-defined on the interval $(0,1]$ of $F$. Thus,
$$
A_{1,F}({\tt p}_0)=\lim_{t\to0^+}(A_{1,F}\circ\alpha)(t)=\lim_{t\to0^+}(A_{2,F}\circ\alpha)(t)\Big(\sum_{j=1}^r(H_{j,F}\circ\alpha)(g_{j,F}\circ\alpha)\Big)(t)=\lim_{t\to0^+}0=0;
$$
hence, $\varphi(h)=\frac{A_{1,F}({\tt p}_0)}{A_{2,F}({\tt p}_0)}=0$ and we conclude $\gtp_i\subset\gtp_0:=\ker\varphi$.

\begin{substeps}{2}
Finally, we prove: {\em ${\tt d}_M(\gtp_i)=\tr\deg_R(\qf({\mathcal S}(M,R)/\gtp_0))+i$ for $i=1,\ldots,d$}. 
\end{substeps}

Indeed, as $\gtP_{\beta_{i+1}}\subsetneq\gtP_{\beta_i}$ and each $\gtP_{\beta_i}:=\supp(\beta_i)$ is a real ideal, it holds 
\begin{equation}\label{>}
\dim(Z(\gtP_{\beta_i}))>\dim(Z(\gtP_{\beta_{i-1}}))
\end{equation}
for $i=1,\ldots,d$. Therefore 
$$
\dim(Z(\gtP_{\beta_d}))\!\overset{\eqref{>}}{\geq}\!d+\dim(Z(\gtP_{\beta_0}))=d+\dim(\beta_0)\overset{\eqref{d}}{=}\dim_{\beta_0}(\widetilde{M}_{\lc})\!\overset{\eqref{def}}{\geq}\!\dim(\beta_d)=\dim(Z(\gtP_{\beta_d})).
$$
We deduce $\dim(Z(\gtP_{\beta_d}))=d+\dim(Z(\gtP_{\beta_0}))$ and so
$$
\dim(Z(\gtP_{\beta_i}))=\dim(Z(\gtP_{\beta_0}))+i
$$
for $i=1,\ldots,d$. By \eqref{pol}
$$
{\tt d}_X(\mu^{-1}(\beta_i))=\dim(Z(\gtP_{\beta_i}))=\dim(Z(\gtP_{\beta_0}))+i
={\tt d}_X(\mu^{-1}(\beta_0))+i.
$$

As $\gtp_i$ is a $z$-ideal and $\gtp_i\cap{\mathcal S}(X,R)=\mu^{-1}(\beta_i)$, we obtain by (\ref{sdtd}.\ref{rmdc})
\begin{multline*}
{\tt d}_X(\mu^{-1}(\beta_i))=\tr\deg_R(\qf({\mathcal S}(X,R)/\mu^{-1}(\beta_i)))\\
\leq\tr\deg_R(\qf({\mathcal S}(M,R)/\gtp_i)={\tt d}_M(\gtp_i))
\leq{\tt d}_X(\mu^{-1}(\beta_i)),
\end{multline*}
that is, ${\tt d}_M(\gtp_i)={\tt d}_X(\mu^{-1}(\beta_i))$. Now since $X$ is a brimming semialgebraic pseudo-compact\-ifica\-tion of $M$ for $\gtp_0$, we conclude by (\ref{sdtd}.\ref{rmd})
\begin{multline*}
{\tt d}_M(\gtp_i)={\tt d}_X(\mu^{-1}(\beta_i))={\tt d}_X(\mu^{-1}(\beta_0))+i\\
=\tr\deg_R(\qf({\mathcal S}(X,R)/(\gtp_0\cap{\mathcal S}(X,R))))+i=\tr\deg_R(\qf({\mathcal S}(M,R)/\gtp_0))+i.
\end{multline*}

\noindent{\bf S1.d.} Pick $f\in\gtp_1$ such that $\dim(Z(f))={\tt d}_M(\gtp_1)$ and define $N_0:=Z(f)$. By \cite{dk} the map $\psi:{\mathcal S}(M,R)\to{\mathcal S}(N_0,R),\, f\mapsto f|_{N_0}$ is an epimorphism. Since $\gtp_1$ is a $z$-ideal and $f\in\gtp_1$, we obtain $\ker\psi\subset\gtp_1\subset\gtp_0$. Therefore $\gtq_i:=\gtp_i/\ker\psi$ is a prime ideal of ${\mathcal S}(N_0,R)={\mathcal S}(M,R)/\ker\psi$ and ${\mathcal S}(M,R)/\gtp_i\cong{\mathcal S}(N_0,R)/\gtq_i$ for $i=0,1$. Using the fact that $\gtp_1$ is a $z$-ideal, one proves directly that $\gtq_1:=\gtp_1/\ker\psi$ is a $z$-ideal of ${\mathcal S}(N_0,R)$. Thus,
$$
\tr\deg_R(\qf({\mathcal S}(N_0,R)/\gtq_0))+1=\tr\deg_R(\qf({\mathcal S}(M,R)/\gtp_0))+1={\tt d}_M(\gtp_1)=\dim(N_0).
$$
We claim: \em $(T_0:=\cl_X(N_0),{\tt j}|_{T_0})$ is a brimming pseudo-compactification of $N_0$ for $\gtq_0$\em.

Indeed, consider the commutative diagram
$$
\xymatrix{
{\mathcal S}(X,R)\ar@{>>}[rr]\ar@{^(->}[d]&&{\mathcal S}({T_0},R)\ar@{^(->}[d]\\
{\mathcal S}(M,R)\ar@{>>}[rr]&&{\mathcal S}(N_0,R)
}
$$
where the rows are epimorphisms and the columns are monomorphisms. Since $\gtq_0:=\gtp_0/\ker\psi$, we obtain
$$
{\mathcal S}({T_0},R)/(\gtq_0\cap{\mathcal S}({T_0},R))\cong{\mathcal S}(X,R)/(\gtp_0\cap{\mathcal S}(X,R)) 
$$
and conclude
\begin{multline*}
\qf({\mathcal S}({T_0},R)/(\gtq_0\cap{\mathcal S}({T_0},R)))\cong\qf({\mathcal S}(X,R)/(\gtp_0\cap{\mathcal S}(X,R)))\\
\cong\qf({\mathcal S}(M,R)/\gtp_0)\cong\qf({\mathcal S}(N_0,R)/\gtq_0). 
\end{multline*}

\noindent{\bf Step 2.} {\em Construction of the semialgebraic set $N$ in the statement}. By Lemma \ref{se0} there exists an appropriately embedded semialgebraic set $N\subset R^n$ and a surjective semialgebraic map $H:T:=\cl(N)\to {T_0}:=\cl(N_0)$ whose restriction $h:=H|_N:N\to N_0$ is a semialgebraic homeomorphism. Then it holds: \em $T$ is a brimming pseudo-compactification of $N$ for $h^*(\gtq_0)$, i.e,
\begin{equation}\label{brimt}
\qf({\mathcal S}(T,R)/h^*(\gtq_0)\cap{\mathcal S}(T,R))=\qf({\mathcal S}(N,R)/h^*(\gtq_0)).
\end{equation}
\em

Indeed, $H$ induces an injective homomorphism $H^*:{\mathcal S}(T_0,R)\hookrightarrow{\mathcal S}(T,R)$ that makes the following diagram commutative
$$
\xymatrix{
{\mathcal S}(T_0,R)\ar@{^(->}[rr]^{H^*}\ar@{^(->}[d]&&{\mathcal S}(T,R)\ar@{^(->}[d]\\
{\mathcal S}(N_0,R)\ar[rr]^{h^*}_{\cong}&&{\mathcal S}(N,R)
}
$$
Thus, we get the following inclusion of fields
\begin{multline*}
\qf({\mathcal S}(N_0,R)/\gtq_0)=\qf({\mathcal S}({T_0},R)/\gtq_0\cap{\mathcal S}(T_0,R))\subset\qf({\mathcal S}(T,R)/h^*(\gtq_0)\cap{\mathcal S}(T,R))\\
\subset\qf({\mathcal S}(N,R)/h^*(\gtq_0))\cong\qf({\mathcal S}(N_0,R)/\gtq_0)
\end{multline*}
and, comparing transcendence degrees, we conclude that equality \eqref{brimt} holds.

\noindent{\bf Step 3.} {\em Construction of the (core) point ${\tt p}\in F^n$ in the statement.} Consider the $R$-homomor\-phism $\ol{\varphi}:{\mathcal S}(N,R)\to F$ such that $\varphi=\ol{\varphi}\circ h^*\circ{\tt i}^*$ where ${\tt i}^*:{\mathcal S}(M,R)\to{\mathcal S}(N_0,R),\ g\mapsto g|_{N_0}$ is the homomorphism induced by the inclusion ${\tt i}:N_0\hookrightarrow M$. Let ${\tt p}:=(\ol{\varphi}(\pi_1),\ldots,\ol{\varphi}(\pi_n))$ where $\pi_i:R^n\to R$ is the projection onto the $i$th coordinate. By (\ref{sdtd}.\ref{rmd}) we have
\begin{multline*}
\dim(N_{F,{\tt p}})=\dim(T_{F,{\tt p}})\geq{\tt d}_T({\tt p})=\tr\deg_R\qf({\mathcal S}(T,R)/h^*(\gtq_0)\cap{\mathcal S}(T,R))\\
=\dim(N_0)-1=\dim(N)-1\geq\dim(N_{F,{\tt p}})-1.
\end{multline*}
If ${\tt p}\in N_F$ (which includes the case ${\tt d}_T({\tt p})=\dim( N_{F,{\tt p}})$) or if ${\tt d}_T({\tt p})=\dim(N_{F,{\tt p}})-1$ and ${\tt p}\not\in N_F$, we deduce by Lemma \ref{st1} that $\ol{\varphi}=\psi_{\tt p}:{\mathcal S}(N,R)\to F$ is the unique $R$-homomorphism from ${\mathcal S}(N,R)$ to $F$ whose core is ${\tt p}$.
\qed

\bibliographystyle{amsalpha}

\end{document}